\documentclass[dvips,aap,preprint]{imsart}

\RequirePackage[OT1]{fontenc}
\RequirePackage{amsthm,amsmath}
\RequirePackage[colorlinks]{hyperref}
\RequirePackage{hypernat}
\usepackage{amssymb}
\usepackage{paralist}
\usepackage{mathrsfs}

\startlocaldefs
\numberwithin{equation}{section}
\theoremstyle{plain}
\newtheorem{thm}{Theorem}[section]
\newtheorem{prop}{Proposition}[section]
\newtheorem{cor}{Corollary}[section]
\newtheorem{lem}{Lemma}[section]
\endlocaldefs

\newcommand{\aaa}{C_{1}}
\newcommand{\bbb}{C_{2}}
\newcommand{\ccc}{C_{3}}
\newcommand{\ddd}{C_{3}}
\newcommand{\eee}{C_{4}}
\newcommand{\fff}{C_{5}}
\newcommand{\hhh}{C_{6}}

\newcommand{\ellv}{L_{2}}
\newcommand{\ellw}{L_{3}}
\newcommand{\ellx}{L_{1}}
\newcommand{\elly}{L_{4}}

\newcommand{\epsa}{\epsilon_{1}}
\newcommand{\epsb}{\epsilon_{2}}
\newcommand{\epsc}{\epsilon_{3}}
\newcommand{\epsd}{\epsilon_{2}}
\newcommand{\epsf}{\epsilon_{2}}

\newcommand{\lsm}{\ell}
\newcommand{\Lbig}{\mathcal{L}}
\newcommand{\osm}{o}

\begin{document}

\begin{frontmatter}
\title{Stability of Join the Shortest Queue Networks}
\runtitle{Stability of Join the Shortest Queue Networks}

\begin{aug}
\author{\fnms{Maury} \snm{Bramson}\thanksref{t1}\ead[label=e1]
{bramson@math.umn.edu}}

\thankstext{t1}{Supported in part by NSF Grants DMS-0226245 and CCF-0729537}
\runauthor{Maury Bramson}

\affiliation{University of Minnesota, Twin Cities} 

\address{University of Minnesota\\
Twin Cities Campus\\
School of Mathematics\\
Institute of Technology\\
127 Vincent Hall\\
206 Church Street S.E.\\
Minneapolis, Minnesota 55455\\
\printead{e1}\\
\phantom{E-mail:\ }}

\end{aug}

\begin{abstract}
Join the shortest queue (JSQ) refers to networks whose incoming jobs are assigned to the 
shortest queue from among a randomly chosen subset of the queues in the system.  
After completion of service at the queue, a job leaves the network.  We show that, 
for all non-idling service disciplines and for general interarrival and service time 
distributions, such networks are stable when they are subcritical.  We then obtain
uniform bounds on the tails of the marginal distributions of the equilibria for families
of such networks; these bounds are employed to show relative compactness of the marginal
distributions.  We also present a family of subcritical JSQ networks whose workloads in
equilibrium are much larger than for the corresponding networks where each incoming job
is assigned randomly to a queue.  Part of this work generalizes results in Foss and Chernova  
\cite{FC}, which applied fluid limits to study networks with the FIFO discipline.  Here, 
we apply an appropriate Lyapunov function.
\end{abstract}



\end{frontmatter}

\section{Introduction}
Join the shortest queue (JSQ) refers to networks whose incoming ``jobs" (or ``customers") are assigned 
to the shortest queue from among a randomly chosen subset of queues in the system.  Here,
shortest queue means the queue with the fewest jobs.  (One could also consider the queue
with the least remaining work.)  Jobs are assumed to arrive in the network through one or
more random streams of jobs.  An arriving job is presented with a random subset B of the
queues, with probability depending only on the stream, and chooses the queue in B with the
fewest jobs; when two or more queues have the fewest jobs, one of these queues is chosen according to some rule. 
Jobs at queues are served according to some discipline, such as first-in, first-out
(FIFO), last-in, first-out (LIFO), or processor sharing (PS) and, upon completion of service, 
leave the network.

A family of such networks is given by the \emph{mean field} rule where $|B|=D$, $D\ge1$, is fixed, 
and B is chosen uniformly from among the $\binom{N}{D}$  such sets, where N is the total number of queues.
An alternative rule is given by choosing $D$ queues with replacement from among the $N$ queues, as in
Vvedenskaya, Dobrushin and Karpelevich \cite{VDK}.
A natural setting for both rules is where all jobs arrive at a single Poisson stream with rate $\alpha N$, with all 
jobs having the same service distribution $F(\cdot)$, with mean $m$, and are served according to the same
discipline at each queue.  Such networks are \emph{subcritical} when $\alpha m < 1$, that is, the long 
term rate at which jobs arrive in the network is strictly less than the total rate at which jobs are
served when no queues are empty.  When $\alpha m < 1$ and $F(\cdot)$ is exponentially distributed, it is 
elementary to show that the network is \emph{stable}, that is, the underlying Markov process is
positive Harris recurrent.  It will therefore have an equilibrium distribution.  In this setting, 
the choice of service discipline (when not relying on the residual service times) does not affect the 
distribution of the number of jobs at each queue, 
and the state space can be chosen so that it depends only on the number of jobs at each queue.  When $D=1$, we will say
that each incoming job is \emph{assigned randomly} to a queue; when, in addition, $F(\cdot)$ is exponentially
distributed, the network will consist of $N$ independent $M/M/1$ queues. 

The asymptotic behavior of the equilibria of such networks, with $D$ fixed as $N \rightarrow \infty$, has
been studied since the mid 1990s.  Let $E^{(N)}(\cdot)$ denote the equilibrium
distribution function at a single queue for the network of $N$ queues.  In \cite{VDK}, it was shown that, for $\rho = \alpha m < 1$,
\begin{equation}
\label{eq1.4.1}
\lim_{N\rightarrow \infty}\bar{E}^{(N)}(\ell) = \rho^{1 + D + \ldots + D^\ell} = \rho^{(D^{\ell + 1} -1)/(D - 1)}
\qquad \text{ for } \ell \in \mathbb{Z}_+, 
\end{equation}
where  $\bar{E}^{(N)}(\cdot) = 1 - E^{(N)}(\cdot)$, and $D > 1$ is required for the second
equality.  Hence, the tail of $\lim_{N\rightarrow\infty}E^{(N)}(\cdot)$ decreases doubly exponentially 
fast when $D > 1$; when $D=1$, the exponential tail is that of the corresponding $M/M/1$ queue.  This rapid 
decrease in the tail has different applications, such as in the design of complex networking systems where
memory is at a premium. See Azar, Broder, Karlin and Upfal \cite{ABKU}, Luczak and McDiarmid \cite{LM1,LM2},
Martin and Suhov \cite{MS}, Mitzenmacher \cite{M}, Suhov and Vvedenskaya \cite{SV}, Vocking \cite{V}, and
Vvedenskaya and Suhov \cite{VS} for related work on JSQ networks and ball-bin models in both theoretical and
applied contexts. 

The study of networks with given $N$ has been more restricted.  Foley and McDonald \cite{FM} studied 
the equilibria for small values of $N$.

Little work has been done on networks with nonexponential service distributions.  In this setting, the 
stability of subcritical networks is no longer obvious.  In particular, jobs might be assigned to 
short queues where the remaining work (or \emph{workload}) is high, which can cause service inactivity after queues with 
many jobs, but low remaining work, empty.  If the system can be ``tricked" too often in this
manner, it is conceivable that it is unstable while nevertheless being subcritical.  For general
service distributions, the evolution of the system will be influenced by the service discipline, which
complicates analysis.

For JSQ networks with general service times, Foss and Chernova \cite{FC} is the main work that analyzes stability.
Under the FIFO discipline, stability for a broad family of subcritical networks is demonstrated in 
\cite{FC}, including those with the JSQ rule, for general service distributions and for arrivals given 
by a single renewal stream.  Fluid limits are employed as the main tool.  In this more general framework,
the appropriate definition for subcritical is no longer transparent.  It will be discussed in the next
subsection.

For our results, we adopt the same basic framework as in \cite{FC} for JSQ networks, but instead
consider general service disciplines.  We also allow multiple arrival streams.  For general service disciplines,
the number of partially served jobs may be large and known fluid limit techniques cannot be applied.  Instead, we
employ an appropriate Lyapunov function.  

In this paper, we first show stability of
subcritical networks for all non-idling disciplines.  We then obtain uniform bounds on the tails of the
marginal distributions of the equilibria for families of such networks; these bounds are employed to show
relative compactness of the marginal distributions.  
Both the uniform bounds and relative compactness will be important tools for investigating, 
in the mean field setting, the limiting
behavior, as $N \rightarrow \infty$, of the equilibria distributions $E^{(N)}(\cdot)$  at single queues 
for service disciplines such as FIFO, LIFO and PS (see Bramson, Lu and Prabhakar \cite{BLP1, BLP2}).
%
%
Appropriate analogs of (\ref{eq1.4.1}), for large values of
$\ell$, will hold under certain restrictions.  We lastly present a family of subcritical mean field 
networks, with $N=D=2$, where the service discipline is chosen so that the 
corresponding equilibria have much larger workload than do the corresponding $M/G/1$ queues.
This shows that the JSQ rule does not always provide efficient service of jobs. 

\subsection*{Main results}
We present our main results here, Theorems \ref{thm1.13.1}, \ref{thm1.14.1} and \ref{thm1.17.2}, and discuss
their ramifications. In the next subsection, we will give a general outline of the paper. 

In order to avoid technical details, we postpone until Section 2 details
regarding the construction of the state space $S$ and of the Markov process $X(t)$, $t \ge 0$, 
underlying a JSQ network.  We require at this point only limited specifics about the construction,
namely that a state $x \in S$ is specified by descriptors that include the number of jobs $z_{n}$ at
each queue $n$, $n = 1,\ldots, N$; the residual interarrival times $u_{k}$, $k = 1,\ldots, K$, at each
of the arrival streams, which are given by independent renewal processes; the residual service times   
$v_{n,i}$, $n = 1, \ldots, N$ and $i = 1, \ldots, z_{n}$, for each of the jobs currently in the 
network; and the ages $\osm_{n,i}$ for each of these jobs.  (In the construction of $S$ in Section 2, normalized
versions of $u_k$, $v_{n,i}$ and $\osm_{n,i}$ are used.)  When the arrival 
streams are Poisson, or when the service time distributions are the same 
and exponentially distributed, the corresponding descriptor may be dropped.  The underlying Markov
process $X(\cdot)$ takes values in $S$ and is strong Markov.

In order to state our results, we need to specify the notion of subcriticality that was mentioned
in the context of \cite{FC}.  This requires the introduction of various terminology.  We denote by $G_{k}(\cdot)$, 
$k=1,\ldots, K$, the distribution function for the interarrival time of jobs at the $k^{\text{th}}$ 
renewal stream, and by $\alpha_k$ the reciprocal of its mean, with $\alpha_k > 0$ being assumed.
We denote by $p_{k,B}$ the probability that a job from arrival stream $k$ chooses the shortest 
queue from the set $B$, $B \subseteq B_{N}$, where $B_{N} = \{1,\ldots,N\}$.
We refer to $B$ as the \emph{selection set} and the rule corresponding to a given choice of $p_{k,B}$,
$k=1,\ldots,K$, $B\subseteq B_N$, as the \emph{selection rule}. For $n\in B$, such a job is 
a \emph{potential arrival} at $n$.

We denote by $F_{j}(\cdot)$, with $j = (k,B,n)$ for $k = 1,\ldots, K$, $B \subseteq B_{N}$ and 
$n \in B$, the distribution function for the service time (i.e., service requirement) of jobs from the $k^{\text{th}}$ 
renewal stream and selection set $B$ that are served at queue $n$; by $m_{j}$, the mean of $F_j(\cdot)$; and, by 
$\mu_{j} = 1/m_{j}$, the corresponding service rate.  As in \cite{FC}, we will require that either
(a) $F_{j}(\cdot)$ depend only on $n$ or (b) $F_{j}(\cdot)$ depend only on $k$ and $B$, in which case we
may write either $F_{n}(\cdot)$ or $F_{k,B}(\cdot)$ (and $m_{n}$ or $m_{k,B}$, respectively, $\mu_{n}$ or
$\mu_{k,B}$) when the context is clear.  (When neither (a) nor (b) holds, analysis is more complicated 
and, as explained in Section 6 of \cite{FC}, stability likely does not follow from subcriticality.)  
In \cite{BLP1}, \cite{BLP2} and \cite{VDK}, $F_{j}(\cdot) = F(\cdot)$ does not 
depend on $k$, $B$ or $N$ and there is a single renewal stream; in this setting, one can employ 
the notation $\alpha$, $m$ and $\mu$.  As in 
\cite{FC}, we refer to networks satisfying (a) as \emph{class independent} and (b) as 
\emph{station independent}.   

For the class independent case, we define the \emph{traffic intensity} 
\begin{equation}
\label{eq1.10.1}
\rho_{1} = \max_{B\subseteq B_{N}}  \left\{\left(\sum_{n\in B} \mu_n \right)^{-1}   
\sum_{k} \sum_{A\subseteq B} \alpha_k p_{k,A}\right\}
\end{equation}
and, for the station independent case,
\begin{equation}
\label{eq1.10.2}
\rho_{2} = \max_{B\subseteq B_{N}}  \left\{ |B|^{-1}   
\sum_{k} \sum_{A\subseteq B} \alpha_k p_{k,A} m_{k,A}\right\}.
\end{equation}
When $\rho_1 < 1$, respectively, $\rho_2 < 1$, we say the network is \emph{subcritical}.  
As was observed in \cite{FC}, it is not difficult to check that when $\rho_i > 1$ in either (\ref{eq1.10.1}) or
(\ref{eq1.10.2}), the corresponding network will be unstable.  This behavior does not depend on the service discipline.

When the network is both class
and station independent, (\ref{eq1.10.1}) and (\ref{eq1.10.2}) reduce to
\begin{equation}
\label{eq1.10.3}
\rho \stackrel{\text {def}}{=} \rho_1 = \rho_2 = \max_{B\subseteq B_{N}}  \left\{ \frac{m}{|B|}   
\sum_{k} \sum_{A\subseteq B} \alpha_k p_{k,A} \right\}.
\end{equation}
Let $H$ be a subgroup of the permutation group on $B_N$ on which all queues communicate (i.e., for given 
$n_1,n_2\in B_N$, $\pi(n_1) = n_2$ for some $\pi\in H$), and assume the symmetry condition
\begin{equation}
\label{eq1.11.1}
\sum_{k} \sum_{A\subseteq B_{\pi}} \alpha_k p_{k,A} = \sum_{k} \sum_{A\subseteq B} \alpha_k p_{k,A}
\qquad \text{for all } B\subseteq B_N \text{ and } \pi \in H
\end{equation}
is satisfied, where $B_\pi = \{n: n=\pi(n^{\prime}) \text{ for some } n^{\prime}\in B\}$.  (This holds, in particular, in
the mean field setting.)  Also, assume the network is both class and station independent.  Then it is not 
difficult to check that (\ref{eq1.10.3}) reduces to 
\begin{equation}
\label{eq1.11.2}
\rho = \left(\sum_{k}\alpha_{k}\right) m/N.
\end{equation}
(Note that, for given $A$, $A\subseteq B_{\pi}$ for at most $|H||B|/N$ permutations $\pi\in H$, 
with equality holding when $A$ is a singleton.)  

In addition to subcriticality, we will require the following condition on JSQ networks for Theorems \ref{thm1.13.1},
\ref{thm1.14.1} and \ref{thm1.17.2}:  for some $\Gamma \in \mathbb{Z}_{+,0}$ and $h(\cdot)$, with $h(t) > 0$ for all
$t$,
\begin{equation}
\label{eq1.11.2a}
P_{x} (\text{at most $\Gamma$ potential arrivals at $n$ occur over $(0,m^{\text{max}}t]$}) \ge h(t)
\end{equation}
for all $x$, $t$ and $n$, where $m^{\text{max}} \stackrel{\text{def}}{=} \max_{j}m_j$.  
Note that this condition depends only on the distributions $G_{k}(\cdot)$ and probabilities $p_{k,A}$.  It 
is met in most cases, for instance, if (a) for each $k$ and $n$, 
$\sum_{A \ni n} p_{k,A} < 1$, in which case one can set $\Gamma = 0$, or (b) for each $k$ and $y$, 
$G_{k}(y) < 1$, where $\Gamma=K+1$ suffices.  Condition (a) always holds in the mean field setting when $D<N$; Condition (b) 
is equivalent to (\ref{eq1.14.2}), which is required for Corollary \ref{cor1.15.2}.   

In Theorems \ref{thm1.13.1}, \ref{thm1.14.1} and \ref{thm1.17.2}, we will employ the nonnegative function $\|x\|$, or 
\emph{norm}, for $x\in S$.  It is defined in terms of the norms $\|x\|_L$, $\|x\|_R$ and $\|x\|_A$ by
\begin{equation}
\label{eq1.11.3}
\|x\| = \|x\|_L + \|x\|_R + \|x\|_A.
\end{equation}
We define these components of $\|x\|$ in Section 4.  Without going into details here, we note that $\|x\|_L$ depends on the 
number of jobs and a truncation of the residual service time of each job; $\|x\|_R$ depends on just
the residual service time of each job and will be employed for residual service times greater than the previous
truncation; and $\|x\|_A$ measures the residual interarrival times with appropriate weighting.  For given
$M > 0$, we denote by $\tau_{M}(1)$ the stopping time 
\begin{equation}
\label{eq1.12.1}
\tau_{M}(1) = \inf \{t \ge 1: \|X(t)\| \le M\}.
\end{equation}

We now state Theorem \ref{thm1.13.1}.  Here and elsewhere in the paper, we implicitly assume the discipline is
non-idling.  In Section 2, we will also specify mild conditions on how the service effort devoted to individual 
jobs is allowed to change over time.
\begin{thm}
\label{thm1.13.1}
For each subcritical JSQ network satisfying (\ref{eq1.11.2a}), there exist $M$ and $\aaa$ so that
\begin{equation}
\label{eq1.13.2}
E_{x}[\tau_{M}(1)] \le C_{1}(\|x\| \vee 1)  \qquad
\text{for all } x,
\end{equation}
where $\|x\|$ is the norm given in  (\ref{eq1.11.3}).
\end{thm}
For certain service disciplines, such as FIFO and PS, the condition (\ref{eq1.11.2a}) can be avoided; this is noted
after the proof of Proposition \ref{prop4.22.4}.  As mentioned above, (\ref{eq1.11.2a}) will 
in most cases be satisfied
irrespective of the service discipline.

The condition (\ref{eq1.13.2}) will imply the positive Harris recurrence of $X(\cdot)$ provided that the states in
the state space $S$ communicate with one another in an appropriate sense.  Petite sets are typically employed for 
this purpose; they will be defined in Section 2.  A petite set $A$ has the property that each measurable set $B$
is ``equally accessible" from all points in $A$ with respect to a given nontrivial measure.
\begin{thm}
\label{thm1.14.1}
Suppose that a JSQ network is subcritical, satisfies (\ref{eq1.11.2a}), and that $A_{M}= \{x: \|x\| \leq M\}$ is
petite for each $M > 0$ for the norm in (\ref{eq1.11.3}).  Then $X(\cdot)$ is positive Harris recurrent.
\end{thm}
Theorem \ref{thm1.14.1} will follow from Theorem \ref{thm1.13.1} by standard reasoning.  More detail is provided in
Section 2.

A standard criterion that ensures the above sets $A_M$ are petite is given by the following two conditions on the
interarrival times.  In various works on stability (e.g., \cite{B0}, \cite{DAI} and \cite{FC}), these conditions are employed
rather than the more abstract notion of petite set.  The first condition is that the distribution $G_{k}(\cdot)$
is unbounded for each $k$, that is,
\begin{equation}
\label{eq1.14.2}
\bar{G}_{k}(y) \stackrel{\text{def}}{=} 1 - G_{k}(y) > 0 \qquad \text{for all } y.
\end{equation}
The second condition is that, for some $\ell_{k} \in \mathbb{Z}_+$, the $\ell_k$-fold convolution 
$G_{k}^{*\ell_k} (\cdot)$ of 
$G_{k}(\cdot)$ and Lebesque measure are not mutually singular.  That is, for some nonnegative $q_{k}(\cdot)$ with
$\int_{0}^{\infty} q_{k}(s)ds > 0$,
\begin{equation}
\label{eq1.15.1}
G_{k}^{*\ell_k} (d) - G_{k}^{*\ell_k} (c) \ge \int_{c}^{d} q_{k}(s)ds
\end{equation}
for all $c < d$.  When the interarrival times are exponentially distributed, both (\ref{eq1.14.2}) and 
(\ref{eq1.15.1}) are immediate.  More detail is given in Section 2.

We therefore have the following corollary of Theorem \ref{thm1.14.1}.  As noted earlier, (\ref{eq1.11.2a}) is automatic
in this setting.  
\begin{cor}
\label{cor1.15.2}
Suppose that a subcritical JSQ network has interarrival times that satisfy (\ref{eq1.14.2})
and (\ref{eq1.15.1}).  Then $X(\cdot)$ is positive Harris recurrent.
\end{cor}
By employing Theorem \ref{thm1.14.1} and the bounds obtained in the derivation of Theorem \ref{thm1.13.1}, one
can obtain uniform bounds on the equilibria restricted to individual queues and to individual arrival streams 
for families of JSQ networks.  
Such a family $\mathcal{A}$ will be required to satisfy the following uniformity conditions on the service and interarrival
distributions $F_{j}^{(a)}(\cdot)$ and $G_{k}^{(a)}(\cdot)$,  for $a\in \mathcal{A}$:  
\begin{equation}
\label{eq1.16.2new}
\sup_{a\in \mathcal A} \, \max_j \, \mu_{j}^{(a)}  \int_{M/\mu_{j}^{(a)}}^{\infty} yF_{j}^{(a)}(dy) \rightarrow 0
\qquad \text{as } M \rightarrow \infty
\end{equation}
and
\begin{equation}
\label{eq1.16.1new}
\sup_{a\in \mathcal A} \, \max_k \, \alpha_{k}^{(a)} \int_{M/\alpha_{k}^{(a)}}^{\infty} yG_{k}^{(a)}(dy) \rightarrow 0 
\qquad \text{as } M \rightarrow \infty.
\end{equation}
%
%
%
We require 
\begin{equation}
\label{eq1.16.3}
\sup_{a\in \mathcal A} \rho_{i}^{(a)} < 1, 
\end{equation}
where $\rho_{i}^{(a)} = \rho_{1}^{(a)}$ or $\rho_{i}^{(a)} = \rho_{2}^{(a)}$, depending on whether
the network is class independent or station independent.  Setting
\begin{equation*}
\left(\stackrel{\circ}{m}\!\!\mbox{}^{\text{ratio}}\right)^{(a)} = \begin{cases}
1 \quad  & \text{for class independent networks}, \\
\frac{\max_{k,B} m_{k,B}^{(a)}}{\min_{k,B} m_{k,B}^{(a)}}  \quad & \text{for station independent networks},  
\end{cases}
\end{equation*}
we also require that
\begin{equation}
\label{eq1.16.3b}
\stackrel{\circ}{m}\!\!\mbox{}^{\text{ratio}} \stackrel{\text{def}}{=} 
\sup_{a\in \mathcal{A}} \left(\stackrel{\circ}{m}\!\!\mbox{}^{\text{ratio}}\right)^{(a)} < \infty.
\end{equation}

We also need a uniform version of (\ref{eq1.11.2a}), namely that, for some 
$\Gamma \in \mathbb{Z}_{+,0}$ and $h(\cdot)$, with $h(t) > 0$ for all $t$,
\begin{equation}
\label{eqN7.1}
P_{x}^{(a)} (\text{at most $\Gamma$ potential arrivals at $n$ occur over $(0,(m^{\text{max}})^{(a)}t]$}) \ge h(t)
\end{equation}
for all $a\in \mathcal{A}$, $x$, $t$ and $n$.  Here, $P_{x}^{(a)}(\cdot)$ denotes the transition kernel with 
respect to $a$.  When the selection rules $p_{k,B}^{(a)}$ satisfy $p_{k,B}^{(a)} = 0$ for $|B| \ge M$ for some
$M$ not depending on $a$, we will say the selection rules have \emph{uniformly bounded support}.  

In many cases of interest, conditions (\ref{eq1.16.2new})--(\ref{eqN7.1}) are not difficult to check. 
For instance, when each network has a single Poisson arrival stream and $F_{j}^{(a)}(\cdot)$ does not depend
on $j$ or $a$, all conditions except for (\ref{eq1.16.3}) and (\ref{eqN7.1}) are automatic.  When, in addition,
the selection rules have uniformly bounded support, then (\ref{eqN7.1}) also holds.  

By Theorem \ref{thm1.14.1}, for any subcritical JSQ network $a\in \mathcal{A}$, with $A_M$ petite for $M>0$, 
the underlying Markov process $X^{(a)}(\cdot)$ is positive Harris recurrent; this follows by employing
the analog of (\ref{eq1.11.3}),
\begin{equation}
\label{eq1.16.4}
\|x\|^{(a)} = \|x\|_L^{(a)} + \|x\|_R^{(a)} + \|x\|_A^{(a)}.
\end{equation}
In the statement of Theorem \ref{thm1.17.2}, we also employ the \emph{local norm} at $n=1,\ldots,N^{(a)}$,
\begin{equation}
\label{eq1.16.5}
|x|_{n}^{(a)} = z_n + \lsm_{n}^{(a)} + w_{n}^{(a)} \qquad \text{for } x\in S^{(a)}.
\end{equation}
Here, $z_n$ is the number of jobs at queue $n$ and $w_{n}^{(a)}$ is the \emph{weighted workload} at $n$, that is,
\begin{equation}
\label{eq1.16.6}
w_{n}^{(a)} = \sum_{i=1}^{z_n} w_{n,i}^{(a)} = \sum_{i=1}^{z_n} \mu_{j_{n,i}}^{(a)} v_{n,i},
\end{equation}
where $\mu_{j_{n,i}}^{(a)}$ denotes the service rate of the $i^{\text{th}}$ 
job at queue $n$ and $v_{n,i}$ is the residual
service time of this job.  The term $\lsm_{n}^{(a)}$ is the \emph{maximum weighted age} at $n$, that is,
\begin{equation}
\label{eqN7.2}
\lsm_{n}^{(a)} = \max_{i=1,\ldots,z{_n}} \lsm_{n,i}^{(a)} = \max_{i=1,\ldots,z{_n}} \mu_{j_{n,i}}^{(a)} \osm_{n,i},
\end{equation}
where $\osm_{n,i}$ is the age of the $i^{\text{th}}$
job at queue $n$ (the time since its arrival at $n$). We set 
\begin{equation}
\label{eq1.16.7}
s_{k}^{(a)} = \alpha_{k}^{(a)} u_k,
\end{equation}
where $u_k$ is the residual interarrival time at the arrival stream $k$; we refer to $s_{k}^{(a)}$ as 
the \emph{weighted interarrival time}.  (Since $x$ and $z$ do not explicitly depend on $a\in \mathcal{A}$,
the corresponding superscripts are omitted.)  We will employ, in Section 4, some of the terminology and conditions introduced in
(\ref{eq1.16.2new})--(\ref{eq1.16.7}), when demonstrating Theorem \ref{thm1.13.1} for a given
JSQ network.

Since $X^{(a)}(\cdot)$ is positive Harris recurrent, it will have a unique equilibrium measure on $S^{(a)}$, which we
denote by $\mathcal{E}^{(a)}$.  In order to ensure that the marginal distribution at a given queue $n$ does
not depend on $n$, we also assume that, for some subgroup $H^{(a)}$ of the permutation group on $B_{N^{(a)}+K^{(a)}}$, with
queues being mapped to queues, arrival streams to arrival streams and on which all queues (but not necessarily
all arrival streams) communicate,
\begin{equation}
\label{eq1.17.1}
X_{\pi}^{(a)}(\cdot) \text{ and } X^{(a)}(\cdot) 
\text{ are stochastically equivalent for all } \pi \in H^{(a)}. 
\end{equation}
(That is, $X_{\pi}^{(a)}(\cdot)$  and  $X^{(a)}(\cdot)$ have the same joint distributions.) 
Here, $X_{\pi}^{(a)}(\cdot)$ is the stochastic process induced from $X^{(a)}(\cdot)$ by permuting the queues and 
arrival streams
according to $\pi$.  We will call such a JSQ network \emph{symmetric}.  As an example of such a JSQ network, one
can consider $N$ queues arranged uniformly along a circle, along with $N$ arrival streams that alternate with the queues, with
jobs arriving at a given stream by selecting the shortest queue within a preassigned distance of the stream, and
the service discipline and service and interarrival distributions not depending on the queue, arrival stream or
selection set.

Employing the preceding conditions, we now state our third main result.  Here and later on, $\mathcal{E}^{(a)}(\cdot)$ denotes
the probability of the indicated event with respect to the measure $\mathcal{E}^{(a)}$, and $X$ and $S^{(a)}$
denote the random variables corresponding to $x$ and $s^{(a)}$.  (The random variables 
$S^{(a)} = (S_{1}^{(a)},\ldots,S_{N}^{(a)})$ should not be confused with the state space.)
\begin{thm}
\label{thm1.17.2}
Suppose that a family $\mathcal{A}$ of JSQ networks satisfies the uniformity conditions  
(\ref{eq1.16.2new})--(\ref{eqN7.1}) and (\ref{eq1.17.1}), and that, for each network $a\in{\mathcal{A}}$,
$A_{M} = \{x: \|x\|^{(a)} \le M\}$ is petite for each $M>0$ with respect to the norms in (\ref{eq1.16.4}).  Then
\begin{equation}
\label{eq1.18.1}
\sup_{a\in \mathcal{A}} \mathcal{E}^{(a)} (|X|_{n}^{(a)} > M) \rightarrow{0} \qquad \text{as } 
M \rightarrow{\infty},
\end{equation}
for each n, and
\begin{equation}
\label{eq1.18.2}
\sup_{a\in \mathcal{A}} \, \max_{k} \mathcal{E}^{(a)} (S_{k}^{(a)} > M) \rightarrow{0} \qquad \text{as } 
M \rightarrow{\infty}.
\end{equation}
\end{thm}
The limits (\ref{eq1.18.1}) and (\ref{eq1.18.2}) supply uniform bounds on the number
of jobs, maximum weighted age and weighted workload at each individual queue $n$, 
and on the weighted interarrival times at each arrival
stream $k$.  When these limits hold, we say the equilibria $\mathcal{E}^{(a)}$, for $a\in \mathcal{A}$,
are \emph{locally bounded}.  Note that, on account of (\ref{eq1.17.1}), the probabilities in (\ref{eq1.18.1})
do not depend on $n$.

In many cases of interest, the conditions in Theorem \ref{thm1.17.2} are not difficult to check.  As mentioned in 
conjunction with (\ref{eq1.16.2new})--(\ref{eqN7.1}), when each member of $\mathcal{A}$
has a single Poisson arrival stream, $F_{j}^{(a)}(\cdot)$ does not depend on $j$ or $a$, and the selection rules
have uniformly bounded support, then all of these 
properties except (\ref{eq1.16.3}) hold.  
Since (\ref{eq1.14.2}) and (\ref{eq1.15.1}) hold for Poisson arrivals, the sets $A_M$ are petite.
Also, under (\ref{eq1.17.1}), the traffic intensity can be
written as in (\ref{eq1.11.2}).  We therefore obtain the following corollary of Theorem \ref{thm1.17.2}.  Recall that
a selection rule is mean field if, for a given $D$, each nonrepeating $D$-tuple is chosen with equal 
probability from among the $N^{(a)}$ queues, and note that JSQ networks with a mean field selection rule and for 
which $F_{j}^{(a)}$ does not depend on $j$ or $a$ are also symmetric. 

\begin{cor}
\label{corN8.1}
Suppose that each member of a family $\mathcal{A}$ of JSQ networks has a single Poisson arrival stream, that
$F_{j}^{(a)}$ does not depend on $j$ or $a$, that the selection rules have uniformly bounded support, and that 
\begin{equation}
\label{eqN8.2}
\sup_{a\in \mathcal{A}} \alpha^{(a)}m/N^{(a)} < 1.
\end{equation}
If (\ref{eq1.17.1}) is satisfied, then (\ref{eq1.18.1}) and (\ref{eq1.18.2}) hold.  In particular, if the selection
rules are mean field, then (\ref{eq1.18.1}) and (\ref{eq1.18.2}) hold.
\end{cor}

%
As mentioned earlier, when a JSQ network has a single Poisson arrival stream, one can omit the interarrival times from
the state space descriptor.  In this case, the limit (\ref{eq1.18.2}) is no longer relevant in Theorem \ref{thm1.17.2}
and Corollary \ref{corN8.1}. 

When $\mathcal{A}$ is given by a family of networks indexed by the number $N$ of queues, Theorem \ref{thm1.17.2} 
provides local bounds on $\mathcal{E}^{(N)}$ as $N\rightarrow{\infty}$.  These bounds can be used to show the relative
compactness of the restriction of $\mathcal{E}^{(N)}$ to finite sets of queues; this is done in Theorem \ref{thm5.22.1}.
%
%
As mentioned earlier, these local bounds and relative compactness of the sequence provide a framework for approximating  
the corresponding marginal distributions for large $N$ (\cite{BLP1, BLP2}).  In this context, 
one employs appropriate mean field equations corresponding to the marginal distributions of the 
equilibrium $\mathcal{E}^{(\infty)}$ of a limiting infinite system.  Under appropriate assumptions on the service times, 
the solutions of these mean field equations satisfy bounds on
the number of jobs that are on the same order as in the limit in (\ref{eq1.4.1}). 
%

In Section 7, we present a family of mean field, symmetric networks, with a single Poisson arrival stream, $N=D=2$, 
and an appropriate service discipline that illustrates how the JSQ rule can produce equilibria for which the
typical workload is incredibly large, much larger than the workload for the analogous network with $D=1$.
So, in terms of workload, the JSQ rule can sometimes yield poor results. 

\subsection*{Outline of the paper and main ideas}
In Section 2, we will provide a brief background on Markov processes that will be relevant to the space $S$ and 
Markov process $X(\cdot)$ employed in the introduction, and we will provide a more detailed construction of $S$ and
$X(\cdot)$.  We will also provide the background that is needed to derive Theorem \ref{thm1.14.1} from Theorem
\ref{thm1.13.1} and to obtain Corollary \ref{cor1.15.2}.  The machinery for this is standard in the context of
queueing networks and is easily modified so as to apply to JSQ networks.

In Section 3, we provide an alternative formulation of the traffic intensities in (\ref{eq1.10.1}) and
(\ref{eq1.10.2}) that we employ in successive sections.  This formulation will
enable us to compare JSQ networks to networks with appropriate random assignment of jobs to queues.  

Theorem \ref{thm1.13.1} is demonstrated in Sections 4 and 5.  The main tool is an appropriate Lyapunov function that is
given in terms of the norms $\|x\|_L$, $\|x\|_R$ and $\|x\|_A$ in (\ref{eq1.11.3}).  Our analysis involves 
decomposing time into random intervals over which no jobs enter the network.  Over each such interval, the evolution
of the system is deterministic and $\|X(t)\|$ is shown to be decreasing at large values.  At times $t$ where
a job enters the network, the average value of $\|X(t)\|-\|X(t-)\|$ is shown to be negative.  Applying the strong 
Markov property and iterating over such intervals until time $\tau_M$, $\tau_M = \inf \{t: \|X(t)\| \le M\}$, 
will imply (\ref{eq1.13.2}) of Theorem
\ref{thm1.13.1}.  At the end of Section 5, we briefly discuss the networks
mentioned at the beginning of the introduction where jobs are assigned to the queue with the least
remaining work.  We refer to such networks as \emph{join the least loaded queue} (JLLQ) networks.  The JLLQ rule 
is easier to handle than the JSQ rule.  It is analyzed in \cite{FC} using fluid limits;
here, we mention an alternative approach.  At the end of the section, we also mention analogs of 
Theorem \ref{thm1.13.1} where the state space $S$ is modified.

Theorem \ref{thm1.17.2} is demonstrated in Section 6.  The basic idea there is to employ estimates from Sections 4 and 5
to obtain lower bounds on the rate at which $\|X^{(a)}(t)\|^{(a)}$ decreases at large values 
of $|X^{(a)}(t)|_{n}^{(a)}$.  If these values
occur with too high a probability at a given $a\in \mathcal{A}$, it will follow that, for appropriate $t$,
\[E_{\mathcal{E}^{(a)}} [\|X^{(a)}(t)\|^{(a)}] < E_{\mathcal{E}^{(a)}} [\|X^{(a)}(0)\|^{(a)}].\]
Since $\mathcal{E}^{(a)}$ is an equilibrium, this is not possible, which gives us upper bounds on the probabilities in
(\ref{eq1.18.1}) and (\ref{eq1.18.2}).  Theorem \ref{thm5.22.1} follows quickly from Theorem \ref{thm1.17.2}  
and Prohorov's Theorem.  

In Section 7, we discuss the family of networks mentioned earlier, with $N=D=2$, whose equilibria have very large
workload.  This behavior arises because of the manner in which the service discipline restricts service 
for jobs with large residual service time.  The main result is given in Theorem \ref{prop7.3.1}.  
Because of the length of the argument, we provide a condensed proof of part of the theorem.
%
\subsection*{Notation}
For the reader's convenience, we mention here some of the notation in the paper.  The term $x$ indicates a state in 
the state space $S$ and
the corresponding term $X(t)$ (or $X$) indicates a random state at time $t$ (or with respect to a given measure); 
quantities such as $z_n$ and $Z_{n}(t)$, and
$w_{n,i}$ and $W_{n,i}(t)$ (or, $Z_n$ and $W_{n,i}$) play analogous roles.  
The results in this paper are for class independent and for station
independent networks; in order to avoid repeating all statements and proofs for the two cases, we employ notation
with the symbol $\circ$, such as $\stackrel{\circ}{m}_{k,A}$, which will have different meanings in the 
two cases. 
The main norms with which we will work are $\|\cdot\|$, $\|\cdot\|_L$, $\|\cdot\|_R$ and $\|\cdot\|_A$.  The
symbols $L$, $R$ and $A$ also appear in other contexts, such as for the numbers $L_i$, $i=1,\ldots,4$, service effort
per job $R_{n,i}$ and sets $A$; the meaning should always be clear from the context.  
The symbols $\mathbb{Z}_+$ and $\mathbb{R}_+$ denote the positive integers and positive real numbers, with
$\mathbb{Z}_{+,0} = \mathbb{Z}_+ \cup \{0\}$ and $\mathbb{R}_{+,0} = \mathbb{R}_+ \cup \{0\}$; $\lfloor y\rfloor$
denotes the integer part of $y \in \mathbb{R}_+$ and $\mathbf{1}\{A\}$ denotes the indicator function of the event $A$.  
\section{Markov process background}

In this section, we provide a more detailed description of the construction of
the Markov process $X(\cdot)$ that underlies a JSQ network. We then show how
Theorem \ref{thm1.14.1} and its corollary follow from Theorem \ref{thm1.13.1}.
Analogs of this material for queueing networks are given in Bramson \cite{B1},
and, for networks with weighted max-min fair policies, in Bramson \cite{B2}.
Because of the similarity of these settings, we present a summary here and
refer the reader to \cite{B1} for more detail.

\subsection*{Construction of the Markov process}
We define the state space $S$ to be the set
\begin{equation}
\label{eq2.1.1}
\left(\mathbb{Z}^3 \times 2^{B_N} \times \mathbb{R}^3\right)^\infty \times \mathbb{R}^K
\end{equation}
subject to the following constraints.  The components $s_k$, $k=1, \ldots,K$, of $\mathbb{R}^K$ are all
positive; they correspond to the residual interarrival times $u_k$ of the $K$ arrival streams, 
scaled by the arrival rates $\alpha_k$ as in (\ref{eq1.16.7}).  Only a finite
number of the 7-tuples of coordinates of $\left(\mathbb{Z}^3 \times 2^{B_N} \times \mathbb{R}^3\right)^\infty $ 
are nonzero.  Such a 7-tuple corresponds to a 
particular job in the network:  the first coordinate $n$, $n=1, \ldots, N$, corresponds to the queue
of the job and the second coordinate $i$, $i=1, \ldots, z_n$, gives its rank at the queue based on the time of 
arrival there, with ``older" jobs receiving a lower rank.  The third and fourth coordinates $k$, $k=1,\ldots,K$, and
$A$, $\emptyset \neq A\subseteq B_N$, correspond to the arrival stream and selection set of the job when it entered its queue. 
The fifth coordinate $\lsm$, $\lsm \ge 0$, measures the age $\osm$ 
of the job, scaled by $\mu_{j_{n,i}}$, as in (\ref{eqN7.2}); the sixth coordinate $w$, $w > 0$, measures
its residual service time $v$, scaled by $\mu_{j_{n,i}}$, as in (\ref{eq1.16.6}); 
and the last coordinate $r$, $r \in [0,1]$, is the current service effort devoted to
the job.  (If one wishes, one can introduce other coordinates in the state space descriptor, such as the elapsed
time since the last arrival from each stream, or the amount of service already received by each job.) 
Since the discipline is assumed to be non-idling, the sum of the last coordinates for all jobs at a
given nonempty queue must equal 1.  Note that for given $N$ and $K$, the state
space $S$ constructed in this manner is unique.

The last five coordinates may be considered as functions of the first two, and 
written as $k_{n,i}$, $A_{n,i}$, $\lsm_{n,i}$, $w_{n,i}$ and $r_{n,i}$.  The third and fourth coordinates,
$k_{n,i}$ and $A_{n,i}$, are needed because of how $\|\cdot\|$ is defined in Section 4 and can be omitted for
class independent networks.  Various coordinates can also be omitted for particular service disciplines 
(such as FIFO).  

For given $N^{\prime} \le N$ and $K^{\prime} = K$, one can define $S^{\prime}$ as above, but with $n=1,\ldots,N^{\prime}$.
Then $S^{\prime}$ is the
\emph{projection} of $S$ obtained by restricting nonzero 7-tuples to the first $N^{\prime}$ queues. 
For $x\in S$, the 
projection $x^{\prime}\in S^{\prime}$ of $x$ is the element obtained by omitting 7-tuples with $n > N^{\prime}$.
One can define projections of $S$ onto spaces $S^{\prime}$
corresponding to other subsets of $\{1,\ldots,N\}$ analogously, but we will not use these in 
the paper.

Employing the above notation, we construct a metric $d(\cdot,\cdot)$ on $S$:  for given $x,x'\in S$, with the coordinates
labelled correspondingly, we set
\begin{equation}
\label{eq2.3.1} 
\begin{split}
d(x,x^\prime)=& \sum_{n=1}^N \sum_{i=1}^\infty \left(\left( |\lsm_{n,i} - \lsm^\prime_{n,i}| + 
|w_{n,i} - w^\prime_{n,i}| + |r_{n,i} - r^\prime_{n,i}| \right) \wedge1 \right) \\
&\quad+ \sum_{n=1}^N \sum_{i=1}^\infty \left( \mathbf{1}\{k_{n,i} \neq k^\prime_{n,i}\} + 
\mathbf{1}\{A_{n,i} \neq A^\prime_{n,i}\}\right) \\
&\quad+ \sum_{n=1}^N 
|z_{n} - z^\prime_{n}| + \sum_{k=1}^K |s_{k} - s^\prime_{k}|.
\end{split}
\end{equation}
One can check that $d(\cdot,\cdot)$ is separable and locally compact; more detail is given on page 82 of \cite{B1}.  One
can also check that the sets $E_M\subset S$, $M > 1$, defined by $z_n \le M$, 
$\lsm_{n,i} \le M$, $1/M \le w_{n,i} \le M$ and $1/M \le s_k \le M$, for
all $n= 1,\ldots,N$, $k=1,\ldots,K$ and $i=1,\ldots,z_n$ are compact with respect to $d(\cdot,\cdot)$.   We
equip $S$ with the standard Borel $\sigma$-algebra inherited from $d(\cdot,\cdot)$, which we denote by 
$\mathscr{S}$.
In Lemma \ref{lem4.9.1}, we will show $\|\cdot \|_L$, $\|\cdot \|_R$ and $\|\cdot \|_A$ are continuous in 
$d(\cdot,\cdot)$.

At the end of Section 6, we will employ the partial completion $\bar{S}$ of $S$ that is obtained by allowing the weighted residual
service times $w_{n,i}$  to take values in $[0,\infty)$ rather than just $(0,\infty)$.  
Otherwise, the construction of $\bar{S}$ is the same as that
just given for $S$.  The metric $\bar{d}(\cdot,\cdot)$ is defined analogously to $d(\cdot,\cdot)$, and the Borel 
$\sigma$-algebra $\bar{\mathscr{S}}$ is defined correspondingly.  Under $\bar{d}(\cdot,\cdot)$, the sets 
$\bar{E}_M\subset \bar{S}$, $M > 1$, defined by $z_n \le M$, $\lsm_{n,i} \le M$, $0\le w_{n,i} \le M$ and $1/M \le s_k \le M$, for
all $n= 1,\ldots,N$, $k=1,\ldots,K$ and $i=1,\ldots,z_n$ are compact.  One can define projections from $\bar{S}$ onto
spaces $\bar{S}^{\prime}$ in the same manner as was done from $S$ onto $S^{\prime}$. 

The Markov process $X(t)$, $t \ge 0$, underlying the network is defined to be the right continuous process taking
values $x$ in $S$ whose evolution is determined by the given JSQ rule together with the assigned service discipline.  Jobs
$(n,i)$ are allocated service according to rates $R_{n,i}(t)$ (the service effort per job) that are assumed to be 
constant in between arrivals and departures
of jobs at the queues.  
Over such an interval, $\Lbig_{n,i}(t)$ increases at rate $\mu_{j_{n,i}}$, and $W_{n,i}(t)$ and $S_k(t)$ decrease at rates
$\mu_{j_{n,i}}R_{n,i}(t)$ and $\alpha_k$, respectively.  (We write $\Lbig_{n,i}(t)$ for the age functions to 
avoid possible confusion later with constants $L_i$ that will be introduced.)
Upon an arrival or departure, rates are re-assigned according to the discipline.  
The restriction made here for the discipline, that service rates remain constant between arrivals and departures of
jobs, is for convenience, and allows one to inductively construct $X(\cdot)$ over increasing times in a simple way.  
The standard service
disciplines satisfy this property.  We also note that the construction here is not restricted to the JSQ rule, and
applies to other rules for assigning arriving jobs to a queue.

Along the lines of page 85 of \cite{B1}, a filtration $(\mathcal{F}_{t})$, $t\in
[0,\infty]$, can be assigned to $X(\cdot)$ so that $X(\cdot)$ is Borel right
and, in particular, is strong Markov. The processes $X(\cdot)$ fall into the
class of piecewise-deterministic Markov processes, for which the reader is
referred to Davis \cite{D} for more detail.
\subsection*{Recurrence}
The Markov process $X(\cdot)$ is said to be {\em Harris recurrent} if, for some
nontrivial $\sigma$-finite measure $\varphi$,
\[ \varphi(B) > 0 \quad \mbox{~implies~} \quad P_{x}(\eta_{B} = \infty) = 1
\quad \mbox{~for all~} x \in S, \]
where $\eta_{B} = \int^{\infty}_{0}\mathbf{1}\{{X(t)\in B}\}dt$. If $X(\cdot)$ is Harris
recurrent, it possesses a stationary measure $\pi$ that is unique up to a
constant multiple. When $\pi$ is finite, $X(\cdot)$ is said to be {\em positive
Harris recurrent}.

A practical condition for determining positive Harris recurrence can be given by
using petite sets. A nonempty set $A \in \mathscr{S}$ is said to be {\em petite}
if for some fixed probability measure $a$ on $(0, \infty)$ and some nontrivial
measure $\nu$ on $(S, \mathscr{S})$,
\[ \nu(B) \le \int^{\infty}_{0}P^{t}(x,B)a(dt) \]
%
for all $x \in A$ and $B \in \mathscr{S}$. Here, $P^{t}(\cdot\,, \cdot)$, $t\ge
0$, is the semigroup associated with $X(\cdot)$. As mentioned in the
introduction, a petite set $A$ has the property that each set $B$ is ``equally
accessible'' from all points $x \in A$ with respect to the measure $\nu$. Note
that any nonempty measurable subset of a petite set is also petite.

For given $\delta > 0$, set
\[ \tau^{A}(\delta) = \mbox{~inf~}\{t \ge \delta: X(t) \in A\} \]
and $\tau^{A} = \tau^{A}(0)$. Then $\tau^{A}(\delta)$ is a stopping time.
Employing petite sets and $\tau^{A}(\delta)$, one has the following
characterization of Harris recurrence and positive Harris recurrence. (The
Markov process and state space need to satisfy minimal regularity conditions,
as on page 86 of \cite{B1}.) The criteria are from Meyn and Tweedie \cite{MT};
discrete time analogs of the different parts of the proposition have long been
known, see, for instance, Nummelin \cite{N} and Orey \cite{O}.
\begin{thm}
\label{thm2.7.1}
(a) A Markov process $X(\cdot)$ is Harris recurrent if and only if there exists
a closed petite set $A$ with
\begin{equation}
\label{eq2.7.2}
P_{x}(\tau^{A} < \infty) = 1 \quad \mbox{~for all~} x \in S.
\end{equation}
(b) Suppose the Markov process $X(\cdot)$ is Harris recurrent. Then $X(\cdot)$
is positive Harris recurrent if and only if there exists a closed petite set $A$
such that, for some $\delta > 0$,
\begin{equation}
\label{eq2.8.1}
\underset{x \in A}{\mbox{$\sup$~}}E_{x}[\tau^{A}(\delta)] < \infty .
\end{equation}
\end{thm}

\subsection*{Theorem \ref{thm1.14.1} and its corollary} Theorem \ref{thm1.14.1} follows from 
Theorem \ref{thm1.13.1}, Theorem \ref{thm2.7.1} and the elementary continuity result, 
Lemma \ref{lem4.9.1}.  To see this, note that both
conditions (\ref{eq2.7.2}) and (\ref{eq2.8.1}) of Theorem \ref{thm2.7.1} are immediate consequences
of (\ref{eq1.13.2}) of Theorem \ref{thm1.13.1}, with $A=A_M$, for appropriate $M$, and $\delta=1$ since,
in Theorem \ref{thm1.14.1}, $A_M$ is assumed to be petite.  By Lemma \ref{lem4.9.1}, the norm
$||\cdot||$ in (\ref{eq1.11.3}) is continuous in the metric $d(\cdot)$, and hence $A_M$ is also closed.
It therefore follows from Theorem \ref{thm2.7.1} that $X(\cdot)$ is positive Harris recurrent, which
implies Theorem \ref{thm1.14.1}.
 
Corollary \ref{cor1.15.2} follows immediately from Theorem \ref{thm1.14.1} and
the assertion, before the statement of the corollary, that the sets $A_{M}$ are
petite under the assumptions (\ref{eq1.14.2}) and (\ref{eq1.15.1}). A somewhat
stronger version of the analogous assertion for queueing networks is
demonstrated in Proposition 4.7 of \cite{B1}. (The proposition states that the
sets $A$ are uniformly small.) The reasoning is the same in both cases and does
not involve the JSQ rule or the service discipline. The argument, in essence, requires that
one wait long enough for the network to have at least a given positive
probability of being empty; the time $t$ does not depend on $x$ for $\|x\|\le M$.
This will follow from (\ref{eq1.14.2}) and the definition of $\|\cdot\|$ in Section 4,
since the work in the network is bounded by a linear function of M.  Since the residual interarrival
times are also bounded by a linear function of $M$, by using (\ref{eq1.15.1}), one can also show
that the joint distribution function of the residual interarrival times has an
absolutely continuous component at this time, whose density is bounded away
from 0. It will follow that the set $A_{M}$ is petite with respect to $\nu$,
with $a$ chosen as the point mass at $t$, if $\nu$ is concentrated on the
empty states, where it is a small enough multiple of $|\mathcal{R}|$-dimensional
Lebesque measure restricted to a small cube.
\section{A useful routing lemma}  In this section, we rephrase the conditions (\ref{eq1.10.1})  and
(\ref{eq1.10.2}) that are used to define subcriticality for class independent and station independent 
JSQ networks by using sums that will be more convenient for us to work with when proving Theorem \ref{thm1.13.1}
in Sections 4 and 5.  For this, we employ Lemma \ref{lem3.2.3}.  The desired sums for class independent and station
independent JSQ networks are then given in the following corollary.

For the lemma, we employ the following notation.  We consider $\beta_{k,A,n} \ge 0$ for $k=1,\ldots, K$,
$n=1,\ldots, N$ and $A \subseteq B_N$, where $B_N = \{1,\ldots, N\}$, 
and assume that, for each $k$ and $A$,
\begin{equation}
\label{eq3.2.1}
\beta_{k,A,n} > 0 \quad \text{iff} \quad \beta_{k,A,n^{\prime}} > 0 
\qquad \text{for all } n,n^{\prime}\in A,	
\end{equation}
that is, whether or not $\beta_{k,A,n}$ is zero does not depend on $n$, for $n\in A$.
For each $B\subseteq B_N$, let $\{r_{B,n}, n=1,\ldots,N\}$ be a probability distribution concentrated
on $B$, with $r_{B,n} > 0$ for $n\in B$, so that for each $k$, $B$, and $A\subseteq B$, and
$n$ restricted to $B$,
\begin{equation}
\label{eq3.2.2}
\gamma_{k,A,B}\stackrel{\text{def}}{=} r_{B,n} \beta_{k,A,n}   
\qquad \text{does not depend on } n.       
\end{equation}
\begin{lem}
\label{lem3.2.3}
Suppose that $\beta_{k,A,n}$,  $\gamma_{k,A,B}$ and $r_{B,n}$ are chosen as in (\ref{eq3.2.1})--(\ref{eq3.2.2}),
and moreover that $\gamma_{k,A,B}$ satisfies
\begin{equation}
\label{eq3.2.4}
\sum_{k} \sum_{A\subseteq B} \gamma_{k,A,B} \leq \rho \qquad \text{for all } B\subseteq B_N
\end{equation}
for some $\rho$.  Then, for each $k$ and $A$, there is a probability distribution 
$\{q_{k,A,n}, n=1,\ldots,N\}$ concentrated on $A$, so that
\begin{equation}
\label{eq3.2.5}
\sum_{k} \sum_{A\subseteq B_N} \beta_{k,A,n} q_{k,A,n} \leq \rho \qquad \text{for all } n.
\end{equation}
\end{lem}
As a consequence of Lemma \ref{lem3.2.3}, we obtain the following two inequalities from (\ref{eq1.10.1})
and (\ref{eq1.10.2}).
\begin{cor}
\label{cor3.3.1}
(a) Suppose that (\ref{eq1.10.1}) holds for a class independent JSQ network.  Then, for each $k$ and
$A$, there is a probability distribution $q_{k,A,n}$ concentrated on $A$ so that
\begin{equation}
\label{eq3.3.2}
\sum_{k} \sum_{A\subseteq B_N} \alpha_{k} p_{k,A} q_{k,A,n}m_n \leq \rho_1 \qquad \text{for all } n.
\end{equation}
(b) Suppose that (\ref{eq1.10.2}) holds for a station independent JSQ network.  Then, for each $k$ and
$A$, there is a probability distribution $q_{k,A,n}$ concentrated on $A$ so that
\begin{equation}
\label{eq3.3.3}
\sum_{k} \sum_{A\subseteq B_N} \alpha_{k} p_{k,A}  q_{k,A,n}m_{k,A} \leq \rho_2 \qquad 
\text{for all } n.
\end{equation}
\end{cor}
Corollary \ref{cor3.3.1} converts (\ref{eq1.10.1}) and (\ref{eq1.10.2}) into (\ref{eq3.3.2}) and
(\ref{eq3.3.3}), which will be easier to work with in Section 4.  When $\rho_1 < 1$, 
respectively, $\rho_2 < 1$, these inequalities imply that the network, where incoming jobs are assigned the
queue $n$ with probabities $q_{k,A,n}$, is subcritical.  On the other hand, by substituting $B$ for $B_N$ in the
inner sum and summing over $n$, one can check that the inequalities in (\ref{eq3.3.2}), respectively,
(\ref{eq3.3.3}), cannot be strict for all $n$.  These conditions therefore give an alternative 
characterization for subcriticality when $\rho_i < 1$.  (We will not need this in the paper.)

One can check that when the JSQ network is symmetric, one may set $q_{k,A,n} = 1/|A|$ for $n\in A$.  This follows
by summing the left side of (\ref{eq3.3.3}) over all $n$ and comparing the sum with the right side of (\ref{eq1.10.2}),
for $B = B_N$.  
\begin {proof}[Proof of Corollary \ref{cor3.3.1}]
(a)  We apply Lemma \ref{lem3.2.3}, setting
\begin{equation}
\label{eq3.4.1}
\beta_{k,A,n} = \alpha_{k} p_{k,A} m_n, \quad r_{B,n} = \mu_{n} \mathbf{1}\{n\in B\} {\Big /}  \sum_{n^{\prime} \in B} 
\mu_{n^{\prime}}, \quad \rho = \rho_1.
\end{equation}
It is easy to check that the conditions (\ref{eq3.2.1})--(\ref{eq3.2.2}) are satisfied with this choice of 
$\beta_{k,A,n}$ and $r_{B,n}$.  Substitution of these quantities into (\ref{eq3.2.4}) gives the quantity
in braces on the right side of (\ref{eq1.10.1}) for each choice of $B$, and substitution into (\ref{eq3.2.5})
gives (\ref{eq3.3.2}).  Part (a) of the corollary therefore follows from the lemma.

(b)  The argument for this part is analogous to the first part; here, we set
\begin{equation}
\label{eq3.5.1}
\beta_{k,A,n} = \alpha_{k} p_{k,A} m_{k,A}, \quad r_{B,n} = \mathbf{1}\{n\in B\} / |B|, 
\quad \rho = \rho_2.
\end{equation}
It is easy to check that the conditions (\ref{eq3.2.1})-(\ref{eq3.2.2}) are again satisfied.  Substitution into
(\ref{eq3.2.4}) gives the quantity in braces on the right side of (\ref{eq1.10.2}) and substitution into 
(\ref{eq3.2.5}) gives (\ref{eq3.3.3}).  Part (b) of the corollary therefore also follows from the lemma.
\end{proof}
We now prove the lemma.
\begin {proof}[Proof of Lemma \ref{lem3.2.3}]
For a family of probability distributions $q_{k,A,n}$, indexed by $k$ and $A$, and concentrated on $A$, set
$\zeta_{k,A,n}^{(q)} = \beta_{k,A,n} q_{k,A,n}$,
\begin{equation}
\label{eq3.6.1}
V(q) = \sum_n \left( \sum_k \sum_{A\subseteq B_N} \zeta_{k,A,n}^{(q)} - \rho \right)^2
\end{equation}
and
\begin{equation}
\label{eq3.6.2}
V^{\text{min}} = \min_q V(q).
\end{equation}
One can check that $V(q)$ is continuous in $q$ and that the set of $q$ is compact with respect to the metric
\begin{equation}
\label{eq3.6.3}
d(q,q^{\prime}) = \max_{k,A,n} |q_{k,A,n} - q_{k,A,n}^{\prime}|.
\end{equation}
So, $V^{\text{min}}$ is attained at some $q^{\text{min}}$.  We set 
\begin{equation}
\label{eq3.7.1}
\begin{split}
A^{\text{min}} =& \left\{n: \sum_k \sum_{A\subseteq B_N} \zeta_{k,A,n}^{\text{min}} > \rho \right\}, \\
\bar{A}^{\text{min}} =& 
\left\{n: \sum_k \sum_{A\subseteq B_N} \zeta_{k,A,n}^{\text{min}} \ge \rho \right\},
\end{split}
\end{equation}
where $\zeta_{k,A,n}^{\text{min}} \stackrel{\text{def}}{=} \beta_{k,A,n} q_{k,A,n}^{\text{min}}$.  In order to
show (\ref{eq3.2.5}), it suffices to show $A^{\text{min}} = \emptyset$.

We first claim that for any $k^{\prime}$, $A^{\prime} \subseteq B_N$ and $n_1 \in 
\bar{A}^{\text{min}}$, 
\begin{equation}
\label{eq3.7.2}
A^{\prime} \subseteq \bar{A}^{\text{min}} \qquad \text{if } \,
\zeta_{k^{\prime},A^{\prime},n_1}^{\text{min}} > 0.
\end{equation}
We argue by contradiction and show that if (\ref{eq3.7.2}) is violated for some $k^{\prime}$, $A^{\prime}$ and
$n_1$, then for 
appropriate $\tilde{q}$, $V(\tilde{q}) < V(q^{\text{min}})$, which is not possible.
The proof of this does not use (\ref{eq3.2.4}). 

For such $A^{\prime}$ and $n_1$, $n_1 \in A^{\prime}$ since $\zeta_{k^{\prime},A^{\prime},\cdot}^{\text{min}}$ is
concentrated on $A^{\prime}$.  We choose $n_2 \in A^{\prime} - \bar{A}_{\text{min}}$ and define
a new family of probability distributions $\tilde{q}_{k,A,\cdot}$ by
\begin{equation}
\nonumber
\begin{split}
\tilde{q}_{k,A,n} &= q_{k,A,n}^{\text{min}} \quad 
\text{unless $k=k^{\prime}$, $A=A^{\prime}$ and either $n=n_1$ or $n=n_2$}, \\
\tilde{q}_{k^{\prime},A^{\prime},n_1}  &= q_{k^{\prime},A^{\prime},n_1}^{\text{min}} - \epsilon, \\ 
\tilde{q}_{k^{\prime},A^{\prime},n_2}  &= q_{k^{\prime},A^{\prime},n_2}^{\text{min}} + \epsilon,
\end{split}
\end{equation}
where $\epsilon > 0$ is small.

For small enough $\epsilon >0$,
\begin{equation}
\label{eq3.8.1}
\begin{split}
&\left( \sum_k \sum_{A\subseteq{B_N}} \tilde{\zeta}_{k,A,n_1} - \rho \right)^2 -
\left( \sum_k \sum_{A\subseteq{B_N}} \zeta_{k,A,n_1}^{\text{min}} - \rho \right)^2 \\  
& \qquad < \left( \sum_k \sum_{A\subseteq{B_N}} \zeta_{k,A,n_2}^{\text{min}} - \rho \right)^2 -
\left( \sum_k \sum_{A\subseteq{B_N}} \tilde{\zeta}_{k,A,n_2} - \rho \right)^2,
\end{split}
\end{equation}
where $\tilde{\zeta}_{k,A,n} = \beta_{k,A,n} \tilde{q}_{k,A,n}$.  The inequality is obvious when
$n_1 \in A^{\text{min}}$, since the left side of (\ref{eq3.8.1}) will be negative and the right
side will be positive.  When $n_1 \in \bar{A}^{\text{min}} -A^{\text{min}}$, the inequality is also
true since the left side is bounded above by a multiple of $\epsilon^2$ (with the second term being
0), and the right side is bounded below by a multiple of $\epsilon$.  For $n \neq n_1,n_2$,
\begin{equation}
\label{eq3.9.1}
\left( \sum_k \sum_{A\subseteq{B_N}}  \tilde{\zeta}_{k,A,n} - \rho \right)^2 =
\left( \sum_k \sum_{A\subseteq{B_N}}  \zeta_{k,A,n}^{\text{min}} - \rho \right)^2.
\end{equation}
So, by (\ref{eq3.8.1}) and (\ref{eq3.9.1}), $V(\tilde{q}) < V(q^{\text{min}})$, which contradicts
the definition of $q^{\text{min}}$.  This shows (\ref{eq3.7.2}). 

Employing (\ref{eq3.7.2}), we now show that $A^{\text{min}} = \emptyset$.  One has
\begin{equation}
\label{eq3.10.1}
\begin{split}
\rho &\ge \sum_k \sum_{A\subseteq \bar{A}^{\text{min}}} \gamma_{k,A,\bar{A}^{\text{min}}}
= \sum_k \sum_{A\subseteq \bar{A}^{\text{min}}} \,\sum_{n=1}^N \gamma_{k,A,\bar{A}^{\text{min}}} q_{k,A,n}^{\text{min}}
\\ &= \sum_k \sum_{A\subseteq \bar{A}^{\text{min}}} \, \sum_{n \in \bar{A}^{\text{min}}} 
r_{\bar A^{\text{min}},n}\beta_{k,A,n} q_{k,A,n}^{\text{min}}
\\ &= \sum_{n \in \bar{A}^{\text{min}}} r_{\bar A^{\text{min}},n} \sum_k \sum_{A\subseteq B_N} \zeta_{k,A,n}^{\text{min}}.
\end{split}
\end{equation}
The inequality follows from (\ref{eq3.2.4}), the second equality follows from (\ref{eq3.2.2}) since $r_{B,\cdot}$
is concentrated on $B$, and the third equality follows from the definition of $\zeta_{k,A,n}^{\text{min}}$ and
(\ref{eq3.7.2}).  But, on account of the definitions of $A^{\text{min}}$ and $\bar{A}^{\text{min}}$, the last
quantity in (\ref{eq3.10.1}) is at least 
\begin{equation}
\nonumber
\rho \sum_{n \in \bar{A}^{\text{min}}} r_{A^{\text{min}},n} = \rho,
\end{equation}
with strict inequality holding if $A^{\text{min}} \neq \emptyset$ because all the terms in the sum are strictly positive.
Since the strict inequality contradicts (\ref{eq3.10.1}), this implies $A^{\text{min}} = \emptyset$, and 
hence the inequality in (\ref{eq3.2.5}), as desired.
\end{proof}
\section{Definitions of norms and basic inequalities}
This section introduces the norms and provides the basic inequalities we will need in Sections 5 and 6 for the proofs of 
Theorems \ref{thm1.13.1} and \ref{thm1.17.2}.
The section consists of two subsections.  We first define the norms appearing in (\ref{eq1.11.3}), $\|\cdot\|_L$,
$\|\cdot\|_R$ and $\|\cdot\|_A$, in terms of which $\|\cdot\|$ was defined.  We then state and prove
Propositions \ref{prop4.10.2} and \ref{prop4.13.1}. These propositions give inequalities on the decrease of $\|\cdot\|$ and lie 
at the heart of the analysis in Sections 5 and 6.  Proposition \ref{prop4.13.1} is the only place in the first six sections
of the paper where the JSQ property is employed.
%
\subsection*{Definition of the norms} In (\ref{eq1.11.3}), we defined the norm $\|\cdot\|$ in terms of the norms 
$\|\cdot\|_L$, $\|\cdot\|_R$ and  $\|\cdot\|_A$.  We now define these norms.

Recall from Section 1 that $z_n$ denotes the number of jobs at queue $n$; $w_n$ denotes the weighted workload
at the queue, which is defined in (\ref{eq1.16.6}) along with $w_{n,i}$; $r_{n,i}$ denotes the service effort for
job $(n,i)$; and $s_k$ denotes the weighted interarrival time at the arrival stream $k$, and is given by
(\ref{eq1.16.7}).  The notation $Z_{n}(\cdot)$, $W_{n}(\cdot)$, $W_{n,i}(\cdot)$, $R_{n,i}(\cdot)$ and
$S_{k}(\cdot)$ will be used for the corresponding quantities of the process $X(\cdot)$.  We also employ the 
arrival rates $\alpha_k$, mean service times $m_j$, service rates $\mu_j$ and transition probabilities
$p_{k,A}$ that were introduced earlier, as well as the transition probabilities $q_{k,A,n}$ that are given
in Corollary \ref{cor3.3.1} for class independent and station independent JSQ networks.  We will set
$\epsa = 1 - \rho_1$ for class independent networks and $\epsa = 1 - \rho_2$ for station
independent networks, where $\rho_i$ are the traffic intensities.

For $x\in S$, set 
\begin{equation}
\label{eq4.3.1}
\|x\|_L = \sum_{n=1}^{N} \|x\|_{L,n}, \quad \|x\|_R = \sum_{n=1}^{N} \|x\|_{R,n}, 
\quad \|x\|_A = \sum_{k=1}^{K} \|x\|_{A,k}.
\end{equation}
(The subscripts $L$, $R$ and $A$ are mnemonics for ``left", ``right" and ``arrivals".)  We define these individual
components as follows.  Since $\|x\|_{L,n}$ is defined in terms of quantities obtained from $\|x\|_{R,n}$, we define
the latter first.

For $n = 1,\ldots,N$, set
\begin{equation}
\label{eq4.6.1}
\|x\|_{R,n} = \sum_{i=1}^{z_n} \stackrel{\circ}{m}_{k_{n,i},A_{n,i}} \psi_{W}(w_{n,i}),
\end{equation}
where $w_{n,i} = \mu_{j_{n,i}} v_{n,i}$.
%
%
%
The other components in (\ref{eq4.6.1}) are defined as follows:
\begin{equation}
\label{eq4.3.3}
\stackrel{\circ}{m}_{k,A} = \begin{cases}
1 \quad  & \text{for class independent networks}, \\
m_{k,A}  \quad & \text{for station independent networks}, 
\end{cases}
\end{equation}
and $k_{n,i}$ and $A_{n,i}$ denote the arrival stream and selection set for the $i^{\text{th}}$ job currently
at queue $n$. 
The function $\psi_{W}: \mathbb{R}_{+,0} \rightarrow \mathbb{R}_{+,0}$ is required to be continuously differentiable,
with $\psi_{W}(0) = 0$, $\psi_{W}^{\prime}(y) > 0$, $\psi_{W}^{\prime}(y) \nearrow \infty$ as
$y \nearrow \infty$, and
\begin{equation}
\label{eq4.6.2}
\int_{0}^{\infty} \psi_{W}(\mu_{j}y) F_{j}(dy) \le \epsd \qquad \text{for all } j,
\end{equation}
where $\epsd = (\epsa)^{2}/40$.  Since $F_j(\cdot)$ has finite mean, it is not difficult to choose
such a $\psi_W(\cdot)$.

%
%
%
%
%
%
%
%

The norm $\|\cdot\|_R$ will be the main contributor to $\|\cdot\|$ for jobs with large residual
service times; service of such a job $(n,i)$ sharply decreases $\|\cdot\|_R$ when $\psi_{W}^{\prime}(w_{n,i})$ is
large.  The term $\stackrel{\circ}{m}_{k,A}$ is needed because of the different definitions of the
traffic intensities $\rho_1$ and $\rho_2$ for class independent and station independent networks.
%

For $n=1,\ldots,N$, we set
\begin{equation}
\label{eq4.3.2}
\|x\|_{L,n} = \left(\sum_{i=1}^{z_n} \stackrel{\circ}{m}_{k_{n,i},A_{n,i}} (w_{n,i}^{+} \wedge \ellv)\right) \psi_{Z}(z_n).
\end{equation}
%
%
%
%
Here,
\begin{equation}
\label{eq4.4.1}
w_{n,i}^+ = w_{n,i} + \epsb
\end{equation}
and 
\begin{equation}
\label{eq4.4.2}
\psi_{Z}(y) = \begin{cases}
\epsa + (\epsc/\ellv) \text{log}(y+1)  \quad & \text{for } y\in [0,\ellw], \\
\epsa + (\epsc/\ellv) \text{log}(\ellw +1)  \quad &  \text{for } y > \ellw, 
\end{cases}
\end{equation}
for a small $\epsc > 0$, which will be defined in (\ref{eq4.13.0}) in terms of
$\epsa$ and other quantities. 

We choose $\ellv$ in (\ref{eq4.3.2}) and (\ref{eq4.4.2}) so that
\begin{equation}
\label{eq4.6.3}
\psi_{W}^{\prime}(\ellv) = 2\ellx
\end{equation}
for given $\ellx$, with $\ellx \ge 4$ and $\ellx$ large enough so that $\ellv \ge \epsb$.  We will 
specify $\ellx$ later as we find convenient.  We choose $\ellw$ so that
\begin{equation}
\label{eq4.7.1}
\psi_{Z}(\ellw) = \epsa + (\epsc/\ellv) \text{log}(\ellw+1) = \ellx;
\end{equation}
it follows that when $w \ge \ellv$,
\begin{equation}
\label{eq4.7.2}
\psi_{Z}(y) \le \frac{1}{2} \psi_{W}^{\prime}(w) \qquad \text{for all } y.
\end{equation}
Note that $\ellv \rightarrow \infty$ and $\ellw \rightarrow \infty$ as $\ellx \rightarrow \infty$.
The inequality (\ref{eq4.7.2}) is used in Proposition \ref{prop4.10.2}, and will tell us that, for large
residual service times, the norm $\|\cdot\|_R$ is ``more powerful" than $\|\cdot\|_L$.
%
%

The following provides some motivation for the definition of $\|x\|_{L,n}$.  The norm $\|\cdot\|_L$ will be the 
main contributor to $\|\cdot\|$ for jobs with moderate residual service times.  The terms $w_{n,i}^{+} \wedge \ellv$ and
$\psi_{Z}(z_n)$ are each bounded, with the term $w_{n,i}^+$ decreasing continuously over time as the 
corresponding job is served; we employ $w_{n,i}^+$ rather than $w_{n,i}$, in (\ref{eq4.3.2}), to ensure 
$\|x\|_{L,n} \rightarrow \infty$ as $z_n \rightarrow \infty$.  The inclusion of the term $\psi_{Z}(z_n)$, which
is nondecreasing in $z_n$, is reasonable since more jobs at a queue should correspond to a greater
value of $\|\cdot\|_L$.

For $k=1,\ldots,K$, we set
\begin{equation}
\label{eq4.8.1}
\|x\|_{A,k} = (1 + \epsa/2) \sum_{A\subseteq B_N}\left( \sum_{n=1}^{N} 
p_{k,A} q_{k,A,n} \stackrel{\circ}{m}_{k,A} \psi_{Z}(z_n)\right) \psi_{A}(s_k).
\end{equation}
Here, $q_{k,A,n}$ is chosen as in the corollary to Lemma \ref{lem3.2.3} and $s_k = \alpha_{k}u_k$.  We require
that $\psi_{A}(\cdot)$ be locally Lipschitz on $[0, \infty)$, with
\begin{equation}
\label{eq4.8.2}
\psi_{A}(y) = M_1 - y \qquad \text{for } y \in [0,M_1],
\end{equation}
and $\psi_{A}^{\prime}(y) > 0$ for $y \in (M_1, \infty)$ and appropriate
$M_1 \ge 1$, with $\psi_{A}^{\prime}(y) \nearrow \infty$ as $y \nearrow \infty$, so that 
\begin{equation}
\label{eq4.8.3}
\int_{M_1 / \alpha_k}^{\infty} \left( \psi_{A} (\alpha_{k}y) + \alpha_{k}y \right) G_{k}(dy) \le \epsf.
\end{equation}
Since $G_{k}(\cdot)$ has finite mean, it is not difficult to choose such $M_1$ and $\psi_A(\cdot)$.
Because $G_{k}(\cdot)$ has  mean 
$1/\alpha_k$ and (\ref{eq4.8.2}) is satisfied, (\ref{eq4.8.3}) implies that
\begin{equation}
\label{eq4.8.4}
M_1 - \int_{0}^{\infty}  \psi_{A} (\alpha_{k}y)  G_{k}(dy) \ge 1 - \epsf,
\end{equation}
which will be used in the proof of Proposition \ref{prop4.22.4}.

The norm $\|\cdot\|_A$ is chosen so that it interfaces properly with $\|\cdot\|_L$: in between arrivals, 
$\|X(t)\|_A$ will increase more slowly than $\|X(t)\|_L$ + $\|X(t)\|_R$ decreases; at an arrival, the average
decrease of $\|X(t)\|_A$ will more than offset the increase in $\|X(t)\|_L$ + $\|X(t)\|_R$ because of (\ref{eq4.6.2}) 
and (\ref{eq4.8.4}), and the choice of $\epsc$ in (\ref{eq4.4.2}).

Note that the weighted ages $\lsm_{n,i}$ are not employed in the definition of $\|\cdot\|$.  They are not needed,
in particular, to show petiteness of bounded sets in $\|\cdot\|$ under (\ref{eq1.14.2}) and (\ref{eq1.15.1}), since
they do not appear when $z_n = 0$.  (See, e.g., the end of Section 2.)

In Section 2, we demonstrated Theorem \ref{thm1.14.1} by employing Theorem \ref{thm1.13.1} and Lemma \ref{lem4.9.1},
with the latter asserting that $\|\cdot\|$ is continuous in the metric $d(\cdot,\cdot)$.  Having defined
$\|\cdot\|$, we now state and prove the lemma.
\begin{lem}
\label{lem4.9.1}
The norm $\|\cdot\|$ given by (\ref{eq1.11.3}) is continuous in the metric $d(\cdot,\cdot)$ given by (\ref{eq2.3.1}). 
\end{lem}
\begin {proof}[Proof]
It suffices to show each of the norms $\|\cdot\|_L$, $\|\cdot\|_R$ and $\|\cdot\|_A$ is continuous in 
$d(\cdot,\cdot)$.  The argument in each case is elementary.  Noting that $d(x,x^{\prime}) < 1$ implies
$z_n = z_{n}^{\prime}$ for all $n$, and that $w_{n,i}^{+} = w_{n,i} +\epsb$ is continuous in
$w_{n,i}$, for given $i$, the continuity of $\|\cdot\|_L$ is not difficult to see.  For the same reasons and 
since $\psi_{W}(\cdot)$ and $\psi_{A}(\cdot)$ are locally Lipschitz, $\|\cdot\|_R$ and $\|\cdot\|_A$ are also continuous.  
\end{proof}
\subsection*{Basic inequalities} In this subsection, we demonstrate Propositions \ref{prop4.10.2} and 
\ref{prop4.13.1}, which are at the foundation of the analysis in Sections 5 and 6.  The evolution of 
$X(t)$ between arrivals of jobs is deterministic.  In Proposition \ref{prop4.10.2}, we provide upper
bounds on the rate of change of $\|X(t)\|$ there by employing its components $\|X(t)\|_L$, $\|X(t)\|_R$ and 
$\|X(t)\|_A$, which exist almost everywhere since the underlying functions are locally Lipschitz except where
jobs arrive or depart, with jumps being negative at departures.  We set
\begin{equation}
\label{eq4.10.1}
\stackrel{\circ}{\mu}_n = \begin{cases}
\mu_n & \quad \text{for class independent networks}, \\
1 & \quad \text{for station independent networks}, 
\end{cases}
\end{equation}
and note that 
\begin{equation}
\label{eq4.10.1a}
\stackrel{\circ}{\mu}_n = \,\stackrel{\circ}{m}_{k,A} \mu_j.
\end{equation}

For Proposition \ref{prop4.10.2}, as elsewhere in the paper, we assume the network is either class or station
independent.  
\begin{prop}
\label{prop4.10.2}
For every subcritical JSQ network,
\begin{equation}
\label{eq4.10.3}
\begin{split}
\|X(t)\|_{L}^{\prime} + \|X(t)\|_{R}^{\prime} &\le \|X(t)\|_{L}^{\prime} + \frac{1}{2}  \|X(t)\|_{R}^{\prime} \\ 
&\le \sum_{n} \stackrel{\circ}{\mu}_n(\epsa - \psi_{Z}(Z_{n}(t)))
\end{split}
\end{equation}
for almost all $t$.  Moreover, for any subset $\mathcal{K}$ of $\{1,\cdots,K\}$,
\begin{equation}
\label{eq4.10.4}
\sum_{k\in \mathcal{K}} \|X(t)\|_{A,k}^{\prime} 
\le (1 +\epsa/2)\rho_i \sum_{n} \stackrel{\circ}{\mu}_n\psi_{Z}(Z_n (t))
\end{equation}
for almost all $t$, for $i = 1,2$.  Consequently, for almost all $t$,
\begin{equation}
\label{eq4.10.5}
\|X(t)\|^{\prime} \le (\epsa/2) \sum_{n} \stackrel{\circ}{\mu}_{n}(2 - \psi_{Z}(Z_n (t))).
\end{equation}
\end{prop}
The inequality (\ref{eq4.10.5}) follows immediately from (\ref{eq4.10.3}) and (\ref{eq4.10.4}), with
$\mathcal{K} = \{1,\ldots,K \}$, together with (\ref{eq1.11.3}), since $\epsa = 1 - \rho_i$.  The
purpose of the term $\frac{1}{2} \|X(t)\|_{R}^{\prime}$ in the middle quantity in (\ref{eq4.10.3}) is
to permit a sharper upper bound, for large $\|X(t)\|_R$, by including the contribution from $\psi_{W}(\cdot)$
in (\ref{eq4.6.1}).  This strengthening of (\ref{eq4.10.5}) will be applied in the proof of case (c)
of Proposition \ref{prop4.22.4}.  Also, since
$\psi_{A}^{\prime}(y) \rightarrow \infty$ as $y\rightarrow \infty$, the inclusion in the first sum in (\ref{eq4.10.4})
of terms corresponding to certain $k\notin \mathcal{K}$ can improve the bound on the right side, which will
allow us to strengthen (\ref{eq4.10.5}) for large $\|X(t)\|_A$. 
\begin {proof}[Proof of Proposition \ref{prop4.10.2}]
We need to demonstrate the first two displays.  For (\ref{eq4.10.3}), we note that 
\begin{equation*}
\begin{split}
& \|X(t)\|_{L}^{\prime} + \frac{1}{2} \|X(t)\|_{R}^{\prime} \\
& \quad \le - \sum_{n} \stackrel{\circ}{\mu}_{n} \psi_{Z}(Z_{n}(t)) 
\sum_{i=1}^{Z_{n}(t)} \mathbf{1} \{W_{n,i}^{+}(t) < \ellv \}R_{n,i}(t) \\
& \qquad - \frac{1}{2}\sum_{n} \stackrel{\circ}{\mu}_{n}  
\sum_{i=1}^{Z_{n}(t)} \psi_{W}^{\prime}(W_{n,i}(t)) \mathbf{1}\{W_{n,i}^{+}(t) \ge \ellv \}R_{n,i}(t) \\
& \quad \le - \sum_{n} \stackrel{\circ}{\mu}_{n} \psi_{Z}(Z_{n}(t)) 
\sum_{i=1}^{Z_{n}(t)} \mathbf{1}\{W_{n,i}^{+}(t) < \ellv \}R_{n,i}(t) \\
& \qquad - \sum_{n} \stackrel{\circ}{\mu}_{n} \psi_{Z}(Z_{n}(t)) 
\sum_{i=1}^{Z_{n}(t)} \mathbf{1}\{W_{n,i}^{+}(t) \ge \ellv \}R_{n,i}(t) \\
& \quad = - \sum_{n} \stackrel{\circ}{\mu}_{n} \psi_{Z}(Z_{n}(t))\mathbf{1}\{Z_{n}(t) > 0\} 
\le \sum_{n} \stackrel{\circ}{\mu}_{n} (\epsa - \psi_{Z}(Z_{n}(t))
\end{split}
\end{equation*}
holds almost everywhere. 
The first inequality follows from the definitions of $\|\cdot\|_L$, $\|\cdot\|_R$ and (\ref{eq4.10.1a}), since
$Z_n(\cdot)$ is constant almost everywhere.  For this, note that
$W^{\prime}_{n,i}(t)  = -\mu_{j_{n,i}}R_{n,i}(t)$ almost everywhere.
The second inequality follows from (\ref{eq4.7.2}), with the last
inequality using $\psi_Z(0) = \epsa$ and $\sum_i R_{n,i}(t) = 1$, when $Z_n(t) > 0$.  This implies (\ref{eq4.10.3}).

For (\ref{eq4.10.4}), we apply the definition of $\|\cdot\|_{A,k}$ in (\ref{eq4.8.1}) to obtain, for almost all $t$,
\begin{equation*}
\begin{split}
\sum_{k \in \mathcal{K}} \|X(t)\|_{A,k}^{\prime} 
& \le (1 + \epsa/2) \sum_{k\in \mathcal{K}} \sum_{A} \sum_{n} 
\alpha_k p_{k,A} q_{k,A,n} \stackrel{\circ}{m}_{k,A} \psi_{Z}(Z_{n}(t)) \\
& \le (1 + \epsa /2) \rho_i \sum_n \stackrel{\circ}{\mu}_{n} \psi_{Z}(Z_{n}(t)),
\end{split}
\end{equation*}
where the first inequality follows from $\psi_{A}^{\prime} (y) \ge -1$, and the second inequality follows from the
corollary to Lemma \ref{lem3.2.3}.
\end{proof}
We still need to specify the constant $\epsc$ that was employed in (\ref{eq4.4.2}); for later reference, we also recall
the constant $\epsd$ from equations (\ref{eq4.6.2}) and (\ref{eq4.4.1}):
\begin{equation}
\label{eq4.13.0}
\epsd = M_{1}\!\stackrel{\circ}{m}\!\!\mbox{}^{\text{ratio}} \epsc = (\epsa)^{2}/40.
\end{equation}
Recall that $M_{1}$ is specified in (\ref{eq4.8.2})--(\ref{eq4.8.3}) 
and $\stackrel{\circ}{m}\!\!\mbox{}^{\text{ratio}}$ is defined as in the equation 
before (\ref{eq1.16.3b}).
%
%
In Proposition \ref{prop4.13.1}, we show that, for this choice of 
$\epsd$ and $\epsc$, the expected increase in $\|\cdot\|$ is nonpositive
at the time $T$ of the first arrival of a job in the network.  We 
note that for $X(0) = x$ fixed, $T$ is deterministic, as is the evolution of $X(\cdot)$ up through $T-$.  
\begin{prop}
\label{prop4.13.1}
For every JSQ network,
\begin{equation}
\label{eq4.14.1}
E_x[\|X(T)\|] \le \|X(T-)\| \qquad \text{for all } x.
\end{equation}
\end{prop}
\begin{proof}
We consider the contribution to $\|\cdot\|$ from $\|\cdot\|_L$, $\|\cdot\|_R$ and $\|\cdot\|_A$, assuming
that a single arrival occurs from stream $k$ at time $T$; when arrivals simultaneously occur from other streams,
the corresponding bounds can be applied sequentially.

For $\|\cdot\|_L$, one has 
\begin{equation}
\label{eq4.14.2}
\begin{split}
& E_{x}[\|X(T)\|_L] - \|X(T-)\|_L 
= \sum_A \sum_n p_{k,A} q_{k,A,n}^{*} \\
&\qquad \times \sum_{i=1}^{Z_{n}(T-)} \stackrel{\circ}{m}_{k_{n,i},A_{n,i}}
(W_{n,i}^{+} (T) \wedge \ellv) (\psi_{Z}(Z_{n}(T-) +1) - \psi_{Z}(Z_{n}(T-))) \\
& \qquad + \sum_A \sum_n p_{k,A} q_{k,A,n}^{*} \stackrel{\circ}{m}_{k,A}
\psi_{Z} (Z_{n}(T-) +1) \int_{0}^{\infty} ((\mu_{j}y + \epsb)\wedge{\ellv}) F_{j}(dy), 
\end{split}
\end{equation}
for each $x$, where, for given $k$ and $A$, $q_{k,A,n}^{*}$ is the probability that the arriving job is assigned to 
queue $n$.  As previously, $j=(k,A,n)$.  Because of the JSQ rule, $q_{k,A,n}^{*}$ is concentrated on the shortest queues in $A$.
Since by (\ref{eq4.4.2}),
\begin{equation}
\label{eq4.15.0}
\psi_{Z}(Z_{n}(T-) +1) - \psi_{Z}(Z_{n}(T-)) \le \epsc / \ellv(Z_{n}(T-) +1) \le \epsc / \ellv,
\end{equation}
the first term on the right side of (\ref{eq4.14.2}) is at most
\begin{equation}
\label{eq4.15.1}
\epsc \!\stackrel{\circ}{m}\!\!\mbox{}^{\text{ratio}} \sum_A p_{k,A} \stackrel{\circ}{m}_{k,A}.
\end{equation}

On the other hand, since
\begin{equation*}
\int_{0}^{\infty} ((\mu_{j}y + \epsb)\wedge{\ellv}) F_{j}(dy)
 \le \mu_j \int_{0}^{\infty} y F_{j}(dy) + \epsb 
= 1 + \epsb,
\end{equation*}
which does not depend on $n$, and since $\psi_{Z}(y)$ is increasing in $y$ and $q_{k,A,n}$ is concentrated
on A, the last term on the right side of (\ref{eq4.14.2}) is at most
\begin{equation}
\label{eq4.15.2}
(1 + \epsb) \sum_{A} \sum_{n} 
 p_{k,A} q_{k,A,n} \stackrel{\circ}{m}_{k,A} \psi_{Z}(Z_{n}(T-)+1).
\end{equation}
In other words, removing the truncation by $\ellv$ in the integral and replacing $q_{k,A,n}^{*}$ (which is
concentrated on the shortest queues in $A$) by $q_{k,A,n}$
can only increase this term.  Note that this is the only place in the first six sections of the paper 
where the JSQ property is employed.
It follows from the bounds (\ref{eq4.15.1}) and (\ref{eq4.15.2}) that
\begin{equation}
\label{eq4.16.1}
\begin{split}
& E_{x}[\|X(T)\|_L] - \|X(T-)\|_L \\ 
& \quad \le (1 + \epsb) \sum_{A} \sum_{n} p_{k,A} q_{k,A,n} \stackrel{\circ}{m}_{k,A} 
(\psi_{Z} (Z_n (T-) + 1) + \!\stackrel{\circ}{m}\!\!\mbox{}^{\text{ratio}} \epsc).
\end{split}
\end{equation}

For $\|\cdot\|_R$, it follows from (\ref{eq4.6.1}) and (\ref{eq4.6.2}) that
\begin{equation}
\label{eq4.16.2}
\begin{split}
& E_{x}[\|X(T)\|_R] - \|X(T-)\|_R \\ 
& \quad = \sum_{A} \sum_{n} p_{k,A} q_{k,A,n}^{*} \stackrel{\circ}{m}_{k,A} 
\int_{0}^{\infty}\psi_{W} (\mu_{j}y) F_j (dy) \\
& \quad \le \epsd \sum_{A} p_{k,A} \stackrel{\circ}{m}_{k,A}.
\end{split}
\end{equation}

On the other hand, it follows from (\ref{eq4.8.1}) that
\begin{equation}
\label{eq4.16.3}
\begin{split}
& E_{x}[\|X(T)\|_A] - \|X(T-)\|_A \\ 
& \quad \le  (1 + \epsa /2)(\epsc /\ellv) M_1 \sum_{A} p_{k,A} \stackrel{\circ}{m}_{k,A} \\
& \qquad  -(1 - \epsf)(1 + \epsa /2) \sum_{A} \sum_{n} p_{k,A} q_{k,A,n} \stackrel{\circ}{m}_{k,A} 
\psi_{Z} (Z_n (T-) + 1). 
\end{split}
\end{equation}
In the first term, the factors $\epsc / \ellv$ and $M_1$ are due to (\ref{eq4.15.0}) and (\ref{eq4.8.2}),
since $\psi_A (0) = M_1$, and in the second term, the factor $1 - \epsf$ is due to (\ref{eq4.8.4}).

Combining (\ref{eq4.16.1})-(\ref{eq4.16.3}), it follows that 
\begin{equation*}
\begin{split}
& E_{x}[\|X(T)\|] - \|X(T-)\| \le 
(4 M_1 \!\stackrel{\circ}{m}\!\!\mbox{}^{\text{ratio}} \epsc + \epsd) \sum_A p_{k,A} \stackrel{\circ}{m}_{k,A} \\
& \quad  -[(1 - \epsf)(1 + \epsa /2) - (1 + \epsb)] 
\sum_{A} \sum_{n} p_{k,A} q_{k,A,n} \stackrel{\circ}{m}_{k,A} 
\psi_{Z} (Z_n (T-) + 1). 
\end{split}
\end{equation*}
Since $\psi_Z (y) \ge \epsa$ for all $y$, it follows from (\ref{eq4.13.0}) that this
is at most
\begin{equation*}
-(\epsa^{2} / 8) \sum_A p_{k,A} \stackrel{\circ}{m}_{k,A} \, \le 0.
\end{equation*}
So
\begin{equation*}
E_{x}[\|X(T)\|] \le \|X(T-)\|,
\end{equation*}
as desired.
\end{proof}
The following upper bound is a consequence of Propositions \ref{prop4.10.2} and \ref{prop4.13.1} and the strong Markov
property.  It will be applied in the proof of Theorem \ref{thm4.19.3}.
\begin{cor}
\label{cor4.18.1}
For every subcritical JSQ network,
\begin{equation}
\label{eq4.18.2}
E_x [\|X(t)\|] - \|x\| \le \bbb t \qquad \text{for all t and x},
\end{equation}
where $\bbb = \sum_n \stackrel{\circ}{\mu}_n$.
\end{cor}
\begin{proof}
Denoting by $T_1,T_2,\ldots$ the times at which arrivals occur and applying the strong Markov property, one can 
repeatedly apply Propositions \ref{prop4.10.2} and \ref{prop4.13.1} over the intervals $(0, T_1 \wedge t], 
(T_1 \wedge t, T_2 \wedge t],\ldots$.  Over each such interval, it follows from (\ref{eq4.10.5}) and (\ref{eq4.14.1})
that
\begin{equation}
\label{eq4.19.1}
\begin{split}
& E_x [\|X(T_{i+1} \wedge t)\|] - E_x [\|X(T_{i} \wedge t)\|] \\ 
& \quad \le \left( \sum_n \stackrel{\circ}{\mu}_n \right) E_x [(T_{i+1} \wedge t) - (T_{i} \wedge t)]
\end{split}
\end{equation}
for each $i$, since $\psi_{Z}(y) > 0$ for all $y$.  Summing over $i$ gives
\begin{equation*}
E_x [\|X(t)\|] - \|x\| \le t \sum_n \stackrel{\circ}{\mu}_n,
\end{equation*}
and hence (\ref{eq4.18.2}).
\end{proof}
\section{Proof of Theorem \ref{thm1.13.1}}

In this section, we demonstrate Theorem \ref{thm1.13.1}.  The proof is organized as follows.  We first show it suffices
to demonstrate Theorem \ref{thm4.19.3}, which is a slight variant of Theorem \ref{thm1.13.1}.  
The demonstration of Theorem \ref{thm4.19.3} is then reduced to showing Proposition \ref{prop4.22.4}.  The inequality 
in the proposition is expressed in terms of expected values 
of the norm $\|\cdot\|$ at stopping times $\sigma$ that will be
introduced shortly.  Most of the rest of the section is devoted to showing Proposition \ref{prop4.22.4}, for which Propositions 
\ref{prop4.10.2} and \ref{prop4.13.1} are employed.  At the end of the section, we briefly discuss a simpler 
alternative approach that can be applied to certain service disciplines; we also mention analogs of Theorem \ref{thm1.13.1}
where the state space is modified. We then discuss the stability of the JLLQ rule.

In order to demonstrate Theorem \ref{thm1.13.1}, we need to verify the inequality (\ref{eq1.13.2}).  In Theorem \ref{thm4.19.3},
we will instead demonstrate the variant (\ref{eq4.19.4}).  We recall that
\begin{equation*}
\tau_M = \text{inf} \{t: \|X(t)\| \le M\}.
\end{equation*}
\begin{thm}
\label{thm4.19.3}
For each subcritical queueing network satisfying (\ref{eq1.11.2a}), there exists $M$ so that
\begin{equation}
\label{eq4.19.4}
E_x [\tau_M] \le \ccc \|x\| \qquad \text{for all x},
\end{equation}
where $\|x\|$ is the norm given in (\ref{eq1.11.3}) and 
$\ccc = \left(\epsa\sum_n \stackrel{\circ}{\mu}_n\right)^{-1}$.
\end{thm}
The inequality (\ref{eq1.13.2}) follows quickly from Theorem \ref{thm4.19.3} and Corollary \ref{cor4.18.1}.
By (\ref{eq4.18.2}), 
\begin{equation*}
E_x[\|X(1)\|] \le \|x\| + \bbb \qquad \text{for all } x,
\end{equation*}
where $\bbb = \sum_n \stackrel{\circ}{\mu}_n$.  Restarting the process at time $1$ and applying (\ref{eq4.19.4}) to
$x^{\prime} = X(1)$ implies that 
\begin{equation*}
E_x[\tau_M (1)] \le \ccc (\|x\| + \bbb) + 1, 
\end{equation*}
which implies (\ref{eq1.13.2}) with $\aaa = \ccc \vee (\bbb \ccc +1)$.

The proof of Corollary \ref{cor4.18.1} did not require any conditions on the evolution of $X(\cdot)$.
In order to demonstrate Theorem \ref{thm4.19.3}, we need to consider the behavior of $X(\cdot)$ when
its norm is large.  If $Z_n (\cdot)$ is uniformly large over an interval for some $n$, we will be able
to apply (\ref{eq4.10.5}) of Proposition \ref{prop4.10.2}.  If either $\|X(\cdot)\|_R$ or $\|X(\cdot)\|_A$
is large, we will be able to employ versions of (\ref{eq4.10.3}) and (\ref{eq4.10.4}).  In each of these cases,
we also apply Proposition \ref{prop4.13.1}.  Iteration of these
bounds and application of the strong Markov property as in the proof of the corollary will then imply the
theorem.  

In order to demonstrate (\ref{eq4.19.4}), we need only consider $\|x\| > M$.  On account of (\ref{eq1.11.3}),
$\|x\| >M$ implies that, for given $M_L$, $M_R$ and $M_A$ with
\begin{equation}
\label{eq4.21.1}
M = M_L + M_R + M_A,
\end{equation}  
either (a) $\|x\|_L > M_L$, (b) $\|x\|_L \le M_L$ and $\|x\|_A > M_A$, or 
(c) $\|x\|_L \le M_L$, $\|x\|_A \le M_A$ and
$\|x\|_R > M_R$.  We will analyze these three cases for appropriate $M_L$, $M_A$ and $M_R$.

Denote by $T$ the time of the first arrival in the network.  We introduce the stopping time $\sigma$, where
\begin{equation}
\label{eq4.22.1}
\sigma = \text{inf} \,\{t: \|X(t)\|_L \le M_L \} \wedge T
\end{equation}
for $x$ satisfying (a) and
\begin{equation}
\label{eq4.22.2}
\sigma = \text{inf} \,\{t: \|X(t)\|_A \le M_A \} \wedge T
\end{equation}
for $x$ satisfying (b).  For $x$ satisfying (c), we set
\begin{equation}
\label{eq4.22.3}
\sigma = t_x \wedge T_{x,\Gamma +1}^{\prime}.
\end{equation}
Here, $t_x$ is deterministic and will be defined in (\ref{eq4.33.1}); $\Gamma$ is given in (\ref{eq1.11.2a}).
(As mentioned there, $\Gamma = 0$ for many applications.)  The term $T_{x,i}^{\prime}$ is the time of the
$i^{\text{th}}$ arrival at the queue $n_x$, with $n_x$ being specified just before (\ref{eq4.33.1}).
We also set
\begin{equation*}
\stackrel{\circ}{\mu}\!\!\mbox{}^{\text{ratio}} = \begin{cases}
\max_n \!\stackrel{\circ}{\mu}_n / \min_n \!\stackrel{\circ}{\mu}_n   & \quad \text{for class independent networks}, \\
1 & \quad \text{for station independent networks}, 
\end{cases}
\end{equation*}
and recall $\ellx$, which was introduced in (\ref{eq4.6.3}).

We will show
\begin{prop}
\label{prop4.22.4}
For each subcritical JSQ network satisfying (\ref{eq1.11.2a}) and for $\ellx$ satisfying 
$\ellx \ge 4N \stackrel{\circ}{\mu}\!\!\mbox{}^{\text{ratio}}$, 
there exist $M_L$, $M_R$ and $M_A$, so that
for $x$ satisfying either (a), (b) or (c), and $\sigma$ chosen as in (\ref{eq4.22.1})-(\ref{eq4.22.3}), 
\begin{equation}
\label{eq4.23.1}
E_x[\sigma] \le \ddd (\|x\| - E_x [\|X(\sigma-)\|]),
\end{equation}
where $\ddd$ is as in Theorem \ref{thm4.19.3}.
\end{prop}

Recall that $\|X(\cdot)\|$ has negative jumps at departures.  It therefore follows from Proposition \ref{prop4.13.1} 
and the strong Markov property that
\begin{equation}
\label{eq4.23.2}
E_x[\|X(\sigma)\|] \le E_x [\|X(\sigma-)\|] \qquad \text{for all } x.
\end{equation}
Together with (\ref{eq4.23.1}), this implies
\begin{equation}
\label{eq4.23.3}
E_x[\sigma] \le \ddd (\|x\| - E_x [\|X(\sigma)\|])
\end{equation}
for $x$ chosen as in the proposition.

It is not difficult to demonstrate Theorem \ref{thm4.19.3} by iterating (\ref{eq4.23.3}) and applying the
strong Markov property.
\begin{proof}[Proof of Theorem \ref{thm4.19.3} using (\ref{eq4.23.3})]
Iteration of (\ref{eq4.23.3}), by applying the stopping rule $\sigma$ at each step, induces a sequence of
stopping times
\begin{equation*}
0 < \sigma_1 < \sigma_2 < \ldots,
\end{equation*}
with the sequence stopping at $\sigma_I$ if $\|X(\sigma_I)\| \le M$.  Repeated application of (\ref{eq4.23.3}),
together with the strong Markov property, implies that for each $i \le I$,
\begin{equation}
\label{eq4.24.2}
E_x[\sigma_i] \le \ddd (\|x\| - E_x [\|X(\sigma_i)\|]) \le \ddd\|x\|
\end{equation}
for all $x$.  (The sum of the bounds obtained from the right side of (\ref{eq4.23.3}) forms a telescoping
series.)  On the other hand, over every finite time interval, there are only a finite number of arrivals
and, in between arrivals, only the norm $\|\cdot\|_A$ can increase; hence only a finite number of stopping
times can occur over a finite interval.  It therefore follows from (\ref{eq4.24.2}) and $\tau_M \le \sigma_I$
that $\sigma_I < \infty$ almost surely, with
\begin{equation}
\label{eq4.24.3}
E_x[\tau_M] \le E_x[\sigma_I] \le \ddd \|x\| \qquad \text{for all } x.
\end{equation} 
This implies (\ref{eq4.19.4}) of Theorem \ref{thm4.19.3}.
\end{proof}

Most of the rest of this section is devoted to demonstrating Proposition \ref{prop4.22.4}.  To do so, we consider
separately the cases (a), (b) and (c) that are given after (\ref{eq4.21.1}).  Cases (a) and (b) will be easy to show;
case (c) requires more effort.  We employ the notation
\begin{equation}
\label{eq4.25.1}
\stackrel{\circ}{m}\!\!\mbox{}^{\text{max}} = \text{max}_{k,A} \stackrel{\circ}{m}_{k,A}.
\end{equation} 
For later use, we also set 
$\stackrel{\circ}{m}\!\!\mbox{}^{\text{min}} = \text{min}_{k,A} \stackrel{\circ}{m}_{k,A}$ and
$\,\stackrel{\circ}{\mu}\!\!\mbox{}^{\text{min}} = \text{min}_{n} \stackrel{\circ}{\mu}_{n}$.
\begin {proof}[Proof of Case (\textnormal{a}) of Proposition \ref{prop4.22.4}]
Under the condition (a), for each $t \in [0, \sigma)$, there exists an $n(t)$ so that
$\|X(t)\|_{L,n(t)} > M_L / N$, and hence by the definition of $\|\cdot\|_{L,n}$ and by (\ref{eq4.7.1}),
\begin{equation*}
Z_{n(t)} (t) > M_L \big/ \left(\stackrel{\circ}{m}\!\!\mbox{}^{\text{max}} \ellv N \psi_Z (Z_{n(t)}(t))\right)
\ge M_L \big/\left(\stackrel{\circ}{m}\!\!\mbox{}^ \text{max} \ellx \ellv N\right).
\end{equation*} 
Setting
\begin{equation}
\label{eq4.25.2}
M_L = \stackrel{\circ}{m}\!\!\mbox{}^{\text{max}} \ellx \ellv \ellw  N,
\end{equation}
it follows that $Z_{n(t)} (t) > \ellw$, and hence $\psi_Z (Z_{n(t)} (t)) = \ellx$. Since 
$\ellx \ge 4N\stackrel{\circ}{\mu}\!\!\mbox{}^{\text{ratio}}$, it 
follows from (\ref{eq4.10.5}) of Proposition \ref{prop4.10.2} that, almost everywhere on $(0, \sigma)$, 
\begin{equation*}
\begin{split}
\|X(t)\|^{\prime} & \le -(\epsa /2) \sum_n \stackrel{\circ}{\mu}_n (\psi_Z (Z_n (t)) - 2) \\
& \le 
- \epsa \sum_n \stackrel{\circ}{\mu}_n.
\end{split}
\end{equation*}
This implies case (a) of (\ref{eq4.23.1}).
\end{proof} 

We next demonstrate case (b).

\begin {proof}[Proof of Case (\textnormal{b}) of Proposition \ref{prop4.22.4}]
Under the condition (b), for each $t \in [0,\sigma)$, there exists a $k(t)$ so that $\|X(t)\|_{A,k(t)} > M_A /K$,
and hence by the definition of $\|\cdot\|_{A,k}$ and by (\ref{eq4.7.1}),
\begin{equation}
\label{eq4.26.1}
\psi_A (S_{k(t)} (t)) > M_A {\Big/} \left(2\!\stackrel{\circ}{m}\!\!\mbox{}^{\text{max}}\ellx K\right).
\end{equation} 
Choose $y_1 > M_1$ large enough so that
\begin{equation}
\label{eq4.26.2}
\psi_{A}^{\prime} (y_1) \ge 2 \left(\sum_n \stackrel{\circ}{\mu}_n\right){\Big /} 
\min_{k,A}(\alpha_k \stackrel{\circ}{m}_{k,A})
\end{equation}
and $\psi_{A}^{\prime} (y_2) \ge \psi_{A}^{\prime} (y_1)$ for $y_2 > y_1$;
this is possible since $\psi_{A}^{\prime} (y) \nearrow \infty$ as $y \nearrow \infty$.  Setting
\begin{equation}
\label{eq4.26.3}
M_A = 2\stackrel{\circ}{m}\!\!\mbox{}^{\text{max}}\ellx K \psi_A (y_1),
\end{equation}
it follows, from (\ref{eq4.26.1}), that $\psi_A (S_{k(t)} (t)) > \psi_A (y_1)$ for each $t$, and hence that
$S_{k(t)} (t) > y_1$ and $\psi_{A}^{\prime} (S_{k(t)} (t)) \ge \psi_{A}^{\prime} (y_1)$.  For $M_A$ as in (\ref{eq4.26.3}),
differentiation of $\|X(\cdot)\|_{A,k}$ using (\ref{eq4.8.1}) and the lower bound $\epsa$ for 
$\psi_{Z}(y)$ therefore imply that, at $k = k(t)$ with $t\in [0,\sigma)$,
\begin{equation}
\label{eq4.27.1}
\|X(t)\|_{A,k}^{\prime} \le -\epsa \sum_{A} \sum_{n} p_{k,A} q_{k,A,n} 
\alpha_k \stackrel{\circ}{m}_{k,A} \psi_{A}^{\prime} (y_1),
\end{equation}
which by (\ref{eq4.26.2}) is at most $-2 \epsa \sum_n \stackrel{\circ}{\mu}_n$.

We apply (\ref{eq4.10.4}) of Proposition \ref{prop4.10.2}, with $\mathcal{K}$ equal to the complement of $\{k(t)\}$.  Adding
(\ref{eq4.10.3}) to (\ref{eq4.10.4}), one obtains the analog of (\ref{eq4.10.5}), but with the additional term
inherited from (\ref{eq4.27.1}), namely, almost everywhere on $[0,\sigma)$,
\begin{equation}
\label{eq4.28.1}
\|X(t)\|^{\prime} \le -(\epsa / 2) \sum_n \stackrel{\circ}{\mu}_{n}\!(2 + \psi_Z (Z_n (t))) \le 
-\epsa \sum_n \stackrel{\circ}{\mu}_n.
\end{equation}
Case (b) of (\ref{eq4.23.1}) follows. 
%
\end{proof}
We now begin the argument for case (c) of Proposition \ref{prop4.22.4}.  This is the only part that requires the
condition (\ref{eq1.11.2a}).  It requires more work than the other two cases since, when $\|x\|_R$ is large,
we need sufficient service of some job $(n,i)$, with large $w_{n,i}$, to ensure a rapid decrease of $\|\cdot\|$.
Since any service discipline is allowed, such service need not occur at all, or even most, times.  We will instead show 
that for $\|x\|_L \le M_L$, with $M_L$ not too large, and $\|x\|_R > M_R$, with $M_R$ enormous, there is a small
(but, nevertheless, large enough) probability that a job $(n,i)$ with enormous $w_{n,i}$ receives sufficient
service so that the corresponding derivative $\psi_{W}^{\prime} (w_{n,i})$ induces an enormous decrease over 
$(0, \sigma)$ of $\|\cdot\|_R$, and hence of $\|\cdot\|$.  In particular, sufficient service of such a job $(n,i)$ must
occur if $(0, \sigma - m^{\text max}]$ is sufficiently long to allow complete service of all jobs at $n$ with smaller
weighted residual service times.  (Recall that $m^{\text max} = \max_{j}m_{j}$.) 

We will identify the job $(n,i)$, with ``large $w_{n,i}$", that was referred to in the last paragraph in
terms of a rapidly increasing sequence $w(1),w(2),\ldots$.  This sequence will also be used to define $t_x$,
which was used in the definition of $\sigma$ in (\ref{eq4.22.3}). We construct $w(i)$ and corresponding 
sequences $p(0), p(1), \ldots$ and $t(0), t(1), \ldots$ inductively.

Set 
\begin{equation}
\label{eq4.30.0}
p(i) = h(t(i) + 1), 
\end{equation}
where $h(\cdot)$ is given in (\ref{eq1.11.2a}).  We choose $w(i)$ and $t(i)$ so that
\begin{equation}
\label{eq4.30.1}
\psi_{W}^{\prime}(w(i) - 1) = \eee 2^{i+\Gamma+2} (t(i-1) + 1) {\Big /} p(i -1)
\end{equation}
%
%
and
\begin{equation}
\label{eq4.30.2}
t(i) = \sum_{\ell = 1} ^i w(\ell) + 2\Gamma.
\end{equation}
Here, $\eee = \sum_n \stackrel{\circ}{\mu}_n \! {\big /}
\!\stackrel{\circ}{\mu}\!\!\mbox{}^{\text{min}}$, $\psi_W (\cdot)$  and $\Gamma$ are as in (\ref{eq4.6.1}) and
(\ref{eq1.11.2a}), respectively.  (Note that, on account of (\ref{eq4.6.2}), the range of $\psi^{\prime}_{W}(\cdot)$
contains $[1,\infty)$, and so (\ref{eq4.30.1}) can always be solved for $w(i)$ and given $i$.)   
Often, $p(i)$ will
decrease very rapidly to $0$, and $w(i)$ and $t(i)$ will increase very rapidly.  The factor $2^i$ in
(\ref{eq4.30.1}) is not required for the proof of Proposition \ref{prop4.22.4}, but will be used in the next
section. Employing $w(i)$ and $M_L$, we set
\begin{equation}
\label{eq4.31.1}
M_R = \stackrel{\circ}{m}\!\!\mbox{}^{\text{max}} N \psi_W \left(\sum_{i=1}^{\lfloor \fff M_L \rfloor} w(i)\right),
\end{equation}
where $\fff = 1/ \left(\epsa \epsb\! \stackrel{\circ}{m}\!\!\mbox{}^{\text{min}}\right)$.  

The sequence $w(i)$ was defined in (\ref{eq4.30.1}) with Lemma \ref{lem4.31.2} in mind.  For the lemma, we relabel
the residual service times $w_{n,i}$, $i=1,\ldots,z_n$, in order  of increasing value at each $n$, 
employing the notation
$w_{n,1}^{\prime},w_{n,2}^{\prime},\ldots,w_{n,z_n}^{\prime}$.
We employ $k^{\prime}_{n,i}$ and $A^{\prime}_{n,i}$ for the corresponding renewal streams and selection sets.
\begin{lem}
\label{lem4.31.2}
Suppose that for given $n$, $\|x\|_{R,n} > M_R / N$ and $z_n \le \fff M_L$.  Then, for some $i=1,\ldots,z_n$, 
$w_{n,i}^{\prime} > w(i)$.  
\end{lem}
\begin{proof}
If $w_{n,i}^{\prime} \le w(i)$ for all $i=1,\ldots,z_n$, then
\begin{equation}
\label{eq4.32.2}
\begin{split}
\|x\|_{R,n} & = \sum_{i=1}^{z_n} \stackrel{\circ}{m}_{k^{\prime}_{n,i},A^{\prime}_{n,i}} \psi_W (w_{n,i}^{\prime}) 
\le \stackrel{\circ}{m}\!\!\mbox{}^{\text{max}} \sum_{i=1}^{\lfloor \fff M_L\rfloor} \psi_W (w(i)) \\ 
& \le \stackrel{\circ}{m}\!\!\mbox{}^{\text{max}} \psi_W \left(\sum_{i=1}^{\lfloor \fff M_L\rfloor} w(i)\right)
< \|x\|_{R,n},
\end{split}
\end{equation}
which is a contradiction.  The first inequality holds since $z_n \le \fff M_L$, the second inequality since 
$\psi_W (\cdot)$ is convex with $\psi_W (0)=0$, and the last inequality since $\|x\|_{R,n} > M_R / N$.
\end{proof}

Suppose that for a given $x\in S$, $\|x\|_{R,n} > M_R / N$ and $z_n \le \fff M_L$ for some queue $n$, which we
denote by $n_x$.  (In case of more than one such $n$, choose one of them.)  As before Lemma \ref{lem4.31.2}, denote by
$w_{n_{x},i}^{\prime}$, $i=1,\ldots,z_{n_x}$, the ordered sequence obtained from 
$w_{n_{x},i}$, $i=1,\ldots,z_{n_x}$.  We set
\begin{equation}
\label{eq4.33.1}
t_x = m^{\text{max}} (t(i_x -1) + 1),
\end{equation}
where $i_x$ is the smallest index $i$ at which $w_{n_x ,i}^{\prime} > w(i)$.  On account of Lemma \ref{lem4.31.2}, such
an index exists.  The time $t_x$ is used to define $\sigma$ in (\ref{eq4.22.3}) in case (c).  

Using the preceding construction, we now complete the proof of Proposition \ref{prop4.22.4}.
\begin{proof}[Proof of Case (\textnormal{c}) of Proposition \ref{prop4.22.4}]
Under case (c), $\|x\|_R > M_R$, and hence $\|x\|_{R,n} > M_R /N$ for some queue $n$.  Moreover, since $\|x\|_L \le M_L$,
\begin{equation}
\label{eq4.33.2}
z_n \le \|x\|_{L,n} \big/ \left(\epsa \epsb \! \stackrel{\circ}{m}\!\!\mbox{}^{\text{min}}\right) 
\le M_L \big/ \left(\epsa \epsb \!\stackrel{\circ}{m}\!\!\mbox{}^{\text{min}}\right) = \fff M_L, 
\end{equation}
with (\ref{eq4.4.1}) and (\ref{eq4.4.2}) being used for the first inequality.  So, the assumptions of Lemma \ref{lem4.31.2}
are satisfied, and $n_x$, $t_x$ and $i_x$ can be defined as immediately following the lemma.

Let $B_x$ denote the event where at most $\Gamma$ potential arrivals occur at $n_x$ by time $t_x$ and their service
times are each at most $2 m^{\text{max}}$.  One can check that
\begin{equation}
\label{eq4.33.3}
P_{x}\left(B_x \right) \ge 2^{-\Gamma} h\left(t_x /m^{\text{max}}\right) = 2^{-\Gamma} p(i_x - 1),
\end{equation}
where $h(\cdot)$ is as in (\ref{eq1.11.2a}).  For the inequality, note that the probability of a service time being at
most $2 m^{\text{max}}$ is at least $1/2$, and that the probability this occurs for all service times for up to 
$\Gamma$ jobs is at least $1/2^{\Gamma}$, after conditioning on knowing the arrival streams and selection sets of the jobs.
The equality follows from (\ref{eq4.30.0}) and (\ref{eq4.33.1}).

On $B_x$, $\sigma = t_x$.  Moreover, by time $t_x$, the total service devoted to jobs at $n_x$, with $i \ge i_x$, is
at least $m^{\text{max}}$, since the total time required to serve all initial jobs $i$, with $i < i_x$, and the at most
$\Gamma$ arriving jobs is at most
\begin{equation}
\label{eq4.34.1}
\sum_{i=1}^{i_x - 1} m_{j^{\prime}_{n_x, i}} w_{n_x, i}^{\prime} + 2 \Gamma m^{\text{max}} 
\le m^{\text{max}} \left(\sum_{i=1}^{i_x - 1} w(i) + 2\Gamma \right) = t_x - m^{\text{max}},  
\end{equation} 
for $j^{\prime}_{n,i}$ defined analogously to $k^{\prime}_{n,i}$ and $A^{\prime}_{n,i}$,
where the equality follows from (\ref{eq4.30.2}) and (\ref{eq4.33.1}).  Since $\psi_{W}^{\prime}(\cdot)$ is nondecreasing,
it follows that, on $B_x$,
\begin{equation}
\label{eq4.34.1a}
\begin{split}
\|x\|_{R,n_x} - \|X(t_x -)\|_{R,n_x} & \ge m^{\text{max}}\stackrel{\circ}{\mu}\!\!\mbox{}^{\text{min}}
\psi_{W}^{\prime} \left(w_{n_x,i_x}^{\prime} - 1 \right) \\
& \ge m^{\text{max}}\stackrel{\circ}{\mu}\!\!\mbox{}^{\text{min}}
\psi_{W}^{\prime} (w(i_x) - 1),
\end{split}
\end{equation}
Consequently, 
\begin{equation}
\label{eq4.34.2}
\begin{split}
E_{x}[\|x\|_{R,n_x} - \|X(\sigma-)\|_{R,n_x}; B_x] & = E_{x}[\|x\|_{R,n_x} - \|X(t_x -)\|_{R,n_x}; B_x] \\
& \ge m^{\text{max}}\stackrel{\circ}{\mu}\!\!\mbox{}^{\text{min}} 2^{-\Gamma} p(i_x -1) \psi_{W}^{\prime} (w(i_x) -1) \\
& = \eee \stackrel{\circ}{\mu}\!\!\mbox{}^{\text{min}} 2^{i_x + 2} t_x
\ge 4\eee \stackrel{\circ}{\mu}\!\!\mbox{}^{\text{min}} t_x,
\end{split}
\end{equation}
with the first inequality following from (\ref{eq4.33.3}) and (\ref{eq4.34.1a}), and the last 
equality following from (\ref{eq4.30.1}).

On the other hand,
\begin{equation}
\label{eq4.35.1}
\begin{split}
& E_x \left[\|x\| - \|X(\sigma-)\|\right] - \frac{1}{2}E_x \left[\|x\|_{R,n_x} -\|X(\sigma-)\|_{R,n_x}; B_{x}\right] \\
& \quad \ge - \left(\sum_n \stackrel{\circ}{\mu}_n \right) E_{x}[\sigma].
\end{split}
\end{equation}
To see this, note that, on intervals between arrivals, $\|X(t)\|_{R,n}$ is decreasing for all $n$, and so
\begin{equation}
\label{eq4.35.2}
\left[\|X(t)\| - \frac{1}{2}\|X(t)\|_{R,n_x} \mathbf{1}\{\omega \in B_x \}\right ]^{\prime}
\le \left[\|X(t)\| - \frac{1}{2}\|X(t)\|_R\right]^{\prime}
\end{equation}
almost everywhere.  But at the time $T_i$ of an arrival in the network,
\begin{equation}
\label{eq4.35.3}
\begin{split}
& \left[\|X(T_i)\| - \frac{1}{2}\|X(T_i)\|_{R,n_x} \mathbf{1}\{\omega \in B_x \}\right ] \\
& \qquad \qquad - \left[\|X(T_{i}-)\| - \frac{1}{2}\|X(T_{i}-)\|_{R,n_x} \mathbf{1}\{\omega \in B_x\}\right] \\ 
& \quad \le \|X(T_i)\| - \|X(T_{i}-)\|
\end{split}
\end{equation}
since $\|X(T_i)\|_{R,n_x} \ge \|X(T_{i}-)\|_{R,n_x}$.

One obtains
\begin{equation}
\label{eq4.36.1}
\left[\|X(t)\| - \frac{1}{2}\|X(t)\|_{R,n_x} \mathbf{1}\{\omega \in B_x \}\right ]^{\prime}
\le \sum_n \stackrel{\circ}{\mu}_n
\end{equation}
by applying (\ref{eq4.10.3}) and (\ref{eq4.10.4}) of Proposition \ref{prop4.10.2} to the right side of (\ref{eq4.35.2})
to get the analog of (\ref{eq4.10.5}) (with $\frac{1}{2}\|X(t)\|_{R}^{\prime}$, rather than $\|X(t)\|_{R}^{\prime}$
being applied in (\ref{eq4.10.3})).  Also, by applying Proposition \ref{prop4.13.1} to the conditional expectation
with respect to $\mathcal{F}_{T_{i}-}$ of the right side of (\ref{eq4.35.3}), one obtains
\begin{equation}
\label{eq4.36.2}
\begin{split}
&E_{x}\left[\|X(T_i)\| - \frac{1}{2}\|X(T_i)\|_{R,n_x} \mathbf{1}\{\omega \in B_x \}\right ] \\
& \qquad \le E_{x}\left[\|X(T_{i}-)\| - \frac{1}{2}\|X(T_{i}-)\|_{R,n_x} \mathbf{1}\{\omega \in B_x \}\right ]. 
\end{split}
\end{equation}
One then obtains (\ref{eq4.35.1}) from (\ref{eq4.36.1}) and (\ref{eq4.36.2}) and the strong Markov property by
arguing as in the proof of Corollary \ref{cor4.18.1}.

It follows immediately from (\ref{eq4.34.2}) and (\ref{eq4.35.1}) that
\begin{equation*}
\|x\| - E_x [\|X(\sigma-)\|] 
\ge 2\eee \stackrel{\circ}{\mu}\!\!\mbox{}^{\text{min}}t_x - \left(\sum_n \stackrel{\circ}{\mu}_n\right) E_{x}[\sigma]. 
\end{equation*}
Since $\eee = (\sum_n \stackrel{\circ}{\mu}_n ) \big/ \stackrel{\circ}{\mu}\!\!\mbox{}^{\text{min}}$
and $t_x \ge \sigma$, this implies
\[ \|x\| - E_x [\|X(\sigma-)\|] \ge \left(\sum_n \stackrel{\circ}{\mu}_n \right) E_{x}[\sigma], \] 
which demonstrates case (c) of (\ref{eq4.23.1}) of Proposition \ref{prop4.22.4}.
\end{proof}
We note that, in some instances, the proof of case (c) of Proposition \ref{prop4.22.4} is not needed, or can
be simplified.  For instance, if all of the service distributions have bounded support then, on account of
(\ref{eq4.33.2}), case (c) will be vacuous if $M_R$ is chosen to be a large enough multiple of $M_L$, and so
the proofs of the  first two parts suffice.

Suppose, instead, that the service discipline is PS.  Then, at any time $t$ and queue $n$ with $Z_n (t) \le M_L$,
$\|X(t)\|_{R,n}$ decreases at rate 
\begin{equation}
\label{eq4.37.1}
\frac{\stackrel{\circ}{\mu}_n}{Z_n (t)} \sum_{i=1}^{Z_n (t)} \psi_{W}^{\prime} (W_{n,i} (t))
\ge \frac{\stackrel{\circ}{\mu}_n}{M_L} \psi_{W}^{\prime} (W_n (t) / M_L).
\end{equation}
If $M_L$ is fixed and $M_R$ is chosen large enough then, for $\|X(t)\|_{R,n} > M_R$, 
the right side of (\ref{eq4.37.1}) is at least 
$4 \sum_{n^{\prime}} \stackrel{\circ}{\mu}_{n^{\prime}}$ for some $n$.
One can then argue analogously to the proof of case (b) that $\|X(t)\|$ decreases at least at
rate $\sum_{n^{\prime}} \stackrel{\circ}{\mu}_{n^{\prime}}$  for $\|X(t)\|_R > M_R$, which will imply (\ref{eq4.23.1})
in case (c) as well.  One can employ a similar argument for the FIFO service discipline, although one first needs
to redefine the state space $S$ so as to suppress information from the state space descriptor on the 
service times of jobs for which service
has not yet begun.  In none of these instances is the assumption (\ref{eq1.11.2a}) used, since lower 
bounds on the service effort devoted to jobs with large residual service times always hold.
(Such lower bounds will not in general hold for LIFO and certain other 
standard service disciplines.)  

When the interarrival times of a queueing network are exponentially distributed, one has the option of removing the
residual interarrival times from the state space descriptor of the process $X(\cdot)$.  The resulting process
$X^{\prime}(\cdot)$, which takes values $x^{\prime}$ in the corresponding space $S^{\prime}$, will also be Markov
under service disciplines that do not employ the omitted information.  
It is easy to see that positive Harris recurrence of $X(\cdot)$ implies the same for $X^{\prime}(\cdot)$; its 
equilibrium is the projection of the equilibrium of $X(\cdot)$.  If one wishes, one can instead demonstrate the 
analog of Theorem \ref{thm1.13.1} directly for $X^{\prime}(\cdot)$ by employing the norm $\|\cdot\|^{\prime}$, with
\begin{equation}
\label{eq4.40.1}
\|x^{\prime}\|^{\prime} = \|x^{\prime}\|_L +  \|x^{\prime}\|_R,
\end{equation}
where $\|\cdot\|_L$ and $\|\cdot\|_R$ are defined as in (\ref{eq4.3.2}) and (\ref{eq4.6.1}).  The proof simplifies 
somewhat, since one can combine Propositions \ref{prop4.10.2} and \ref{prop4.13.1}, and one only requires the bounds
(\ref{eq4.10.3}), (\ref{eq4.16.1}) and (\ref{eq4.16.2}) given there.  Since the state with no jobs is accessible and
petite, positive Harris recurrence follows from this version of Theorem \ref{thm1.13.1}.  One can also show the analog
of Theorem \ref{thm1.17.2}, which is proved in the next section.

As mentioned in Section 2, one can enrich the space $S$ by introducing other coordinates in the state space
descriptor, such as (a) the elapsed time since the last arrival from each stream, or (b) the amount of service
already received by each job.  The arguments we have given for Theorems \ref{thm1.13.1} and \ref{thm1.14.1}  
remain the same in these settings,
since the randomness of the networks -- due to future interarrival and service times, and choices of the selection
set -- is not affected; nor is the norm $\|\cdot\|$. Theorem \ref{thm1.17.2} can also be extended without 
difficulty to these settings, with the corresponding quantities being included in (\ref{eq1.18.1}) and (\ref{eq1.18.2})
if desired.  
It is easy to see from the summary of its proof that Corollary \ref{cor1.15.2} holds under the
enrichment (b); Corollary \ref{cor1.15.2} also holds under (a), although an extra step is required at the end
of the proof to match up the new
coordinates.  Similar comments hold for Corollary \ref{corN8.1}.  In each of the above cases, when the 
interarrival times are exponentially distributed, one can simultaneously include new
coordinates as in (b) while removing the coordinates corresponding to the residual interarrival time from the state
space descriptor, as in the previous paragraph. 

\subsection*{JLLQ networks}
At the beginning of Section 1, we briefly mentioned JLLQ networks, where jobs are assigned
to the queue with the smallest workload, that is, to the queue $n$ where
$\sum_{i=1}^{n} v_{n,i}$ is smallest.  When two or more queues have the smallest workload, one of these queues is
chosen according to some rule.  The traffic
intensities $\rho_i$, $i=1,2$, are defined as in (\ref{eq1.10.1}) and (\ref{eq1.10.2}) for the class and station
independent cases, and the network is said to be subcritical if $\rho_i < 1$.

The stability of the network is not affected by the (non-idling) service discipline, since the evolution of the workload
at a queue is not affected.  Analogs of Theorems \ref{thm1.13.1} and \ref{thm1.14.1}, and Corollary \ref{cor1.15.2}
for the stability of subcritical JLLQ networks hold, 
as do the uniform bounds in Theorem \ref{thm1.17.2} and Corollary \ref{corN8.1},
with the assumptions (\ref{eq1.11.2a}) and (\ref{eqN7.1}) no longer being needed.

Stability of subcritical JLLQ networks is intuitively more obvious than for subcritical JSQ networks, since the system
cannot be ``tricked" into sending jobs into queues with high remaining work.  Stability is also easier to show.  In
\cite{FC}, it was shown by using the same argument involving fluid limits that was used for JSQ networks with the FIFO service
discipline.  One can also show stability, as well as uniform bounds on the marginal distributions, 
with the aid of an appropriate norm that satisfies the analog of Theorem 
\ref{thm1.13.1}.  We do not supply a proof here, but provide some motivation. 

Such a norm $\|\cdot\|$ is given by
\begin{equation}
\label{eq4.42.1}
\|x\|  = \sum_{n=1}^{N}  \psi(g_n),
\end{equation}
with
\begin{equation}
\label{eq4.43.1}
g_n = \sum_{i=1}^{z_n} \stackrel{\circ}{m}_{k_{n,i},A_{n,i}} w_{n,i}^{+} 
+ (1 + \epsa /2) \sum_{k=1}^{K} \sum_{A\subseteq B_N} p_{k,A} q_{k,A,n} \stackrel{\circ}{m}_{k,A} \psi_A (s_k).
\end{equation}
The quantities $\stackrel{\circ}{m}_{k,A}$, $w_{n}^{+}$, $\epsa$, $p_{k,A}$, $q_{k,A,n}$, $s_k$ and $\psi_A (\cdot)$
are the same as those that were employed in Section 4 to define the norm $\|\cdot\|$ there.  The function 
$\psi: \mathbb{R}_{+,0} \rightarrow \mathbb{R}_{+,0}$ is required to be twice continuously differentiable, with
$\psi(0) = 0$, $\psi ^{\prime}(y) > 0$, $\psi^{\prime}(y) \rightarrow \infty$ as $y \rightarrow \infty$, 
$\psi^{\prime\prime} (y) > 0$, $\psi^{\prime\prime} (y) \rightarrow 0$ as $y \rightarrow \infty$, and
\[ \int_{0}^{\infty} \psi(y) F_j (dy) < \infty \qquad \text{for all } j.\]

The first sum in (\ref{eq4.43.1}) plays a role similar to $\|\cdot\|_{L,n}$ for the JSQ rule, and the double sum in
(\ref{eq4.43.1}) plays a role similar to $\|\cdot\|_{A}$.  (When the interarrival times are exponentially distributed,
one can remove them from the state space descriptor, and omit the double sum.)  The function $\psi(\cdot)$ has been chosen
so that for large $\|X(t)\|$, $\|X(t)\|^{\prime} < - \hhh$ for given $\hhh > 0$.  Since 
$\psi^{\prime}(y) \rightarrow \infty$ as $y \rightarrow \infty$, the idling at empty queues does not affect this bound
except in the computation of the constant.
Also, since $\psi^{\prime\prime}(y) \rightarrow 0$ as $y\rightarrow \infty$,  $\psi^{\prime}(y)$ is ``almost 
constant" locally for large
$y$, which can be employed in conjunction with the subcriticality of the 
network to induce a negative drift for $\|X(t)\|$ at large values.  We omit the 
details. 
\section{Uniform bounds on families of JSQ networks and tightness}
In Theorem \ref{thm1.14.1}, we demonstrated positive Harris recurrence for the Markov process $X(\cdot)$ for subcritical JSQ
networks, provided the sets $A_M$ given there are petite.  Such a network therefore has an equilibrium probability measure 
$\mathcal{E}$.  Here, we consider the equilibria $\mathcal{E}^{(a)}$, $a \in \mathcal{A}$, of families $\mathcal{A}$ of
such networks, and demonstrate the uniform bounds on the tails of these equilibria that are given in (\ref{eq1.18.1}) and
(\ref{eq1.18.2}) of Theorem \ref{thm1.17.2}.  The proof of (\ref{eq1.18.1}) occupies most of this section.  After proving the
theorem, we then apply it to show tightness and relative compactness of the marginal distributions 
under additional assumptions on the service disciplines.  

We introduce the notation
\begin{equation}
\label{eq5.3.1}
\begin{split}
\gamma_{z}^{(a)}(M) = & \, \mathcal{E}^{(a)}(Z_1 > M), \qquad \gamma_{\lsm}^{(a)}(M) = \mathcal{E}^{(a)}(\Lbig_{1}^{(a)} > M), \\ 
& \gamma_{w}^{(a)}(M) = \mathcal{E}^{(a)}(W_{1}^{(a)} > M)
\end{split}
\end{equation}
and 
\begin{equation}
\label{eq5.3.2}
\gamma_{s}^{(a)}(M) = \max_{k} \mathcal{E}^{(a)}(S_{k}^{(a)} > M).
\end{equation}
The quantities $z_n$, $\lsm_{n}^{(a)}$, $w_{n}^{(a)}$ and $s_{k}^{(a)}$ were defined in (\ref{eq1.16.5})--(\ref{eq1.16.7});
$Z_n$, $\Lbig_{n}^{(a)}$, $W_{n}^{(a)}$ and $S_{k}^{(a)}$ are the corresponding random variables.  (Recall that
$\Lbig_{n}^{(a)}$ is employed to avoid possible confusion with the constants $L_i$.)  On account of the
symmetry condition (\ref{eq1.17.1}),
the probabilities in (\ref{eq5.3.1}) do not depend on the specific queue $n$ that is chosen; we denote by $z_1$,
$\lsm_{1}^{(a)}$ and 
$w_{1}^{(a)}$ the coordinates of a particular queue.  
To demonstrate (\ref{eq1.18.1}) and (\ref{eq1.18.2}) of Theorem \ref{thm1.17.2}, it suffices to show that
\begin{equation}
\label{eq5.3.3}
\gamma_{z}^{(a)}(M), \gamma_{\lsm}^{(a)}(M), \gamma_{w}^{(a)}(M),\gamma_{s}^{(a)}(M) 
\rightarrow 0 \qquad \text{as } M \rightarrow \infty,  
\end{equation}
uniformly in $a\in \mathcal{A}$.

The argument for $\gamma_{s}^{(a)}(M)$ is elementary.  The arrival flows are renewal processes with densities
$G_{k}^{(a)}(dy)$.  The densities of the corresponding stationary distributions are therefore 
$\alpha_k y G_{k}^{(a)}(dy)$. Substituting this into (\ref{eq1.16.1new}) implies $\gamma_{s}^{(a)}(M) \rightarrow 0$ uniformly
in $a$ as $M \rightarrow \infty$, as desired.

The main idea in showing (\ref{eq5.3.3}) for $\gamma_{z}^{(a)}(M)$ and $\gamma_{w}^{(a)}(M)$ will be to employ the bounds from
Sections 4 and 5, for large $\|\cdot\|_{L,n}^{(a)}$ and $\|\cdot\|_{R,n}^{(a)}$, to show that if at a given queue $n$ either the
queue length $z_n$ or the weighted workload $w_{n}^{(a)}$ is large with nonnegligible probability with respect to a given initial
measure $\nu$, then $E_{\nu} [\|X^{(a)}(t)\|^{(a)}]$ will decrease over an appropriate time interval.  Since the measures 
$\mathcal{E}^{(a)}$ are stationary, this behavior will provide a contradiction unless (\ref{eq5.3.3}) holds for both 
$\gamma_{z}^{(a)}(\cdot)$ and $\gamma_{w}^{(a)}(\cdot)$.  We first demonstrate (\ref{eq5.3.3}) 
for $\gamma_{z}^{(a)}(\cdot)$, which is not
difficult.  The argument for $\gamma_{w}^{(a)}(\cdot)$ is more involved, and relies on the argument for
$\|\cdot\|_{R,n}$ in Section 5.

The limit in (\ref{eq5.3.3}) for $\gamma_{\lsm}^{(a)}(M)$ follows without difficulty from that for $\gamma_{w}^{(a)}(M)$.  The
basic idea is that jobs will not have time to age significantly before the queue empties, if the workload is typically low.  
We present this argument next.

\begin{proof}[Proof of (\ref{eq5.3.3}) for $\gamma_{\lsm}^{(a)}(\cdot)$]
It suffices to show that, for each $\delta>0$,
\begin{equation}
\label{eqN37.1}
\gamma_{\lsm}^{(a)}(M) \ge \delta \quad \Rightarrow \quad M \le M_1(\delta)
\end{equation}  
for some function $M_1(\cdot)$ that does not depend on $a$.  The argument involves partitioning the time interval
$[0,g^{(a)}M]$, $g^{(a)} \stackrel{\text{def}}{=} (m^{\text{max}})^{(a)}$, into subintervals of length 
$g^{(a)}M_2(\delta)$, for appropriate $M_2(\delta)$.  One applies (\ref{eq5.3.3}), 
for $\gamma_{w}^{(a)}(\cdot)$, and (\ref{eqN7.1}) to obtain lower
bounds on the probability that a given queue is empty sometime on such an interval.  Appropriate events corresponding
to the intervals will be disjoint, and adding their probabilities will then imply (\ref{eqN37.1}).

We choose $M_2(\cdot)$ so that 
\begin{equation}
\label{eqN37.2}
\sup_{a\in \mathcal{A}} \gamma_{w}^{(a)}(M_2(\delta) - 2\Gamma) \le \delta /2,
\end{equation}
where $\Gamma$ is as in (\ref{eqN7.1}).  Also, denote by $A_{i}^{(a)}(\delta)$, $i=0,1,2,\ldots$, the events on which
at most $\Gamma$ potential arrivals occur at the queue over 
$(ig^{(a)}M_2(\delta),(i+1)g^{(a)}M_2(\delta)]$, with each such arrival
having service time at most $2g^{(a)}$.  As in (\ref{eq4.33.3}),
\begin{equation}
\label{eqN37.3}
P_x(A_{i}^{(a)}(\delta)) \ge 2^{-\Gamma}h(M_2(\delta))
\end{equation}
for all $x$ and $a$.  If we choose $M$ so that $\gamma_{\lsm}^{(a)}(M) \ge \delta$
for a given $\delta$ and $a$, it follows from (\ref{eqN37.2}), (\ref{eqN37.3}) and the stationarity
of $\mathcal{E}^{(a)}$ that
\begin{equation}
\label{eqN37.4}
\begin{split}
& P_{\mathcal{E}^{(a)}}\left(\Lbig_{1}^{(a)}(iM_2(\delta)) > M,  W_{1}^{(a)}(iM_2(\delta)) \le M_2(\delta) - 2\Gamma;
A_{i}^{(a)}(\delta)\right) \\ 
& \qquad \qquad \qquad \qquad \qquad \qquad \ge \delta 2^{-\Gamma - 1}h(M_2(\delta)).
\end{split}
\end{equation}

If at some time $\tau \in [ig^{(a)}M_2(\delta), g^{(a)}M]$, $Z_{1}^{(a)}(\tau) = 0$ holds then, for each
$t\in [\tau, g^{(a)}M]$, it is immediate that $\Lbig_{1}(t) \le M$.  On the other hand, under the event 
$B_{i}^{(a)}(\delta)$ on the left side of (\ref{eqN37.4}), the total amount of work initially at or entering the queue
over $[ig^{(a)}M_2(\delta), (i+1)g^{(a)}M_2(\delta)]$ is at most $g^{(a)}(M_2(\delta) - 2\Gamma + 2\Gamma) =
g^{(a)}M_2(\delta)$, which implies it will be empty at some time in $[ig^{(a)}M_2(\delta), (i+1)g^{(a)}M_2(\delta)]$.
Hence, on $B_{i}^{(a)}(\delta)$,
\begin{equation}
\label{eqN37.5}
\Lbig_{1}^{(a)}(t) \le M \qquad \text{for } t\in [(i+1)g^{(a)}M_2(\delta), g^{(a)}M]. 
\end{equation}

It follows from (\ref{eqN37.5}) and the definition of $B_{i}^{(a)}(\delta)$ that, for 
$\gamma_{\lsm}^{(a)}(M) \ge \delta$ and $M \ge IM_2(\delta)$, $I\in \mathbb{Z}_{+}$, the events $B_{i}^{(a)}(\delta)$, 
$i=0,\ldots,I-1$, are disjoint.  Taking their union, it follows from (\ref{eqN37.4}) that
\begin{equation*}
P_{\mathcal{E}^{(a)}}\left(\bigcup_{i=0}^{I-1} B_{i}^{(a)}(\delta)\right) \ge I\delta 2^{-\Gamma - 1}h(M_2(\delta)).
\end{equation*}
Consequently, $I\le 2^{\Gamma + 1}/(\delta h(M_2(\delta)))$, and so, under $\gamma_{\lsm}^{(a)}(M) \ge \delta$,
\begin{equation*}
M \le \left(2^{\Gamma + 1}/(\delta h(M_2(\delta))) + 1\right)M_2(\delta) \stackrel{\text{def}}{=} M_1(\delta),
\end{equation*}
which does not depend on $a$.  This implies (\ref{eqN37.1}).
\end{proof}
\subsection*{Demonstration of (\ref{eq5.3.3}) for $\gamma_{z}^{(a)}(\cdot)$} 
The uniformity conditions (\ref{eq1.16.2new})--(\ref{eq1.16.3b}) ensure that the norms $\|\cdot\|^{(a)}$, $a \in 
\mathcal{A}$, can be defined, for a given $\ellx$, by choosing the quantities $\psi_Z (\cdot)$,
$\psi_W (\cdot)$, $\psi_A (\cdot)$, $\ellv$, $\ellw$, $M_1$, $\epsa$, $\epsb$ and $\epsc$ from 
Section 4 so as not to depend on $a$.  The uniform
bounds obtained from these choices will be applied to show (\ref{eq5.3.3}) for both $\gamma_{z}^{(a)}(\cdot)$ and 
$\gamma_{w}^{(a)}(\cdot)$.  The term $\ellx$, $\ellx \ge 4$, should be thought of as a free variable; 
when showing (\ref{eq5.3.3}), we will let $\ellx \rightarrow \infty$.

To see the above claim on the choice of these quantities, first note that there exist 
functions $\psi_W (\cdot)$ and $\psi_A (\cdot)$ satisfying the
regularity and monotonicity properties given after (\ref{eq4.3.3}) and (\ref{eq4.8.1}), with $\psi_A (\cdot)$ satisfying
(\ref{eq4.8.2}) for appropriate $M_1$, so that the following analogs of (\ref{eq4.6.2}) and (\ref{eq4.8.3}) hold:
\begin{equation}
\label{eq5.5.1}
\sup_{a\in \mathcal A} \, \max_j \int_{0}^{\infty} \psi_{W}(\mu{_j}^{(a)} y)F_{j}^{(a)}(dy) \le \epsd
\end{equation}
and
\begin{equation}
\label{eq5.5.2}
\sup_{a\in \mathcal A} \, \max_k \int_{M_1 / \alpha_{k}^{(a)}}^{\infty}\left(\psi_{A}(\alpha_{k}^{(a)} y)
+ \alpha_{k}^{(a)} y \right) G_{k}^{(a)}(dy) \le \epsf,
\end{equation}
where $\epsd = (\epsa)^2 / 40$ as in (\ref{eq4.13.0}); $\epsa$ is determined by 
the left side of (\ref{eq1.16.3}).
Both inequalities follow without difficulty from (\ref{eq1.16.2new}) and (\ref{eq1.16.1new}); the uniform limits
of the tails in (\ref{eq1.16.2new}) and (\ref{eq1.16.1new}) permit the choice of $\psi_W (\cdot)$ and
$\psi_A (\cdot)$ as in (\ref{eq5.5.1}) and (\ref{eq5.5.2}), with $\psi_{W}^{\prime} (y) \nearrow \infty$ and
$\psi_{A}^{\prime} (y) \nearrow \infty$ as $y\nearrow \infty$.

As in (\ref{eq4.13.0}), set $\epsc = (\epsa)^{2}\big/ (40M_1\!\! \stackrel{\circ}{m}\!\!\mbox{}^{\text{ratio}})$, but with 
$\stackrel{\circ}{m}\!\!\mbox{}^{\text{ratio}}$ given by (\ref{eq1.16.3b}).  For 
given $\ellx$, $\ellv$ and $\ellw$ are chosen as in (\ref{eq4.6.3})
and (\ref{eq4.7.1}).  The function $\psi_Z (\cdot)$ is then specified in (\ref{eq4.4.2}), using these choices of
$\epsa$, $\epsc$, $\ellv$ and $\ellw$.  Employing these quantities, one defines $\|\cdot\|_{R}^{(a)}$,
$\|\cdot\|_{L}^{(a)}$ and $\|\cdot\|_{A}^{(a)}$ as in (\ref{eq4.6.1}), (\ref{eq4.3.2}) and (\ref{eq4.8.1}), and
$\|\cdot\|^{(a)}$ as in (\ref{eq1.16.4}).  These norms will depend on $a$ in general.

Rather than deal directly with $\|\cdot\|^{(a)}$, we need to employ a truncated version in order to guarantee that
its expectation with respect to $\mathcal{E}^{(a)}$ is finite.  We denote by $\|\cdot\|^{(a,\ellx)}$ the norm on
$S^{(a)}$ for given $\ellx$ and, by
\begin{equation}
\label{eq5.6.1}
\|x\|^{(a,\ell)} \stackrel{\text{def}}{=} \|x\|^{(a,\ellx)} \wedge \bar{M} \qquad \text{for } x\in S^{(a)},
\end{equation}
its truncation at a given value $\bar{M}$, with 
$\ell \stackrel{\text{def}}{=} (\ellx, \bar{M})$.  (If the expectation of $\|\cdot\|^{(a,\ellx)}$
with respect to $\mathcal{E}^{(a)}$, $a \in \mathcal{A}$, is finite, one can set $\bar{M} = \infty$.)  Since
$\mathcal{E}^{(a)}$, $a \in \mathcal{A}$, is invariant,
\begin{equation}
\label{eq5.6.3}
E_{\mathcal{E}^{(a)}} \left[\|X^{(a)}(t)\|^{(a,\ell)}\right] = E_{\mathcal{E}^{(a)}} \left[\|X^{(a)}(0)\|^{(a,\ell)}\right]  
\qquad \text{for all } t.
\end{equation}

As in Section 5, we denote by $T$ the time of the first arrival in the network.  Here and later on, when the context is
clear, we drop the superscript $(a)$ for quantities such as $X(\cdot)$ and $Z_n (\cdot)$ that are associated with
the networks.
\begin{proof}[Proof of (\ref{eq5.3.3}) for $\gamma_{z}^{(a)}(\cdot)$]
We decompose
\begin{equation}
\label{eq5.7.0}
E_{\mathcal{E}^{(a)}} \left[\|X(0)\|^{(a,\ell)} - \|X(t)\|^{(a,\ell)}\right]   
\end{equation}
into two parts, depending on whether $T>t$ or $T\le t$.  For small $t$, the part with $T>t$ will contribute the main term.
For $\|x\|^{(a,\ellx)} \le \bar{M}$ and given $\ellx$ and $\bar{M}$, one has
\begin{equation}
\label{eq5.7.1}
\begin{split}
& E_{x}\left[\|x\|^{(a,\ell)} - \|X(t)\|^{(a,\ell)};\, T>t\right] \\
& \quad \ge E_{x}\left[\|x\|^{(a,\ellx)} - \|X(t)\|^{(a,\ellx)};\, T>t\right] \\
& \quad \ge E_{x} \left[(\epsa \stackrel{\circ}{\mu}\!\!\mbox{}^{(a)}/2) 
\sum_n \int_{0}^{t} \left(\psi_Z (Z_n (t^{\prime})) -2 \right) dt^{\prime};\, T>t\right] \\
& \quad \ge E_{x} \left[(\epsa \stackrel{\circ}{\mu}\!\!\mbox{}^{(a)}/2) 
\sum_n \left((\ellx /2)\int_{0}^{t} \mathbf{1} \{Z_n (t^{\prime}) > \ellw\}dt^{\prime} \right.\right. \\ 
& \quad \qquad \left.\left.  -2 \int_{0}^{t} \mathbf{1} \{Z_n (t^{\prime}) \le \ellw\}dt^{\prime}\right);\, T>t\right], 
\end{split}
\end{equation}
with the first inequality holding since $\|X(t)\|^{(a,\ell)} \le \|X(t)\|^{(a,\ellx)}$, the second inequality following
from (\ref{eq4.10.5}) of Proposition \ref{prop4.10.2}, and $\psi_Z (\ellw) = \ellx \ge 4$ being used for the third inequality.
Here, $\stackrel{\circ}{\mu}\!\!\mbox{}^{(a)} \stackrel{\text{def}}{=} \,\, \stackrel{\circ}{\mu}_{n}\!\!\!\mbox{}^{(a)}$, 
which is constant in $n$ by assumption.  

On the other hand, for fixed $a$ and $\ellx$, 
\begin{equation*}
\mathcal{E}^{(a)} \left(\|X\|^{(a,\ellx)} > \bar{M} \right) \rightarrow 0 \qquad \text{as } \bar{M} \rightarrow \infty.  
\end{equation*}
Also, $\|X(0)\|^{(a,\ell)} \ge \|X(t)\|^{(a,\ell)}$ for $\|X(0)\|^{(a,\ellx)} > \bar{M}$, and the quantity inside 
$E_x[\cdot]$ for the last term in (\ref{eq5.7.1}) is always at most 
$(\epsa \!\stackrel{\circ}{\mu}\!\!\mbox{}^{(a)}/4)\ellx N^{(a)}t$.  It therefore follows from (\ref{eq5.7.1})
that
\begin{equation}
\label{eq5.8.1}
\begin{split}
& E_{{\mathcal{E}^{(a)}}}\left[\|X(0)\|^{(a,\ell)} - \|X(t)\|^{(a,\ell)};\, T>t\right] \\
& \quad \ge E_{{\mathcal{E}^{(a)}}} \left[(\epsa \! \stackrel{\circ}{\mu}\!\!\mbox{}^{(a)}/2) 
\sum_n \left((\ellx /2)\int_{0}^{t} \mathbf{1} \{Z_n (t^{\prime}) > \ellw\}dt^{\prime} \right.\right. \\ 
& \quad \qquad \left.\left.  -2 \int_{0}^{t} \mathbf{1} \{Z_n (t^{\prime}) \le \ellw\}dt^{\prime}\right);\, T>t\right]
- \ellx N^{(a)} t \delta_{1}^{(a)}(\bar{M}) 
\end{split}
\end{equation}
for appropriate $\delta_{1}^{(a)}(\cdot)$, with $\delta_{1}^{(a)}(\bar{M}) \searrow 0$
as $\bar{M} \nearrow \infty$.  

We now set $M = \ellw$, for $M$ in (\ref{eq5.3.3}).  Since $Z_n (\cdot)$ has right continuous sample paths
and so cannot immediately fall below $\ellw$ if $Z_n (0) > \ellw$, the right side of (\ref{eq5.8.1}) is at least
\begin{equation*}
N^{(a)}t\left(\epsa \stackrel{\circ}{\mu}\!\!\mbox{}^{(a)}\left((\ellx /4) \gamma_{z}^{(a)}(\ellw) -1\right)
- \ellx \delta_{1}^{(a)}(\bar{M}) - \ellx \delta_{2}^{(a)}(t)\right)
\end{equation*}
for appropriate $\delta_{2}^{(a)}(\cdot)$, with $\delta_{2}^{(a)}(t) \searrow 0$ as
$t \searrow 0$.  Hence, for large enough $\bar{M}$ and small enough $t$,
\begin{equation}
\label{eq5.8.2}
\begin{split}
& E_{{\mathcal{E}^{(a)}}}\left[\|X(0)\|^{(a,\ell)} - \|X(t)\|^{(a,\ell)};\, T>t\right] \\
& \qquad \qquad \ge \epsa \! \stackrel{\circ}{\mu}\!\!\mbox{}^{(a)}N^{(a)}t \left((\ellx /4) \gamma_{z}^{(a)}(\ellw) -2\right).
\end{split}
\end{equation}

On the other hand, by employing (\ref{eq4.10.5}) of Proposition \ref{prop4.10.2}, Proposition \ref{prop4.13.1} and the
strong Markov property, it follows that 
\begin{equation}
\label{eq5.9.1}
E_{{\mathcal{E}^{(a)}}}\left[\|X(0)\|^{(a,\ell)} - \|X(t)\|^{(a,\ell)};\, T\le t \right]
\ge -\epsa \! \stackrel{\circ}{\mu}\!\!\mbox{}^{(a)}N^{(a)}t  P_{{\mathcal{E}^{(a)}}}(T\le t).
\end{equation}
One can see this by considering the above difference over $[0,T]$ and $[T,t]$, and, for the second part, arguing as in
the proof of Corollary \ref{cor4.18.1}.

Combining (\ref{eq5.8.2}) and (\ref{eq5.9.1}), and choosing $\bar{M}$ large enough and $t$ small enough, one obtains
\begin{equation}
\label{eq5.9.2}
E_{{\mathcal{E}^{(a)}}}\left[\|X(0)\|^{(a,\ell)} - \|X(t)\|^{(a,\ell)} \right]
\ge \epsa \! \stackrel{\circ}{\mu}\!\!\mbox{}^{(a)}N^{(a)}t  \left((\ellx /4) \gamma_{z}^{(a)}(\ellw) - 3 \right).
\end{equation}
By the invariance of $\mathcal{E}^{(a)}$, the left side of (\ref{eq5.9.2}) is zero, and so
\begin{equation}
\label{eq5.10.1}
\gamma_{z}^{(a)}(\ellw) \le 12 / \ellx.
\end{equation}
The limit (\ref{eq5.3.3}) follows by letting $\ellx \rightarrow \infty$.
\end{proof}
\subsection*{Demonstration of (\ref{eq5.3.3}) for $\gamma_{w}^{(a)}(\cdot)$}
On account of the limiting behavior in (\ref{eq5.3.3}) for $\gamma_{z}^{(a)}(\cdot)$, it suffices to instead show that
\begin{equation}
\label{eq5.11.1}
\tilde{\gamma}_{w}\!\!\mbox{}^{(a)}(M) \rightarrow 0 \qquad \text{as } M\rightarrow \infty
\end{equation}
uniformly in $a$ for
\begin{equation}
\label{eq5.11.2}
\tilde{\gamma}_{w}\!\!\mbox{}^{(a)}(M) \stackrel{\text{def}}{=} \mathcal{E}^{(a)}(W_{1}^{(a)} > M^{\prime}, Z_1 \le M),
\end{equation}
where $M^{\prime}$ is a function of $M$ with $M^{\prime} \rightarrow \infty$ as $M \rightarrow \infty$.  We employ
the same basic framework here as we did in analyzing $\gamma_{z}^{(a)}(\cdot)$, and will apply the truncated norm in
(\ref{eq5.6.1}) to (\ref{eq5.6.3}), which we will show is violated unless the limit in (\ref{eq5.3.3}) holds.  As in
the analysis of $\gamma_{z}^{(a)}(\cdot)$, we set $M=\ellw$ here; we will set $M^{\prime} = \elly$, which will be a
function of $\ellw$ (and hence of $\ellx$) and will be defined in (\ref{eq5.14.1}).  As in the demonstration of (\ref{eq5.3.3})
for $\gamma_{z}^{(a)}(\cdot)$, we continue to drop the superscript $(a)$ when convenient and the context is clear.

The argument relies heavily on the constructions employed for the demonstration of case (c) of Proposition \ref{prop4.22.4}.
We briefly recall the quantities that were employed and whether they depend on $a\in \mathcal{A}$.  The quantities
$\psi_Z (\cdot)$, $\psi_W (\cdot)$, $\psi_A (\cdot)$, $\ellv$, $\ellw$, $M_1$, $\epsa$, $\epsb$ and $\epsc$
were specified in the previous subsection and do not depend on $a$.  As before, the quantity $\ellx$ is allowed to vary,
with $\ellv$ and $\ellw$ being functions of $\ellx$.  We recall the terms $\Gamma$ and $h(\cdot)$ from (\ref{eqN7.1}), and the
terms $p(\cdot)$, $w(\cdot)$ and $t(\cdot)$ from (\ref{eq4.30.0})--(\ref{eq4.30.2}), none of which depends on $a$.  
Instead of employing $\eee = \sum_n \stackrel{\circ}{\mu}_n \big/ \stackrel{\circ}{\mu}\!\!\mbox{}^{\text{min}}$ as in
(\ref{eq4.30.1}), we now set 
\begin{equation}
\label{eq5.14.0}
\eee = \ellw.
\end{equation}
Also, rather than employing the quantities $M_L$ and $M_R$ from (\ref{eq4.25.2}) and (\ref{eq4.31.1}), we use $\ellw$ and
\begin{equation}
\label{eq5.14.1}
\elly \stackrel{\text{def}}{=} \psi_W \left( \sum_{i=1}^{\ellw} w(i) \right).
\end{equation}
Defining $w_{n,i}^{\prime}$ as before Lemma \ref{lem4.31.2}, the conclusions of the lemma continue to hold if 
$w_{n}^{(a)} > \elly$ and $z_n \le \ellw$ replace the conditions on $\|x\|_{R,n}$ and $z_n$ there, 
with the argument in the proof being the same.  
If $w_{n}^{(a)} > \elly$ and $z_n \le \ellw$, for given $x$ and $n$, we define $t_{x,n}$ in terms of $i_{x,n}$ as in (\ref{eq4.33.1}),
but with $i_{x,n}$ being the smallest index $i$ at which $w_{n,i}^{\prime} > w(i)$. 

Both $i_{x,n}$ and $t_{x,n}$ depend on $x$, and we wish to employ a fixed time that does not.  For this, we note
it follows from (\ref{eq5.11.2}) that, for each $a\in \mathcal{A}$, there is a value $i^{(a)}$ at which
\begin{equation}
\label{eq5.15.1}
\mathcal{E}^{(a)}\left(i_{X,1}=i^{(a)},\, W_{1}^{(a)} > \elly,\, Z_1 \le \ellw\right) 
\ge 2^{-i^{(a)}}\tilde{\gamma}_{w}\!\!\mbox{}^{(a)}(\ellw).
\end{equation}
After having chosen $i^{(a)}$, we set
\begin{equation}
\label{eq5.15.2}
t^{(a)} = \left(m^{\text{max}}\right)^{(a)} \left(t(i^{(a)} -1) + 1\right).
\end{equation}

We also denote by $B_{n}^{(a)}$, $n=1,\ldots,N$, the events where at most $\Gamma$ potential arrivals occur at queue $n$ by
time $t^{(a)}$, their service times are each at most $2(m^{\text{max}})^{(a)}$, $i_{X(0),n} = i^{(a)}$, $W_{n}^{(a)}(0) > \elly$ 
and $Z_n(0) \le \ellw$.  It follows from (\ref{eq4.33.3}) and (\ref{eq5.15.1}) that 
\begin{equation}
\label{eq5.15.3}
P_{{\mathcal{E}^{(a)}}}(B_{n}^{(a)}) \ge 2^{-i^{(a)}- \Gamma}p(i^{(a)}-1)\tilde{\gamma}_{w}\!\!\mbox{}^{(a)}(\ellw) 
\qquad \text{for all } n.
\end{equation}
In the following proof, $i^{(a)}$, $t^{(a)}$ and $B_{n}^{(a)}$ will assume the roles played by $i_x$, $t_x$ and $B_x$ in
the proof of case (c) of Proposition \ref{prop4.22.4}.  
\begin {proof}[Proof of (\ref{eq5.3.3}) for $\gamma_{w}^{(a)}(\cdot)$]
We first note that, for given $\bar{M}$,
\begin{equation}
\label{eq5.16.1}
\begin{split}
& E_{{\mathcal{E}^{(a)}}}\left[\|X(0)\|^{(a,\ell)}  - \|X(t^{(a)}\!-)\|^{(a,\ell)}\right] \\
& \quad\ge E_{{\mathcal{E}^{(a)}}}\left[\|X(0)\|^{(a,\ellx)} - \|X(t^{(a)}\!-)\|^{(a,\ellx)};\, \|X(0)\|^{(a,\ellx)} \le \bar{M}\right]
\end{split}
\end{equation}
since $\|X(t^{(a)}\!-)\|^{(a,\ell)} \le \|X(t^{(a)}\!-)\|^{(a,\ellx)}$.  Setting
\begin{equation}
\label{eq5.16.2}
\Delta_{1}^{(a)} = \frac{1}{2}\sum_n E_{{\mathcal{E}^{(a)}}}
\left[\|X(0)\|_{R,n}^{(a,\ellx)} - \|X(t^{(a)}\!-)\|_{R,n}^{(a,\ellx)}; B_{n}^{(a)},\|X(0)\|^{(a,\ellx)} \le \bar{M}\right]
\end{equation}
and
\begin{equation}
\label{eq5.16.3}
\Delta_{2}^{(a)} = E_{{\mathcal{E}^{(a)}}}
\left[\|X(0)\|^{(a,\ellx)} - \|X(t^{(a)}\!-)\|^{(a,\ellx)}; \|X(0)\|^{(a,\ellx)} \le \bar{M}\right] -\Delta_{1}^{(a)},
\end{equation}
we rewrite the right side of (\ref{eq5.16.1}) as
\begin{equation}
\label{eq5.16.4}
E_{{\mathcal{E}^{(a)}}}\left[\|X(0)\|^{(a,\ellx)} - \|X(t^{(a)}\!-)\|^{(a,\ellx)}; \|X(0)\|^{(a,\ellx)} \le \bar{M}\right] 
= \Delta_{1}^{(a)} + \Delta_{2}^{(a)}.
\end{equation}

We first analyze $\Delta_{1}^{(a)}$, which will be the main term.  The same reasoning as in (\ref{eq4.34.1a}) shows that
$\Delta_{1}^{(a)}$ is at least 
\begin{equation*}
\frac{1}{2}(m^{\text{max}})^{(a)} \stackrel{\circ}{\mu}\!\!\mbox{}^{(a)}N^{(a)}
P_{{\mathcal{E}^{(a)}}}\left(B_{1}^{(a)}, \|X(0)\|^{(a,\ellx)} \le \bar{M} \right) \left(\psi_{W}^{\prime}(w(i^{(a)}) -1)\right).
\end{equation*}
This is at least
\begin{equation}
\label{eq5.17.2}
\begin{split}
& \frac{1}{2}(m^{\text{max}})^{(a)} \stackrel{\circ}{\mu}\!\!\mbox{}^{(a)}N^{(a)}
\left(2^{-i^{(a)} - \Gamma} p(i^{(a)} -1) \tilde{\gamma}_{w}\!\!\mbox{}^{(a)}(\ellw) - \delta_{3}^{(a)}(\bar{M})\right) \\
& \qquad \qquad \times \psi_{W}^{\prime}(w(i^{(a)}) -1) \\
& \quad \ge \eee \stackrel{\circ}{\mu}\!\!\mbox{}^{(a)}N^{(a)} \tilde{\gamma}_{w}\!\!\mbox{}^{(a)}(\ellw) t^{(a)} 
\end{split}
\end{equation}
for $ \tilde{\gamma}_{w}\!\!\mbox{}^{(a)}(\ellw) > 0$.
The first line follows from (\ref{eq5.15.3}), with $\delta_{3}^{(a)}(\bar{M}) \searrow 0$ as $\bar{M} \nearrow \infty$,
and the following inequality follows (like the last equality in (\ref{eq4.34.2})) from (\ref{eq4.30.1}), for large
enough $\bar{M}$.  Combining the above inequalities, one obtains, for large enough $\bar{M}$, 
\begin{equation}
\label{eq5.18.1}
\Delta_{1}^{(a)} \ge 
\eee \stackrel{\circ}{\mu}\!\!\mbox{}^{(a)}N^{(a)} \tilde{\gamma}_{w}\!\!\mbox{}^{(a)}(\ellw) t^{(a)}.
\end{equation}
This holds trivially for $ \tilde{\gamma}_{w}\!\!\mbox{}^{(a)}(\ellw) = 0$.

We claim, on the other hand, that
\begin{equation}
\label{eq5.18.2}
\Delta_{2}^{(a)} \ge - \stackrel{\circ}{\mu}\!\!\mbox{}^{(a)}N^{(a)} t^{(a)}.
\end{equation}
The argument for this is essentially the same as that given for (\ref{eq4.35.1}) in the proof of case (c) of
Proposition \ref{prop4.22.4}.  On intervals between arrivals, $\|X(t)\|_{R,n}$ is decreasing for all $n$, and so, for
all $\omega$,
\begin{equation}
\label{eq5.18.3}
\begin{split}
& \left[\|X(t)\|^{(a,\ellx)} - \frac{1}{2}\sum_n \|X(t)\|_{R,n}^{(a,\ellx)} 
\mathbf{1}\{\omega \in B_{n}^{(a)}\}\right]^{\prime} \\
& \qquad \qquad \le \left[\|X(t)\|^{(a,\ellx)} - \frac{1}{2}\|X(t)\|_{R}^{(a,\ellx)}\right]^{\prime}
\end{split}
\end{equation}
holds almost everywhere.  Also, at the time $T_i$ of an arrival in the network, 
\begin{equation}
\label{eq5.19.1}
\begin{split}
& \left[\|X(T_i)\|^{(a,\ellx)} - \frac{1}{2}\sum_n \|X(T_i)\|_{R,n}^{(a,\ellx)} 
\mathbf{1}\{\omega \in B_{n}^{(a)}\}\right] \\
& \qquad \qquad - \left[\|X(T_{i}-)\|^{(a,\ellx)} - \frac{1}{2}\sum_n \|X(T_{i}-)\|_{R,n}^{(a,\ellx)} 
\mathbf{1}\{\omega \in B_{n}^{(a)}\}\right] \\
& \quad \le \|X(T_i)\|^{(a,\ellx)} - \|X(T_{i}-)\|^{(a,\ellx)}
\end{split}
\end{equation}
since $\|X(T_i)\|_{R,n}^{(a,\ellx)} \ge \|X(T_{i}-)\|_{R,n}^{(a,\ellx)}$.  Arguing as after (\ref{eq4.36.1}), one obtains
the upper bound $\stackrel{\circ}{\mu}\!\!\mbox{}^{(a)}N^{(a)}$ for the right side of (\ref{eq5.18.3}), and $0$ for
the expectation, over $\|X(0)\|^{(a,\ellx)} \le \bar{M}$, of the right side of (\ref{eq5.19.1}).  Application of these
bounds, together with the strong Markov property, will then imply (\ref{eq5.18.2}).

Combining the bounds in (\ref{eq5.18.1}) and (\ref{eq5.18.2}), it follows from (\ref{eq5.16.1}) that
\begin{equation}
\label{eq5.20.1}
\begin{split}
& E_{\mathcal{E}^{(a)}}\left[\|X(0)\|^{(a,\ell)}  - \|X(t^{(a)}\!-)\|^{(a,\ell)}\right] \\ 
& \qquad \ge \eee \stackrel{\circ}{\mu}\!\!\mbox{}^{(a)}N^{(a)} \tilde{\gamma}_{w}\!\!\mbox{}^{(a)}(\ellw) t^{(a)} 
 - \stackrel{\circ}{\mu}\!\!\mbox{}^{(a)}N^{(a)} t^{(a)} 
\end{split}
\end{equation}
for large $\bar{M}$.  Since we have set $\eee = \ellw$, this is at least
\begin{equation*}
\left(\ellw \tilde{\gamma}_{w}\!\!\mbox{}^{(a)}(\ellw) - 1\right) \stackrel{\circ}{\mu}\!\!\mbox{}^{(a)}N^{(a)} t^{(a)}.
\end{equation*}
On the other hand, since $\mathcal{E}^{(a)}$ is invariant, the left side of (\ref{eq5.20.1}) equals 0.  Therefore, 
\begin{equation*}
\tilde{\gamma}_{w}\!\!\mbox{}^{(a)}(\ellw) \le 1/\ellw,
\end{equation*} 
and hence $\tilde{\gamma}_{w}\!\!\mbox{}^{(a)}(\ellw) \rightarrow 0$ uniformly in $a$ as 
$\ellw \rightarrow \infty$.  This implies (\ref{eq5.11.1}), and hence (\ref{eq5.3.3}) for $\gamma_{w}^{(a)}(\cdot)$, 
as desired. 
\end{proof}

At the end of Section 5, it was noted that the proof of Theorem \ref{thm1.13.1} simplifies for certain service disciplines,
such as PS and FIFO.  In particular, the proof of case (c) of Proposition \ref{prop4.22.4} can be replaced by a simpler
argument, as was outlined at the end of that section.  The same is true for the demonstration of (\ref{eq5.3.3}) for 
$\gamma_{w}^{(a)}(\cdot)$.  For disciplines such as PS and FIFO, one can give a simpler argument along the lines of
(\ref{eq5.3.3}) for $\gamma_{z}^{(a)}(\cdot)$, by investigating the decrease in the expected value of the corresponding norms
over a small enough time interval $[0,t]$.  The reasoning is similar to that provided at the end of Section 5.  Analogs of
the other observations at the end of Section 5 also hold in the uniform setting of Section 6 as well. 
\subsection*{Tightness and relative compactness of families of networks}
We are interested here in the behavior of the projections of the equilibria 
measures $\mathcal{E}^{(a)}$ and $\bar{\mathcal{E}}^{(a)}$ 
onto $(S^{\prime})^{(a)}$ and $(\bar{S}^{\prime})^{(a)}$, $a \in \mathcal{A}$, 
for families $\mathcal{A}$ of JSQ networks, with $K^{(a)}\equiv K$, for $(S^{\prime})^{(a)}$ 
and $(\bar{S}^{\prime})^{(a)}$ defined as in
Section 2.  The networks will be assumed to satisfy the hypotheses of
Theorem \ref{thm1.17.2}.  The measures $\bar{\mathcal{E}}^{(a)}$, $a \in \mathcal{A}$, are the natural extensions of 
$\mathcal{E}^{(a)}$, $a \in \mathcal{A}$, from $S^{(a)}$ to $\bar{S}^{(a)}$; they are concentrated on $S^{(a)}$.  
The spaces $(S^{\prime})^{(a)}$ and $(\bar{S}^{\prime})^{(a)}$ are assigned a fixed value $N^{\prime}$, 
$N^{\prime} \le N^{(a)}$, for all $a\in \mathcal{A}$ 
(corresponding to the first $N^{\prime}$  queues); they
are identical for different $a$, and hence can be denoted by $S^{\prime}$ and $\bar{S}^{\prime}$.  
Such $S^{\prime}$ and $\bar{S}^{\prime}$, which are projections of each $S^{(a)}$ and $\bar{S}^{(a)}$, will
be referred to as 
\emph{common projections} of the family $\mathcal{A}$.  We denote the corresponding
projected measures by $(\mathcal{E}^{(a)})^{\prime}$ and $(\bar{\mathcal{E}}^{(a)})^{\prime}$.  

As examples of such families of networks, one can think of JSQ networks indexed by $N$, the number of queues for the network, 
with the number of arrival streams $K$ being fixed.  For $N \ge N^{\prime}$, for given
$N^{\prime}$, the networks will share the common projection $S^{\prime}$ obtained by retaining only the first 
$N^{\prime}$ queues.  As $N \rightarrow \infty$, one can 
investigate the limiting behavior of the projections
$(\mathcal{E}^{(N)})^{\prime}$ of the equilibria $\mathcal{E}^{(N)}$ onto $S^{\prime}$. 

To examine the common projections of $\mathcal{E}^{(a)}$ and $\bar{\mathcal{E}}^{(a)}$, for $a\in \mathcal{A}$,
we recall from Section 2 the compact sets $E_M$ in $S$ and the compact sets $\bar{E}_M$ in $\bar{S}$.  
The corresponding sets for the projections $S^{\prime}$ of
$S$ and $\bar{S}^{\prime}$ of $\bar{S}$ will be compact with respect to their respective metrics;
we also denote these sets by $E_M$ and $\bar{E}_M$, respectively.  

Under the assumptions of Theorem \ref{thm1.17.2}, the uniform limits (\ref{eq1.18.1}) and (\ref{eq1.18.2}) hold.
The projections $(\bar{\mathcal{E}}^{(a)})^{\prime}$ onto $\bar{S}^{\prime}$ 
of the measures $\bar{\mathcal{E}}^{(a)}$ are therefore \emph{tight} with respect to the induced metrics.  That is, 
for each $\epsilon >0$, there is a compact set $B \subset \bar{S}^{\prime}$
so that $(\bar{\mathcal{E}}^{(a)})^{\prime}(B) \ge 1 - \epsilon$ for all $a\in \mathcal{A}$.   
The sets $\bar{E}_M$ can be employed to see this: The conditions $z_n \le M$, $\lsm_{n,i}^{(a)} \le M$ and 
$0\le w_{n,i}^{(a)} \le M$ in the definition of $\bar{E}_M$ clearly
suffice for (\ref{eq1.18.1}).  On the other hand, the probability of at least one arrival in the network over 
$(0,1/\alpha_{k}M)$ from a given arrival stream $k$ is at most $1/M$ for a network in equilibrium, 
and so the condition $1/M \le s_{k}^{(a)} \le M$ suffices for (\ref{eq1.18.2}).

For certain service disciplines, the projections $(\mathcal{E}^{(a)})^{\prime}$ onto $S^{\prime}$ 
of the measures $\mathcal{E}^{(a)}$ are also tight.  In addition to employing the reasoning of the previous paragraph, one also
needs to show that 
\begin{equation}
\label{eq5.21.1}
\sup_{a\in \mathcal{A}} \mathcal{E}^{(a)} \left(W_{n,i}^{(a)} < 1/M \text{ for some $i$},\, Z_n > 0\right) 
\rightarrow{0} \qquad \text{as } 
M \rightarrow{\infty},
\end{equation}
for given $n$.  When the projections $(\mathcal{E}^{(a)})^{\prime}$ are tight, it will be more informative to work with them
than with the projections $(\bar{\mathcal{E}}^{(a)})^{\prime}$, since all limits will be concentrated
on $S^{\prime}$.

Tightness of $(\mathcal{E}^{(a)})^{\prime}$ is not difficult to show for PS networks since, 
when the number of jobs at a queue is bounded, 
each job must be served at at least a given rate.  When there are jobs with scaled residual service times 
close to $0$, these jobs will quickly
leave the queue.  On the other hand, the scaled rate at which jobs leave the queue is bounded in equilibrium, which 
therefore gives an upper bound on the expected number of jobs in equilibrium with scaled residual service times close to $0$.

For FIFO networks, (\ref{eq5.21.1}) will not be true in general since, depending on the choice of the network $a$, the
distributions $F_{j}^{(a)}(\cdot)$ might be concentrated arbitrarily close to $0$, and jobs that are not the oldest at their
queue will not be served until the departure of older jobs.  For families $\mathcal{A}$ of networks where the service time
distributions do not depend on $a$, (\ref{eq5.21.1}) is not difficult to check, with the argument being similar to the 
argument for PS networks just mentioned.  On the other hand, equation (\ref{eq5.21.1}) will hold for
general service distributions in the setting obtained by restricting the state space $S$ by suppressing 
information on the service times of jobs for which service has not yet begun; this setting was mentioned at
the end of Section 5.  The argument in this setting proceeds as before.

Families of measures $(\mathcal{E}^{(a)})^{\prime}$, $a \in \mathcal{A}$, need not be tight for arbitrary service 
disciplines.  This is the case for LIFO service disciplines, even when $N^{(a)}=K^{(a)}\equiv 1$.   
For example, consider a family of networks $\mathcal{A} = \{5,6,\ldots,N,\ldots\}$ having service distributions
$F^{(N)}(\cdot)$ with $\mu^{(N)} \equiv 2$ and point masses of size at least $1/3$ at $1$, and having interarrival 
distributions $G^{(N)}(\cdot)$ with $\alpha^{(N)} \equiv 1$ and support on $(0,2]$, and with point masses of size 
at least $1/2$ at $1 - 1/N$.  With probability at least $1/2$,
an arriving job will immediately begin receiving service that continues until its residual service time 
is reduced to $1/N$, at which time
its service is taken over by a new arrival.  Using this, one can show that, for each $N$, the probability under
the corresponding equilibrium measure $\mathcal{E}^{(N)}$ of there being at least one job with residual service time 
at most $1/N$ is at least $1/25$. This contradicts (\ref{eq5.21.1}) and hence contradicts tightness.

If on the other hand, for a family of networks with LIFO service disciplines, the arrival streams are Poisson, it seems 
clear that (\ref{eq5.21.1}) will hold, although the author does not see how to show this.      

A family of probability measures on a metric space  is \emph{relatively compact} if, for each sequence 
drawn from the family, there is a subsequence that converges to some probability measure on the space.  It follows from
Prohorov's Theorem that a tight family of measures is automatically relatively compact (see, e.g., \cite{Bi}).  Combining
this with the preceding discussion of the projected measures $(\bar{\mathcal{E}}^{(a)})^{\prime}$ and
$(\mathcal{E}^{(a)})^{\prime}$ on common projections $\bar{S}^{\prime}$, respectively $S^{\prime}$, for a family
 of JSQ networks, one obtains the following conclusion from Theorem \ref{thm1.17.2}. 
\begin{thm}
\label{thm5.22.1}
Suppose that a family $\mathcal{A}$ of JSQ networks, with $K^{(a)}\equiv K$, satisfies the uniformity conditions
(\ref{eq1.16.2new})--(\ref{eqN7.1}) and (\ref{eq1.17.1}), and that for each network $a\in{\mathcal{A}}$,
$A_{M} = \{x: \|x\|^{(a)} \le M\}$ is petite for each $M>0$ with respect to the norms in (\ref{eq1.16.4}).  Then
the projected measures $(\bar{\mathcal{E}}^{(a)})^{\prime}$ on each common projection $\bar{S}^{\prime}$ of
the family $\mathcal{A}$ are relatively compact.  Moreover, if (\ref{eq5.21.1}) also holds, then the projected
measures $(\mathcal{E}^{(a)})^{\prime}$ on each common projection $S^{\prime}$ are relatively compact.
\end{thm}

By restricting the family $\mathcal{A}$ of networks under consideration, one obtains the following
analog of Corollary \ref{corN8.1}. 
\begin{cor}
\label{corN45.1}
Suppose that each member of a family $\mathcal{A}$ of JSQ networks has a single Poisson arrival stream, that
$F_{j}^{(a)}$ does not depend on $j$ or $a$, that the selection rules are mean field and have uniformly 
bounded support, and that (\ref{eqN8.2}) holds.  Then
the projected measures $(\bar{\mathcal{E}}^{(a)})^{\prime}$ on each common projection $\bar{S}^{\prime}$ of
the family $\mathcal{A}$ are relatively compact.  Moreover, if (\ref{eq5.21.1}) also holds, then the projected
measures $(\mathcal{E}^{(a)})^{\prime}$ on each common projection $S^{\prime}$ are relatively compact.
\end{cor}

\section{A family of JSQ networks with large workload}
In Section 6, we demonstrated Theorem \ref{thm1.17.2}.  The limit (\ref{eq1.18.1}) there states that 
in equilibrium, at each queue, 
the tails for the distribution on the number of jobs, their weighted ages and the weighted workload 
can be bounded uniformly for general families of networks.  This bound does not depend on the service discipline.
If one examines the proof of the theorem, one sees that the bounds that are obtained for the weighted workload 
are actually extremely weak.  The term $M^{\prime}$ in the definition of $\tilde{\gamma}_{w}^{(a)}(\cdot)$
in (\ref{eq5.11.2}) is given by $M^{\prime} = \elly$, with $\elly$ being defined in (\ref{eq5.14.1}) in terms of
the sequence $w(1),w(2),\ldots$ and $\ellw$.  As noted after (\ref{eq4.31.1}), $w(i)$ will often increase very rapidly.  In 
fact, one can check that, for Poisson arrival streams and service distributions having any given number of moments, 
$w(i)$ can grow like 
\begin{equation}
\label{eq7.1.1}
w(i+1) \ge e^{bw(i)},
\end{equation}
where $b>0$ depends on the number of moments.  The growth of $M^{\prime}$ in terms of $M$, with $M = \ellw$, will therefore be 
far too rapid to infer anything useful quantitatively about the tail of the distribution of the weighted 
workload in equilibrium for a member of the family of networks.

Rather than providing quantitative information, the purpose of the uniform bounds in (\ref{eq1.18.1}) of 
Theorem \ref{thm1.17.2} was to establish
tightness of the distributions on the number of jobs, weighted ages and weighted workload  
at a queue.  This was discussed in Sections 1 and 6.  One can, however, ask whether the rapid growth exhibited in 
(\ref{eq7.1.1}), and hence by $\elly$, is an artifact of the proof or whether similar bad behavior is in fact possible for 
the workloads of subcritical JSQ queueing networks.  This point is of course relevant in deciding whether it is always
advantageous to employ the JSQ rule for assigning arriving jobs, as opposed to, say, randomly choosing the 
queue, e.g., letting $D = 1$, in the setting of Section 1.

In this section, we present a family of networks for which the weighted workload exhibits bad behavior in 
an extreme manner that is of the order as that suggested by (\ref{eq7.1.1}).  The structure of the networks in the 
family is elementary in most aspects.  Each network in the family possesses a single Poisson arrival 
stream and two queues to which jobs are directed; the state space $S$ is therefore the same for each network. The 
selection set always consists of both queues, and the service
distributions are both class and station independent.  When the two queues have equal numbers of jobs, an arriving job
will choose each of them with probability $\frac{1}{2}$.  The rates of the
Poisson arrivals will depend on the network as will the service distributions, but the traffic intensity for the 
different networks will be uniformly bounded away from one.  The service discipline will be the
same for all networks, but the discipline is concocted so as to produce inefficient service.  In particular,
service will be \emph{nonlocal} in the sense that the choice of which jobs to serve will depend on the entire state of
the network.  

As we will see, this discipline will cause certain jobs with large service times to be served
very slowly, with service being concentrated on the jobs with shorter service times at that queue, which therefore
quickly leave the queue.  The other queue will, during such times, serve its
jobs with longer service times first, which causes it to accumulate jobs with short service times.  The presence of these
unserved jobs induces arriving jobs to be directed to the first queue.  Since these arriving jobs at the first queue
will be served before
a job with large service time is served, this slows down the rate at which such a job is served.  Before
such a job with large service time receives all of its service, with high probability, a job with
substantially greater service time will enter the queue, which causes the workload at the queue to grow.
Iterating this behavior produces the growth in workload in which we are interested.
This behavior occurs with few jobs at the queue in relation to its workload.          
  
The remainder of the section is organized as follows.  We first complete the description of 
the family of JSQ networks whose description was begun in the next to last paragraph.  
We then state Theorem \ref{prop7.3.1}, which describes the behavior of the weighted workload, in equilibrium, of 
these networks.  The proof of the theorem consists of three parts and includes an induction argument on the
size of a quantity related to the weighted workload.    


The family of queues is indexed by $\epsilon$, with $\epsilon > 0$.  Rather than directly give the service
distributions and Poisson intensity of arrivals, we specify them as follows.  The service distribution 
$F^{(\epsilon)}(\cdot)$ of jobs at each network is assumed to be discrete, having point masses at 
$h(0), h(1),h(2),\ldots$, with $h(0) =\delta \stackrel{\text{def}}{=} \gamma_0 \epsilon$, 
for $\gamma_0 \in (0, \frac{1}{200}]$, $h(1) = 1$, $h(2)\in 2\mathbb{Z}_+$ with $h(2) \ge c_1$, where $c_1\ge 100$  
will be specified in the proof of Proposition \ref{lem7.3.3} and does not depend on 
$\epsilon$ or $\gamma_0$, $h(3) = (h(2))^3$, and
\begin{equation}
\label{eq7.1.2}
h(i+1) = e^{\sqrt{h(i)}} \qquad \text{for } i = 3,4,\ldots;  
\end{equation}
we will restrict the family to indices $\epsilon$ with $\epsilon \le 1/(h(3))^5$.  (For convenience, we set
$h(-1) = 0$.)  We will refer to jobs with (initial) service time $\delta$ as \emph{quick}, 
those with service time $1$ as \emph{moderate}, and those with service time $h(i)$, $i\ge2$, as \emph{large}.
Jobs with these service times are assumed to arrive in the network at rate 
$2/\epsilon$ for $h(0)$, $2(1-\eta)$ for $h(1)$, where $\eta \in (0,\frac{1}{100}]$, 
and at rate $2(h(i))^{-i}$ for $h(2),h(3),\ldots$.

One can check that the traffic intensity $\rho$ is given by
\begin{equation}
\label{eq7.2.1}
\rho = \gamma_0 + 1 - \eta + \sum_{i=2}^{\infty} (h(i))^{1 - i}.
\end{equation}
We require that
\begin{equation}
\label{eq7.2.2}
\rho \le 1 - \eta/2,
\end{equation}
which is easy to do, because of the freedom in our choices for $\delta$ and $h(2)$.  Note that the arriving
moderate jobs require most of the work.  Because of the growth of the exponents in the arrival rates
for $h(i)$ as $i \rightarrow \infty$, the distribution functions $F^{(\epsilon)}(\cdot)$ have all moments.  Note that as 
$\epsilon \searrow 0$, the mean of the service time goes to $0$.  This, together with the fixed arrival rates for
moderate and large jobs, implies that the uniformity condition (\ref{eq1.16.2new}) in Theorem \ref{thm1.17.2} is not satisfied.

We still need to specify the service discipline; as mentioned above, it depends on the entire state of the network.
We first introduce some terminology.  
At each time, one of the queues will be the \emph{designated queue}, with the other being the \emph{other queue}.
Provided the network is not empty, the designated queue will not be empty, in which case one of its jobs 
will be the \emph{designated job}.  We denote by
$Y^{(\epsilon)}(t)$ the residual service time at time $t$ of the designated job; when both queues are empty, set
$Y^{(\epsilon)}(t) = 0$. 
 
At $t=0$, if the network is not empty, we choose one of the nonempty queues as the designated queue and a job with
the largest residual service time for that queue as the designated job.  
This queue will remain the designated queue until it is empty and the other queue is not.

As time increases, a job becomes
the designated job upon its arrival at the designated queue, if its service time is larger than or equal to the residual
service time of the current designated job (and automatically becomes the designated job if the queue is empty).  
Only new jobs can become designated jobs, with the exception being when service of a designated job is completed 
(that is, the job leaves the network).  
It will follow from the service discipline given in the next paragraph that the queue must then be empty; the other
queue (if not empty) then becomes the designated queue, with
the job with the largest residual service time at the queue becoming the 
designated job.  In order to indicate the designated queue and job, an 
additional coordinate needs to be added to the state space $S$ (which we continue to denote by $S$).

The allocation of service is specified as follows.  At the designated queue, the designated job will only be served
when there are no other jobs there.  When a job arrives or departs from the network, among the nondesignated jobs, the 
job with the shortest residual service time at the 
designated queue receives all of the service at the queue; this service continues until the next arrival or departure
from the network.  At the other queue, upon an arrival or departure, the job with the longest residual service 
time receives all of the service, with the job continuing to receive this service until the next arrival or 
departure.  When the designated queue and other queue switch, the service rules also switch.       
One can check that the designated job has the largest residual service time among jobs at its queue, although
not necessarily among jobs in the entire network.

By (\ref{eq7.2.2}), each member of the family of JSQ networks defined above is subcritical.  Moreover, since the
arrival stream is Poisson, the conditions (\ref{eq1.14.2}) and (\ref{eq1.15.1})  are satisfied.  Consequently, by Corollary
\ref{cor1.15.2}, the Markov process $X^{(\epsilon)}(\cdot)$ underlying each network is positive Harris recurrent, and
so has an equilibrium measure $\mathcal{E}^{(\epsilon)}$.  The following result gives a lower bound on the 
distribution of the sum $W_{1}^{(\epsilon)} + W_{2}^{(\epsilon)}$ of the weighted workloads of 
$\mathcal{E}^{(\epsilon)}$.  (As shown in the proof, the same bounds hold for the sum of the unweighted
workloads.)  Here, we set $\gamma_1 = \frac{1}{2000}$.   
\begin{thm}
\label{prop7.3.1}
For the family of networks defined above, 
\begin{equation}
\label{eq7.3.2}
\mathcal{E}^{(\epsilon)}\left(W_{1}^{(\epsilon)} + W_{2}^{(\epsilon)} \le h(\lfloor\gamma_1/\epsilon\rfloor)\right) 
\le 1/h(\lfloor\gamma_1/\epsilon\rfloor). 
\end{equation}
\end{thm}
On account of the recursion for $h(i)$ given by (\ref{eq7.1.2}), $h(\lfloor\gamma_1/\epsilon\rfloor)$ will be 
enormous when $1/\epsilon$ is a moderate multiple of $2000$, and in particular grows much more rapidly 
than $1/\epsilon$; presumably, this rapid growth also holds for much smaller values than $2000$.  The
square root in (\ref{eq7.1.2}) is chosen for convenience; it can be replaced by any power strictly less than $1$, if
$h(3)$ is replaced by a correspondingly higher power of $h(2)$.  Note that the rate of growth given here is 
somewhat slower than that of $w(i)$ in (\ref{eq7.1.1}).    

It follows from (\ref{eq7.3.2}) that the weighted workload under $\mathcal{E}^{(\epsilon)}$ is typically at least of 
order $h(\lfloor\gamma_1/\epsilon\rfloor)$.  This contrasts with the mean weighted workload for the equilibria of the 
networks obtained by setting $D=1$, which is of order $1/\epsilon$.  (The mean workload for the equilibria 
of those networks is bounded over all $\epsilon$.)

Theorem \ref{prop7.3.1} holds because, when a designated queue has designated job with residual service
time $h(i)$, a job with service time $h(i+1)$ is more likely to arrive at the queue before the residual service time of the 
designated job decreases to $h(i-1)$, provided that $i \le \gamma_1 /\epsilon$.  
One can therefore compare $X^{(\epsilon)}(\cdot)$ with a discrete time birth-death process on 
$0,1,\ldots,\lfloor\gamma_1 /\epsilon\rfloor + 1$, with uniform positive drift.  The comparison ceases to be valid when 
the designated queue has too large a multiple of $1/\epsilon$ jobs.  

\subsection*{Proof of Theorem \ref{prop7.3.1}} In order to make the preceding paragraph precise, 
we will employ a recursion argument that requires two propositions.  The first proposition asserts
that the above behavior occurs for $i=2$; the second proposition employs the first proposition and an induction argument
to show this behavior for general $i$.  Since the proof of the first proposition is fairly long and involves a substantial
number of estimates, we give a condensed proof of it. 

For these propositions, we need to employ a modification of the process $X^{(\epsilon)}(\cdot)$.  
We define a new family of Markov processes $X^{(\epsilon, \kappa)}(\cdot)$, $\kappa = 0,1,\ldots,K^{(\epsilon)}$, with 
$K^{(\epsilon)} = \lfloor\gamma_{1}/\epsilon\rfloor$, 
where $X^{(\epsilon, \kappa)}(\cdot)$ evolves in the same
manner as $X^{(\epsilon)}(\cdot)$, but with the following modification of the JSQ rule: 
an arriving job at time $t$ selects the queue with the smaller value of $Z_{d}^{(\epsilon,\kappa)}(t-) + \kappa$ 
and $Z_{o}^{(\epsilon,\kappa)}(t-)$, where $Z_{d}^{(\epsilon,\kappa)}(\cdot)$ and
$Z_{o}^{(\epsilon,\kappa)}(\cdot)$ denote the number of jobs at the designated queue and other queue for the network indexed by
$(\epsilon,\kappa)$.  The term $\kappa$ will serve as a ``handicap" for the designated queue.  As before, we denote by
$Y^{(\epsilon,\kappa)}(\cdot)$ the residual service time of the designated job.  Note that, for $\kappa = 0$, the queue
selection rule is JSQ and $Y^{(\epsilon,0)}(\cdot) = Y^{(\epsilon)}(\cdot)$.  
    
Proposition \ref{lem7.3.3} gives the following uniform bound on $Y^{(\epsilon,\kappa)}(\cdot)$ over all $\kappa\le K^{(\epsilon)}$.
We denote by $y$ the residual service time of the designated job for a state $x \in S$.  
\begin{prop}
\label{lem7.3.3}
Suppose that $x\in S$ is any state for which $y  = h(2)$.   
Let $T$ denote the stopping time at which either $Y^{(\epsilon,\kappa)}(T) = h(1)$ or $Y^{(\epsilon,\kappa)}(T) = h(\ell)$,
for $\ell \ge 3$, first occurs.  
Then, for $\kappa \le K^{(\epsilon)}$, 
\begin{equation}
\label{eq7.3.4}
P_{x}(Y^{(\epsilon,\kappa)}(T) = h(1)) \le 1/h(2).
\end{equation}
%
%
\end{prop}

The bound (\ref{eq7.3.4}) depends strongly on our construction of $X^{(\epsilon,\kappa)}(\cdot)$, where
the designated queue tends to receive more jobs than the other queue.  

Proposition \ref{lem7.4.1} generalizes 
Proposition \ref{lem7.3.3} to $h(i)$, $2\le i\le K^{(\epsilon)} + 1$, by applying 
Proposition \ref{lem7.3.3} together with an induction argument. 
The proposition is stated for $\kappa \le K^{(\epsilon)} - i + 1$,
although only the case $\kappa = 0$ is directly applied in the demonstration of Theorem \ref{prop7.3.1}. 
General $\kappa$ are needed for the comparison in the first paragraph of the proof in the induction argument.
\begin{prop}
\label{lem7.4.1}
Suppose that $x\in S$ is any state for which 
$y = h(i)$ for given $i$, with $2 \le i \le K^{(\epsilon)} + 1$.   
Let $T$ denote the stopping time at which either $Y^{(\epsilon,\kappa)}(T) = h(i-1)$ or $Y^{(\epsilon,\kappa)}(T) = h(\ell)$,
for $\ell \ge i+1$, first occurs.  
Then, for $\kappa \le K^{(\epsilon)} - i + 1$,  
\begin{equation}
\label{eq7.4.2}
P_{x}(Y^{(\epsilon,\kappa)}(T) = h(i-1)) \le 1/h(i).
\end{equation}
%
%
\end{prop}
We first prove Theorem \ref{prop7.3.1}, assuming Proposition \ref{lem7.4.1}.  As mentioned earlier,
the argument compares $Y^{\epsilon}(\cdot)$ to a birth-death process.
\begin{proof}[Proof of Theorem \ref{prop7.3.1} assuming Proposition \ref{lem7.4.1}]  
Theorem \ref{prop7.3.1} will follow from Proposition \ref{lem7.4.1} by using the
reasoning outlined in the paragraph before the beginning of this subsection. To see this, we first consider the 
sequence of stopping times $T(0),T(1),T(2),\ldots$ that are defined inductively as follows, when 
$Y^{(\epsilon)}(0) = h(i_0)$ for some $i_0 \ge 1$.  We set $T_0 = 0$ and, when $i_0 \ge 2$, denote by $T(1)$ the 
first time at which either 
$Y^{(\epsilon)}(T(1)) = h(i_0-1)$ or $Y^{(\epsilon)}(T(1)) = h(\ell)$, for $\ell \ge i_0+1$.  When $i_0 =1$, 
$T(1)$ denotes the first time at which $Y^{(\epsilon)}(T(1)) = h(\ell)$, for $\ell \ge 2$.  
The times $T(2),T(3),\ldots$ are defined analogously.  

For $n = 0,1,2,\ldots$, we define the function $H_{1}^{(\epsilon)}(n)$ by setting 
$H_{1}^{(\epsilon)}(n) = i \wedge (K^{(\epsilon)} + 1)$, 
where $i$ is the value at which $Y^{(\epsilon)}(T(n)) = h(i)$,
and denote by $\mathcal{G}_n = \mathcal{F}_{T(n)}$ the $\sigma$-algebra at this time.  It follows 
from the strong Markov property and Proposition \ref{lem7.4.1}, with $\kappa = 0$ and
$2 \le i \le K^{(\epsilon)} + 1$ that, for $H_{1}^{(\epsilon)}(n) = i$,
\begin{equation}
\label{eq7.5.1}
P(H_{1}^{(\epsilon)}(n+1) = i-1\,|\,\mathcal{G}_n) \le 1/4
\end{equation}
for all $\epsilon$; when $i=1$, the corresponding probability is $0$.

One can compare $H_{1}^{(\epsilon)}(\cdot)$ with the birth-death process $H_{2}^{(\epsilon)}(\cdot)$ on
$i= 1,2,\ldots,K^{(\epsilon)} + 1$, where
\begin{equation*}
\begin{split}
P(H_{2}^{(\epsilon)}(n+1) = i-1\,|\,H_{2}^{(\epsilon)}(n) = i) & = 1/4, \\
P(H_{2}^{(\epsilon)}(n+1) = i+1\,|\,H_{2}^{(\epsilon)}(n) = i) & = 3/4,
\end{split}
\end{equation*}
for $i\neq 1,K^{(\epsilon)}+1$, with the values $i$ instead of $i-1$ and $i$ instead of $i+1$ being taken at $1$ and
$K^{(\epsilon)}+1$, respectively.  In particular, under $H_{1}^{(\epsilon)}(0) = H_{2}^{(\epsilon)}(0)$, one can couple 
$H_{1}^{(\epsilon)}(\cdot)$ and  $H_{2}^{(\epsilon)}(\cdot)$ so that
\begin{equation}
\label{eqN51.1}
H_{1}^{(\epsilon)}(n) \ge  H_{2}^{(\epsilon)}(n) \qquad \text{for all } n.
\end{equation} 
Note that $H_{2}^{(\epsilon)}(\cdot)$ is Markov, but $H_{1}^{(\epsilon)}(\cdot)$ is not.

Let $\mathcal{E}_{2}^{(\epsilon)}$ denote the equilibrium measure on $\{1,2,\ldots,K^{(\epsilon)} + 1\}$ 
of the process $H_{2}^{(\epsilon)}(\cdot)$.  The process is reversible with respect to $\mathcal{E}_{2}^{(\epsilon)}$,
with
\begin{equation}
\label{eqN51.2}
\mathcal{E}_{2}^{(\epsilon)}(B) = 1/3, \qquad  \mathcal{E}_{2}^{(\epsilon)}(\{K^{\epsilon}+1\}) = 2/3,
\end{equation}
where $B = \{1,2,\ldots,K^{(\epsilon)}\}$.

Also, let $L_{N,1}^{(\epsilon)}(A)$, $L_{N,2}^{(\epsilon)}(A)$ denote the number of visits for $H_{1}^{(\epsilon)}(\cdot)$,
$H_{2}^{(\epsilon)}(\cdot)$ to a set $A$, over $1,\ldots,N$, starting from any initial state. By (\ref{eqN51.1})
and (\ref{eqN51.2}),
\begin{equation}
\label{eqN51.3}
\begin{split}
& \limsup_{N\rightarrow\infty} \frac{1}{N} L_{N,1}^{(\epsilon)}(B) 
\le \lim_{N\rightarrow\infty} \frac{1}{N} L_{N,2}^{(\epsilon)}(B) = 1/3, \\
& \liminf_{N\rightarrow\infty} \frac{1}{N} L_{N,1}^{(\epsilon)}(\{K^{(\epsilon)} + 1 \}) 
\ge \lim_{N\rightarrow\infty} \frac{1}{N} L_{N,2}^{(\epsilon)}(\{K^{(\epsilon)} + 1 \}) = 2/3
\end{split}
\end{equation}
both hold almost surely.

When $\mathcal{G}_n$ is given, with $H_{1}^{(\epsilon)}(n) = i \le K^{(\epsilon)}$, then $T(n+1) - T(n)$ is 
stochastically dominated by an exponentially distributed random variable with mean $\frac{1}{2}h(i+1)^{i+1}.$
On the other hand, when $\mathcal{G}_n$ is given, with $H_{1}^{(\epsilon)}(n) = K^{(\epsilon)} + 1$, then 
$T(n+1) - T(n)$ stochastically dominates a random variable with mean
$\frac{1}{8}h(K^{(\epsilon)}+2)^{K^{(\epsilon)}+2}$.  The upper bound is due to the rate at which jobs of size at
least $h(i+1)$ enter the network, as given below (\ref{eq7.1.2}).  The lower bound employs this rate, together with
(\ref{eq7.5.1}), to compare $T(n+1) - T(n)$, restricted to the outcome $H_{1}^{(\epsilon)}(n+1) = K^{(\epsilon)} + 1$, 
to the restriction of an exponentially distributed random variable.    
 
We let 
\begin{equation*}
U_{N}^{(\epsilon)}(A) = \int_{0}^{T(N)}\mathbf{1}\{Y^{(\epsilon)}(t) \in A \}\, dt
\end{equation*} 
denote the occupation time for $Y^{(\epsilon)}(\cdot)$ of the set $A$ up to the random time $T(N)$, starting
from a given initial state.  Recall that if $H_{1}^{(\epsilon)}(n) = i$, then $Y^{(\epsilon)} > h(i-1)$ for
$t \in [T(n),T(n+1))$.  It follows from this, (\ref{eqN51.3}), the two bounds in the previous paragraph, and the strong
law applied to the incremental times $T(n+1)-T(n)$ that 
\begin{equation*}
\begin{split}
& \limsup_{N\rightarrow\infty} \frac{1}{N} U_{N}^{(\epsilon)}([0,h(K^{(\epsilon)})]) 
\le \frac{1}{3}\cdot \frac{1}{2} \cdot h(K^{(\epsilon)} + 1)^{K^{(\epsilon)} + 1}, \\
& \liminf_{N\rightarrow\infty} \frac{1}{N} U_{N}^{(\epsilon)}((h(K^{(\epsilon)}), \infty)) 
\ge \frac{2}{3}\cdot \frac{1}{8} \cdot h(K^{(\epsilon)} + 2)^{K^{(\epsilon)} + 2}
\end{split}
\end{equation*}
hold almost surely. Since $U_{N}^{(\epsilon)}([0,\infty)) = T(N)$, it follows that
\begin{equation}
\label{eqN51.5}
\begin{split}
\limsup_{N\rightarrow\infty} \frac{U_{N}^{(\epsilon)}([0,h(K^{(\epsilon)})])}{T(N)} 
& \le \frac{2h(K^{(\epsilon)} + 1)^{K^{(\epsilon)} + 1}}{h(K^{(\epsilon)} + 2)^{K^{(\epsilon)} + 2}} \\
& \le 1/h(K^{(\epsilon)} + 2) \le 1/h(K^{(\epsilon)}).
\end{split}
\end{equation}
Since the process $X^{(\epsilon)}(\cdot)$ is ergodic, it follows from (\ref{eqN51.5}) that
\begin{equation}
\label{eqN51.6}
\mathcal{E}^{(\epsilon)}(Y^{(\epsilon)} \le h(K^{(\epsilon)})) \le 1/h(K^{(\epsilon)}).
\end{equation}

The residual service time $Y^{(\epsilon)}$ of the designated job is bounded above by the 
(unweighted) workload in the
network, which is at most $2 \epsilon(W_{1}^{(\epsilon)} + W_{2}^{(\epsilon)})$ since the mean service
time is less than $2\epsilon$.  Theorem \ref{prop7.3.1} therefore follows from (\ref{eqN51.6}) and the
definition of $K^{(\epsilon)}$, after dropping the $2\epsilon$ coefficient.
\end{proof}
We now prove Proposition \ref{lem7.3.3}.  The
proof for the proposition is fairly long and requires a number of steps, where one obtains various bounds on the number of
jobs entering each of the two queues.  For these bounds, we will repeatedly employ elementary large deviation bounds
involving sums of independent random variables.  We assume the reader is familiar with these types of estimates and
sketch the corresponding steps.  The first part of the proof, through (\ref{eq7.15.4a}), is given in detail.  

\begin{proof}[Condensed proof of Proposition \ref{lem7.3.3}]
Our basic reasoning will be as follows:  When the other queue is not empty, there will typically be either moderate or
large jobs there.  On account of the service discipline at the queue, the moderate and large jobs will be served before the
quick jobs, which allows quick jobs to accumulate.  This contrasts with service at the designated queue, 
where quick jobs are served first and so do not accumulate.  Therefore, because of the JSQ
rule, when there are jobs in the other queue, most arriving jobs will select the designated
queue.  Since there will be jobs in the other queue most of the time, this implies that the amount of work in the designated
queue will typically have a positive drift.  The probability will therefore be small that 
$Y^{(\epsilon,\kappa)}(t) \le \frac{1}{2}h(2)$ before time $t_0 \stackrel{\text{def}}{=} (h(3))^4 = (h(2))^{12}$.  
But by time $t_0$, there
will, with high probability, be many arriving jobs in the network with service time at least $h(3)$,
and hence at least one such large job will arrive at the designated queue.  So (\ref{eq7.3.4}) 
of the proposition will follow.

We first introduce some notation consisting of the events $A_1(n),A_2(n)$,
$\ldots,A_8(n)$, with $n=1,\ldots,t_0$, and the 
stopping time $T_0$.  
We denote by $A_1 (n)$ the event on which the amount of work present at the designated queue 
at time $n$ is at least
$\frac{1}{2} h(2)$ and at all previous times is at least $\frac{1}{2} h(2) - 1$.  Note that the designated queue remains
the same over $[0,n]$ on $A_1(n)$.  It is easy to see that, under $A_1(n)$,
$Y^{(\epsilon,\kappa)}(t) \ge \frac{1}{2}h(2) - 1$ for $t\le n+1$, that
\begin{equation}
\label{eq7.12.1}
A_1 (n)^c  =  \emptyset \qquad \text{for } n\le \frac{1}{2} h(2),
\end{equation}
and that $Y^{(\epsilon,\kappa)}(T) = h(1)$, with $T\le t_0$, can only occur when $A_1 (t_0)^c$ occurs, 
where $T$ is given in the statement of the proposition.  

All of the other events $A_2 (n),\ldots,A_8 (n)$ are assumed to be subsets of $A_1 (n)$ that, in addition, satisfy the
following properties.  The event $A_2 (n)$ occurs if the designated queue at no time in $[0,n]$ has more than
$1/(1000\epsilon)$ jobs; $A_3 (n)$ is the event for which the other queue at no time in $[0,n]$ has more than
$1/(500\epsilon)$ jobs that were not present at time $0$.  The event $A_4 (n)$ occurs if the subset of $[0,n]$ on which the other
queue is empty has Lebesque measure at most $\frac{5}{13}n$.  The event $A_5 (n)$ occurs if a total of at most $n/(200\epsilon)$
jobs ever arrive at the other queue at times in $(0,n]$ at which there is a job in the queue with residual service time strictly
greater than $\delta$; $A_6 (n)$ is the event on which at most $\frac{1}{100}n$ of these arriving jobs are moderate or large.  Let $T_0$ be
the first time at which the other queue has at most $1/(500\epsilon)$ jobs.  The event $A_7 (n)$ occurs if the subset of $[T_0,n]$
on which the other queue is not empty and has only jobs with residual service times at most $\delta$ has Lebesque measure
at most $\frac{1}{50}n$.  The event $A_8 (n)$ occurs if at least $n$ moderate or large jobs arrive at the designated queue over
$(0,n]$.  Except for (\ref{eq7.13.1}), where $A_8(n)$ is employed, the precise definitions of the events
$A_2(n),\ldots,A_8(n)$ will not be needed until the paragraph after (\ref{eq7.15.4a}).

One can check that
\begin{equation}
\label{eq7.13.1}
A_1 (n)^c \subseteq  A_8 (n-1)^c 
\qquad \text{for } n\ge \frac{1}{2} h(2).
\end{equation}
For this, note that $A_1(n-1)^c \subseteq A_8(n-1)^c$ and that if $\omega \in A_1(n-1) \cap A_1(n)^c$, then the amount
of work present in the designated queue at time $n$ is less than $\frac{1}{2}h(2)$, which implies $\omega \in A_8(n-1)^c$.  

In order to show (\ref{eq7.3.4}), we may assume that, at time $0$, the designated queue has only a 
single job, consisting of a designated job with residual service time $h(2)$.  Otherwise, 
one can wait until all jobs, except for the designated 
job, leave the queue, or a job with service time at least $h(3)$ arrives at the designated queue.  Since the designated job is served 
last, one or the other event must eventually happen.  If the former occurs, reset the time to $0$.

Let $B(n)$ be the subset of $A_1(n)$ on which no job with service time at least $h(3)$ arrives at the designated
queue by time $n$.
Most of the work in showing (\ref{eq7.3.4}) consists of showing the following bounds involving $B(t_0)$
and $A_{\ell} (n)$:
\begin{equation}
\label{eq7.15.1}
P_x(B(t_0)) \le 2 e^{-c_2 h(2)}
\end{equation}
and
\begin{equation}
\label{eq7.15.2}
P_x(A_1 (n)\cap A_{\ell}(n)^c) \le e^{-c_2 h(2)} \qquad \text{for } \ell = 2,\ldots,8,
\end{equation}
for integers $n\in [\frac{1}{2}h(2),t_0]$ and appropriate $c_2 > 0$ not depending on the choice of $c_1$, which is 
given before (\ref{eq7.1.2}). The inequalities (\ref{eq7.15.1}) 
and (\ref{eq7.15.2}), for $\ell =8$, will be applied
in the next paragraph.  The other inequalities in (\ref{eq7.15.2}) are needed to show this inequality.
We assume that $c_1$ is chosen large enough
so that $(c_{1}^{\prime})^{14} \le e^{c_{1}^{\prime} c_2}$ for $c_{1}^{\prime} \ge c_1$.

We now show (\ref{eq7.3.4}) and afterwards motivate the inequalities.  We may assume, by induction on $n$,
\begin{equation}
\label{eq7.15.3}
P_x(A_1 (n-1)^c) \le (n-1) e^{-c_2 h(2)};
\end{equation}
(\ref{eq7.15.3}) trivially holds for $n-1 \le \frac{1}{2}h(2)$ on account of (\ref{eq7.12.1}).  
Because of (\ref{eq7.13.1}) and
(\ref{eq7.15.2}), with $\ell = 8$, it follows from (\ref{eq7.15.3}) that
\begin{equation}
\label{eq7.15.4}
P_x(A_1 (n)^c) \le n e^{-c_2 h(2)}.
\end{equation}
Consequently, (\ref{eq7.15.4}) holds for all $n\le t_0$ and, in particular,
\begin{equation}
\label{eq7.15.4p}
P_x(A_1 (t_0)^c) \le t_0 e^{-c_2 h(2)} = (h(2))^{12} e^{-c_2 h(2)}.
\end{equation}
As pointed out after (\ref{eq7.12.1}), $Y^{(\epsilon,\kappa)}(T) = h(1)$, with $T\le t_0$, can only occur under this event.

We have so far not employed (\ref{eq7.15.1}).  By (\ref{eq7.15.1}),
\[P_x(T > t_0; A_1(t_0)) \le 2 e^{-c_2 h(2)}.\]
Together with (\ref{eq7.15.4p}), this implies that the event in (\ref{eq7.3.4}) occurs with probability at most
\begin{equation}
\label{eq7.15.4a}
2(h(2))^{12} e^{-c_2 h(2)} \le 1/h(2), 
\end{equation}
with the inequality holding because of the assumption on $c_1$ given above and $h(2) \ge c_1$.  
This implies (\ref{eq7.3.4}).

The rest of the proof of (\ref{eq7.3.4}) is spent justifying (\ref{eq7.15.1}) and (\ref{eq7.15.2}).  
We begin by justifying (\ref{eq7.15.2}), which requires most of the work; (\ref{eq7.15.1}) will then
follow quickly. 

The exponential bounds in (\ref{eq7.15.2}) follow from elementary large deviation bounds,
involving sums of i.i.d. random variables, that can be
obtained by using Markov's Inequality or by comparison with a birth-death process.  
We assume here that the reader is familiar with such bounds.  Since the bounds for 
some of the previous indices $\ell$ are used to derive
(\ref{eq7.15.2}) for larger $\ell$, the coefficient $c_{2,\ell}$ in the exponent of the inequality corresponding
to (\ref{eq7.15.2}) for a specific $\ell$ may depend on $\ell$; after computing this bound for each $\ell$, 
one can then set $c_2 = \min_{\ell} c_{2,\ell}$ to obtain (\ref{eq7.15.2}). 

The inequality in (\ref{eq7.15.2}) for $A_2 (n)^c$ follows from elementary large deviations bounds and the 
JSQ rule (with handicap $\kappa$).  To see this, note that there 
is initially only one job at the designated queue and
(a) The number of moderate and large jobs entering 
the network will be of smaller
order of magnitude than $1/(1000\epsilon)$ off of a set of probability $e^{-h(2)}$ 
(since $1/(1000\epsilon) \ge (h(3))^5/1000 \ge h(2)$ and $t_0 = (h(3))^4 \le 10^{-6}(h(3))^5 \le 10^{-6}/\epsilon$). 
(b) Work from quick jobs enters the network at rate
at most $1/100$ (since $\gamma_0 \le 1/200$), which is less than the rate $1$ of service at the designated queue.  
Comparison with the corresponding birth-death process therefore implies that the net number of quick jobs entering 
and leaving the designated queue 
up to any time before $t_0$ is of smaller order than $1/(1000\epsilon)$, off of a set of probability $e^{-h(2)}$
(since $t_0 e^{-1/(1000\epsilon)} \le (h(3))^4 \exp\{-(h(3))^4\} \ll e^{-h(2)}$). 

The inequality
for $A_3 (n)^c$ follows immediately from the inequality for $A_2 (n)^c$, since it is assumed that the
handicap $\kappa$ is at most $K^{\epsilon} \le 1/(2000\epsilon)$ and so 
$1/(1000 \epsilon) + K^{\epsilon} +1 \le 1/(500\epsilon)$.

In order to show the inequality (\ref{eq7.15.2}) for $A_4 (n)^c$, note that when the other queue is empty but the
designated queue is not, the rate at which jobs with service time $1$ arrive at the other queue is 
$2(1-\eta) \ge 9/5$.  
When these jobs are being served, the other queue is not empty.
Under $A_{1}(n)$, the designated queue is never empty on $[0,n]$.  
Noting that $1-\frac{9}{5}\cdot\frac{5}{14} = \frac{5}{14}$, the frequency 
over $[0,n]$, for which the other queue is empty, is therefore typically at most $5/14$.  The 
inequality in (\ref{eq7.15.2}) follows from an elementary large deviations bound.  

We next show (\ref{eq7.15.2}) for $A_5 (n)^c$.  On $A_1 (n)$, the
other queue remains the same over $(0,n]$.  Over this interval, the
set of times having a job at the other queue, with residual service
time strictly greater than $\delta$, is the union of disjoint intervals,
each with length at least $2/3$, except for the first interval (since $h(1) = 1$ and $\delta < 1/3$).  The number of
such intervals is at most $\frac{3}{2}n + 2 \le 2n$.  Note that no quick jobs are served
at the other queue over such intervals.  Also, at most $n$ moderate or large
jobs entering the other queue over $(0,n]$ can complete their service over $(0,n]$.
It follows that, on $A_1 (n) \cap A_3 (n)$, there
are at most $1/(500\epsilon)$ jobs arriving at the other queue over
any such interval, and hence at most $n/(250\epsilon) + n \le n/(200\epsilon)$ jobs arriving
at the other queue over such intervals. The inequality (\ref{eq7.15.2})
for $A_5 (n)^c$ therefore follows from that for $A_3 (n)^c$.

Ordering these at most $n/(200\epsilon)$ jobs in the order in which
they arrive at the other queue, the service time of a job is 
independent of the service time of previous such jobs (although the
total number of arriving jobs will not be independent).  Since the
probability a job is either moderate or large is less than $\epsilon$,
the probability of there being at least $n/100$ moderate and large
jobs among $n/(200\epsilon)$ jobs, $n \ge h(2)/2$, is exponentially
small in $h(2)$.  The inequality (\ref{eq7.15.2}) for $A_6 (n)^c$
therefore follows from that for $A_5 (n)^c$.

For $A_7 (n)^c$, note that there will typically be approximately $2n/\epsilon$ arrivals in the network by
time $n\ge \frac{1}{2} h(2)$ and so, by an elementary large deviations argument, there will be at least
$3n/\epsilon$ arrivals with an exponentially small probability in $h(2)$.  At time 
$T_0$, there are at most $1/(500\epsilon)$ jobs at the other queue, and so there will be a total of at most 
$4n/\epsilon$ jobs at the other queue after
time $T_0$.  The amount of time spent serving a job after it has residual service time $\delta$ is of course
$\delta$, and so the time spent serving all of these jobs is at most $4(\delta / \epsilon)n \le \frac{1}{50}n$,
which demonstrates (\ref{eq7.15.2}) in this case as well.

We still need to show the inequality for $A_8 (n)^c$.  For this, note that by (\ref{eq7.15.2}), with $\ell = 4$
and $\ell = 7$, 
\begin{equation}
\label{eq7.18.1}
P_x\left(A_1 (n) \cap (A_4 (n)^c \cup A_7(n)^c)\right) \le 2e^{-(c_{2,4}\wedge c_{2,7}) h(2)},
\end{equation}
and denote by $H\subseteq [T_0,n]$ the random set where there is at least one job at the other queue with residual service
time strictly greater than $\delta$.  Then (\ref{eq7.18.1}) gives an upper bound on the probability that $A_1 (n)$ 
occurs and the measure of $H$ is less than
\[n - T_0 - \left(\frac{5}{13}n + \frac{1}{50}n\right) \ge \frac{7}{12}n - T_0.\] 
Setting $H^{\prime} = (0,T_0] \cup H$, this implies $|H^{\prime}| \ge \frac{7}{12} n$ off of the exceptional event.  
Since the service time of an arriving job is independent of the service times of previous jobs and $\eta \le 1/100$, 
a large deviations bound
implies that at least $\frac{8}{7}n$ moderate and large jobs enter the network on $H^{\prime}$, off of an event
of exponentially small probability in $h(2)$.  

On the other hand, on $A_6 (n)$, at most $\frac{1}{100}n$ moderate and large jobs arrive at the other queue
during $H$.  Also, on $A_2 (n)$, no jobs can arrive at the other queue before time $T_0$, since the 
handicap $\kappa \le 1/(2000\epsilon)$.  So, on 
$A_2 (n) \cap A_6 (n)$, at most $\frac{1}{100}n$ moderate and large jobs arrive 
at the other queue during $H^{\prime}$.  Together
with the last paragraph and (\ref{eq7.15.2}), for $\ell =2$ and $\ell = 6$, this implies that, off of an event
of exponentially small probability in $h(2)$, when $A_1(n)$ occurs, at least $\frac{8}{7}n - \frac{1}{100}n \ge n$ 
medium and large jobs
arrive at the designated queue over $H^{\prime}$, and hence over $(0,n]$.  This gives the inequality
(\ref{eq7.15.2}) for $A_8 (n)^c$.

%
We now demonstrate (\ref{eq7.15.1}).  Setting $n=t_0$ and $\ell=8$, we first note that, by (\ref{eq7.15.2}),
\begin{equation}
\label{eq7.18.1a}
P_x(A_1(t_0) \cap A_8(t_0)^c) \le e^{-c_2 h(2)}.
\end{equation}
%
%
%
%
On $A_8(t_0) \subseteq A_1(t_0)$, at least $t_0 = (h(3))^4$ medium and large jobs arrive at the designated queue over $(0,t_0]$.  
The probability such a job has service time at least $h(3)$ is greater than $(h(3))^{-3}$.  So, when
$(h(3))^4$ such jobs occur, the probability at least one of them has service time at least $h(3)$ is greater than
$1 - 2e^{-h(3)}$.  The bound (\ref{eq7.15.1}) follows from this and (\ref{eq7.18.1a}).  This completes the
proof of the proposition.
%
%
%
%
\end{proof}
We now prove Proposition \ref{lem7.4.1}.  The argument involves repeated couplings.
\begin{proof}[Proof of Proposition \ref{lem7.4.1}]
We apply induction to demonstrate the proposition.  The case where $i=2$ was demonstrated in Proposition \ref{lem7.3.3},  
so we still need to show
that (\ref{eq7.4.2}) holds for $i+1$ given that it holds for $i$.  The argument for this involves repeatedly coupling
the process $X^{(\epsilon,\kappa)}(\cdot)$, with $Y^{(\epsilon,\kappa)}(0) = h(i+1)$, to one of two processes 
$X^{(\epsilon,\kappa)}(\cdot)$ and $ X^{(\epsilon,\kappa+1)}(\cdot)$, with
$Y^{(\epsilon,\kappa)}(0) = Y^{(\epsilon,\kappa+1)}(0) = h(i)$.  In our couplings, we refer to the  
networks corresponding to  $X^{(\epsilon,\kappa)}(\cdot)$, with $Y^{(\epsilon,\kappa)}(0) = h(i+1)$, 
on the one hand, and to the other two processes, on the other hand,
as the \emph{first} and \emph{second} networks.  As in the proof of Proposition \ref{lem7.3.3}, we may assume that, 
at time $0$, the designated queue of the first network has only a single 
job, consisting of a designated job at $h(i+1)$.   

We compare the first network, with index $(\epsilon, \kappa)$ for given $\kappa\le K^{(\epsilon)}-i$, to 
the second network with the same index and having the exact same number of jobs with the exact same residual
service times and designated job, except that the designated job of the second network has residual service time
$h(i)$ instead of $h(i+1)$.  (Recall that the designated job is not required to have the largest residual
service time in the network.)   One can couple the processes so that their evolution is exactly the same,
with respect to service of individual jobs and arrivals, until either (a) the 
residual service time of the designated job of the second network reaches $h(i-1)$, 
(b) an arrival at the designated queue has service time $h(i)$, or (c) an arrival at the designated queue has service time at least
$h(i+1)$.  In the last two cases, the arriving job becomes the designated job in the second network
and, in the last case, it becomes the designated job in the first network.  

By the induction assumption, 
the event in (a) occurs with probability at most $1/h(i)$.  We consider
such an event a \emph{minor failure}.  At the time $\sigma$ at which it occurs, the designated job of the 
first network will have residual service time 
\[h(i+1)-h(i)+h(i-1) \ge h(i+1) - h(i).\]   
When the event in (c) occurs, we refer to it as a \emph{minor success}.  

At the time $\sigma$ the event in (b) occurs, we change the coupling.  We now couple the first network to the second network
with index $(\epsilon, \kappa+1)$, and having the same 
number of jobs with the same residual service times and designated queue as the first
network, after 
ignoring the designated job of the first network.  The second network therefore has one less job than the first
network at time $\sigma$.  
The job that has arrived at time $\sigma$ has service time $h(i)$
and so will be the designated job for the second network.  It has the greatest residual service time 
for the designated queue in the first
network, except for the designated job, and will be served last except for that job; it will referred to as 
the \emph{associate designated job} of the first network. 

Since the first and second networks differ only by the presence of the large designated job in the former, and since
the handicap of the second network is one greater than that for the first network, the networks can be coupled so
that their evolution is the same until the time $\sigma_1$ at which either the event (a) or the event (c) in the
second to last paragraph occurs.  Under (a), a minor failure occurs whereas, under (c), a minor success occurs.
The minor failure occurs with probability at most $1/h(i)$, at which time the associate designated job of the first
network and the designated job of the second network each have residual service time $h(i-1)$ and the designated job
of the first network still has residual service time at least $h(i+1)-h(i)$.

Combining the preceding two couplings, it follows from the above reasoning that the probability a minor failure occurs before a
minor success is at most $2/h(i)$.  One can repeat this reasoning up to $n(i)$ times, where
\[n(i) \stackrel{\text{def}}{=}\lfloor h(i+1)/h(i)\rfloor -2,\]
so that either (d) $n(i)$ minor failures or (e) a minor success has occurred.  At this time $\sigma_{2}$, the residual
service time of the designated job of the first network is still greater than $h(i)$.  The probability of the event
in (d) is at most $(2/h(i))^{n(i)}$.  If the residual service time of this job decreases to $h(i)$ before (e) occurs,
we say a \emph{failure} occurs. 

When a minor success occurs, the arriving job has service time at least $h(i+1)$ and 
becomes the designated job.  If the service time of the arriving job is at least $h(i+2)$, we
call the minor success a \emph{success} and stop the procedure.  If a minor success that is not a success occurs, one can 
repeat the above comparisons, starting as before with a single job at the designated queue.  
Letting $Q(i)$ denote the number of minor successes occurring before a failure and $B$ denote the event that a
success occurs before a failure, it follows from the previous paragraph
and the definition of $h(\cdot)$ that
\begin{equation}
\label{eq7.7.1}
\begin{split}
& P_x\left(Q(i) \le e^{2\sqrt{h(i+1)}};\, B^c\right) \le e^{2\sqrt{h(i+1)}}\left(\frac{2}{h(i)}\right)^{n(i)} \\ 
& \qquad \qquad \qquad \qquad \le \left(\frac{2e}{h(i)}\right)^{h(i+1)/h(i)} \le \frac{1}{2h(i+1)}. 
\end{split}
\end{equation}
%
%

On the other hand, the probability that a minor success will actually be a success is at least 
\begin{equation}
\label{eq7.7.2}
\frac{1}{2}(h(i))^i / (h(i+1))^{i+1} \ge 1/ (h(i+1))^{i+1} \ge e^{-\sqrt{h(i+1)}}.
\end{equation}
Together with (\ref{eq7.7.1}), (\ref{eq7.7.2}) implies that the probability of a failure occurring before a success
is at most 
\begin{equation*}
\begin{split}
&\left(1-e^{-\sqrt{h(i+1)}}\right)^{e^{2\sqrt{h(i+1)}}} + \frac{1}{2h(i+1)} \\
&\qquad \qquad \qquad \le \text{exp}\left\{-e^{\sqrt{h(i+1)}}\right\} + \frac{1}{2h(i+1)} \le \frac{1}{h(i+1)}.
\end{split}
\end{equation*}
Since success and failure correspond to the two events in  the statement of Proposition
\ref{lem7.4.1}, this implies (\ref{eq7.4.2}) of the proposition.
\end{proof}

\section*{Acknowledgment}  The author thanks an anonymous referee for a careful reading of the
paper.

\end{document}